\documentclass[review]{elsarticle}
\usepackage{srcltx}
\usepackage{eurosym}
\usepackage{mathtools}
\usepackage{amsmath}
\usepackage{amsfonts}
\usepackage{amssymb}
\usepackage{amsthm}
\usepackage{graphicx}
\usepackage{mathrsfs}
\usepackage{xcolor}
\usepackage{exscale}
\usepackage{latexsym}
\newcommand{\redref}[2]{\texorpdfstring{\protect\hyperlink{#1}{\textcolor{black}{(}\textcolor{red}{#2}\textcolor{black}{)}}}{}}
\newcommand{\redlabel}[2]{\hypertarget{#1}{\textcolor{black}{(}\textcolor{red}{#2}\textcolor{black}{)}}}
\usepackage[colorlinks,plainpages=true,pdfpagelabels,hypertexnames=true,colorlinks=true,pdfstartview=FitV,linkcolor=blue,citecolor=red,urlcolor=black]{hyperref}
\PassOptionsToPackage{unicode}{hyperref}
\PassOptionsToPackage{naturalnames}{hyperref}
\usepackage{enumitem}
\usepackage{bookmark}
\usepackage{wasysym}
\usepackage{esint}

\usepackage[ddmmyyyy]{datetime}

\makeatletter
\newcommand{\rmh}[1]{\mathpalette{\raisem@th{#1}}}
\newcommand{\raisem@th}[3]{\hspace*{-1pt}\raisebox{#1}{$#2#3$}}
\makeatother

\newcommand{\lsb}[2]{#1_{\rmh{-3pt}{#2}}}

\usepackage{geometry}

\makeatletter
\newcommand{\vo}{\vec{o}\@ifnextchar{^}{\,}{}}
\makeatother

\numberwithin{equation}{section}
\everymath{\displaystyle}

\newtheorem{theorem}{Theorem}[section]
\newtheorem{lemma}[theorem]{Lemma}
\newtheorem{corollary}[theorem]{Corollary}
\newtheorem{proposition}[theorem]{Proposition}

\newtheorem{definition}[theorem]{Definition}
\newtheorem{remark}[theorem]{Remark}        
  
\numberwithin{equation}{section}

\def\be{\beta}
\def\ga{\gamma}
\def\de{\delta}
\def\ve{\varepsilon}

\def\ep{\ve}

\def\sig{\sigma}

\def\om{\omega}

\def\Om{\Omega}

\def\la{\lambda}
\def\La{\Lambda}

\def\tOm{\tilde{\Om}}

\def\pa{\partial}
\def\bff{{\bf f}}
\def\tr{\tilde{r}}

\def\tu{\tilde{u}}
\def\tw{\tilde{w}}
\def\tv{\tilde{v}}

\def\tbg{\tilde{\bf g}}

\def\ov{\overline{v}}

\def\Omt{\tilde{\Om}}

\newcommand{\avg}[2]{ \left({#1}\right)_{#2}}

\def\aa{\mathcal{A}}
\def\bb{\mathcal{B}}
\def\bbb{\overline{\mathcal{B}}}

\def\flam{F_{\la}}
\def\vlam{v_{\la}}
\newcommand{\mm}{\mathcal{M}}

\newcommand\RR{\mathbb{R}}

\newcommand\NN{\mathbb{N}}

\newcommand{\iprod}[2]{\langle #1 ,  #2\rangle}

\newcommand{\lbr}[1][(]{\left#1}
\newcommand{\rbr}[1][)]{\right#1}
\newcommand{\pp}{p(\cdot)}
\newcommand{\qq}{q(\cdot)}

\newcommand{\mt}{\mathfrak{t}}

\DeclareMathOperator{\dv}{div}
\DeclareMathOperator{\spt}{spt}

\DeclareMathOperator{\diam}{diam}

\DeclareMathOperator{\llll}{Log}

\newcommand{\gss}{(\ga,\sigma,S_0)}

\newcommand{\plog}{p^{\pm}_{\log}}
\newcommand{\qlog}{q^{\pm}_{\log}}
\newcommand{\slog}{s^{\pm}_{\log}}
\newcommand{\rlog}{r^{\pm}_{\log}}
\newcommand{\logh}{\llll^{\pm}}

\def\btr{B_{2\tr}}

\def\btrh{B_{\tr}}

\newcommand{\pu}{p^+_{\Om_{4R}}}
\newcommand{\pd}{p^-_{\Om_{4R}}}

\newcommand{\sss}{s(\cdot)}

\newcommand{\omf}{\Om_{4R}}
\newcommand{\omth}{\Om_{3R}}
\newcommand{\omt}{\Om_{2R}}

\newcommand{\integral}[3]{\int_{#1} #2 \ #3}

\newcommand{\mint}[3]{\fint_{#1} #2 \ #3}
\newcommand{\Norm}[1]{\left|\hspace{-0.2mm}\left| #1 \right|\hspace{-0.2mm}\right|}
\newcommand{\gh}[1]{\left( #1\right)}
\newcommand{\mgh}[1]{\left\{ #1\right\}}
\newcommand{\bgh}[1]{\left[ #1\right]}

\newcommand{\OO}{\Omega}
\newcommand{\M}{\mathcal{M}}
\begin{document}

\begin{frontmatter}

\title{Sharp gradient estimates for quasilinear elliptic equations with $p(x)$ growth on nonsmooth domains.}

\author[myaddress]{Karthik Adimurthi\corref{mycorrespondingauthor}\tnoteref{thanksfirstauthor}}
\cortext[mycorrespondingauthor]{Corresponding author}
\ead{karthikaditi@gmail.com and kadimurthi@snu.ac.kr}
\tnotetext[thanksfirstauthor]{Supported by the National Research Foundation of Korea grant (NRF-2015R1A2A1A15053024).}

\author[myaddress]{Sun-Sig Byun\tnoteref{thankssecondauthor}}
\ead{byun@snu.ac.kr}
\tnotetext[thankssecondauthor]{Supported by the National Research Foundation of Korea grant (NRF-2017R1A2B2003877). }

\author[myaddress]{Jung-Tae Park\tnoteref{thanksthirdauthor}}
\ead{ppark00@snu.ac.kr}
\tnotetext[thanksthirdauthor]{Supported by the National Research Foundation of Korea grant (NRF-2015R1A4A1041675).} 

\address[myaddress]{Department of Mathematical Sciences, Seoul National University, GwanAkRo 1, Gwanak-Gu, Seoul 08826, South Korea.}

\begin{abstract}
In this paper, we study quasilinear elliptic equations with the nonlinearity modelled after the $p(x)$-Laplacian on nonsmooth domains and obtain sharp Calder\'on-Zygmund type estimates in the variable exponent setting. In a recent work of \cite{BO}, the estimates  obtained  were strictly above the natural exponent and hence there was a gap between the natural energy estimates and estimates above $p(x)$, see \eqref{energy_introduction} and \eqref{byun_ok_estimate}.  Here, we bridge this gap to obtain the end point case of the estimates obtained in \cite{BO}, see \eqref{our_estimate}. In order to do this, we have to obtain significantly  improved a priori estimates below  $p(x)$, which is the main contribution of this paper.  We also improve upon the previous results by obtaining the estimates for a larger class of domains than what was considered in the literature.
\end{abstract}

\begin{keyword}
variable exponent, Calderon-Zygmund theory, end-point estimate.
\MSC[2010] 35J92, 35B65, 35R05, 46F30
\end{keyword}

\end{frontmatter}

\section{Introduction}

Calder\'on-Zygmund theory was first developed for the  Poisson equation in \cite{CZ}, which related the integrability of the gradient  of the solution for the Poisson equation with that of the associated  data. This represented the starting point of obtaining a priori estimates in Sobolev spaces for elliptic and parabolic equations.

\emph{All the estimates mentioned in this introduction are quantitative in nature, but to avoid being too technical, we only recall the qualitative nature of the bounds. This is sufficient to highlight the nature of the results that we will prove in this paper. }

For the problem 
\[
 \left\{
 \begin{array}{rcll}
  \dv (|\nabla u|^{p-2} \nabla u) &=& \dv (|\bff|^{p-2} \bff) & \quad \text{in} \ \Om,\\
  u &=& 0 & \quad \text{on} \ \partial \Om,
 \end{array}
\right.
\]
T. Iwaniec in \cite{IW1} established the Calder\'on-Zygmund type estimates, in particular he proved the following a priori relation
\[
 |\bff| \in L^{q}_{loc} \Longrightarrow |\nabla u| \in L^q_{loc} \qquad \text{for all} \ q >p. 
\]
After this pioneering work, there have been numerous publications which extended these estimates to other quasilinear elliptic and parabolic equations with the constant $p$-growth, see \cite{AM2,VB,CP,BM,DMS,KZ,MM} and references therein. In this paper, we are interested in obtaining Calder\'on-Zygmund type bounds for the problem
\begin{equation}
 \label{basic_pde}
\left\{ \begin{array}{rcll}
  \dv \aa(x,\nabla u) &=& \dv (|\bff|^{p(x)-2} \bff) & \quad \text{in} \ \Om,\\
  u &=& 0 & \quad \text{on} \ \partial \Om. 
 \end{array}\right. 
\end{equation}
Here $\OO$ is a bounded domain of $\RR^n$, $n \ge 2$, and the quasilinear operator $\aa(x,\nabla u)$ is modelled after well known $p(x)$-Laplacian operator having the  form $ |\nabla u|^{p(x) - 2} \nabla u$.  See Section \ref{assumptions} for the precise assumptions on $\aa(\cdot,\cdot)$, $\pp$ and $\Om$. For more on the importance of variable exponent problems, see \cite{RR,Ru,VVZ,CLR,AS,HU} and the references therein. 


The first estimate for the $p(x)$-Laplacian was obtained by Acerbi and Mingione in \cite{AM}, wherein they obtained a local Calder\'on-Zygmund type estimate by proving
\begin{equation}
 |\bff|^{\pp} \in L^q_{loc} \Longrightarrow |\nabla u|^{\pp} \in L^q_{loc} \qquad \text{for all} \ q \in (1,\infty)
\end{equation}
under the assumption that the variable exponent $\pp$ satisfies $\lim_{r\rightarrow0}\rho(r) \log \lbr \frac{1}{r} \rbr = 0$ (see Section \ref{exponent_structure} for the relation between $\rho$ and $\pp$). This work was subsequently extended in \cite{BaB} to parabolic systems.

 This estimate was further improved upon in \cite{BO1} for more general equations of the form \eqref{basic_pde},  to obtain global Calder\'on-Zygmund type estimates, provided the nonlinearity $\aa(x,\zeta)$ satisfied a small BMO (bounded mean oscillation) condition with respect to $x$ and $\Om$ was sufficiently flat in the sense of Reifenberg (see \cite{BW2,GP,IKM,Mil,Se} and the references therein for more about the  BMO condition and Reifenberg flat domains).

 On the other hand, there are very few results about $L^{\qq}$ estimates, i.e., a priori relations of the form
 \[
  |\bff|^{\pp} \in L^{\qq}(\Om) \Longrightarrow |\nabla u|^{\pp} \in L^{\qq}(\Om) \quad \text{for a  suitably chosen} \ \qq. 
 \]
The first work in this direction was due to \cite{DLR} in which they considered the linear Poisson equation to obtain $L^{\qq}$-estimates. The main approach is based on the boundedness of the associated singular integral operators in the variable exponent Lebesgue spaces $L^{\qq}$ with the assumption that $\qq$ is $\log$-H\"older continuous (this is essentially a necessary condition and cannot be avoided). More general extensions are now well known as an application of the theory developed in \cite{USW}.

Recently in \cite{USW}, the authors proved a very interesting generalization of the Theory of Extrapolation by Rubia de Francia and Garcia Cuerva in the setting of variable exponent Lebesgue spaces.  More specifically,  the problem of obtaining $L^{\qq}$ estimates for constant exponent equations becomes equivalent to obtaining weighted estimates for a fixed exponent $p_0$ with the weight in a suitable Muckenhoupt class $A_{p_0}$ (see \cite[Theorem 2.7]{USW} and details therein) .  In the past several decades, there have been a plethora of weighted bounds with weights in Muckenhoupt class obtained for a wide range of constant exponent operators including those considered in \cite{DLR}. 

Subsequently, the interesting equations to study now are those where the nonlinear structure depends on the variable exponent (an example being \eqref{basic_pde}). It is easy  to obtain energy estimates for solutions $u \in W_0^{1,\pp}(\Om)$ on any bounded domain, i.e.,  the following relation holds:
\begin{equation}
 \label{energy_introduction}
 |\bff|^{\pp} \in L^{1}(\Om) \Longrightarrow |\nabla u|^{\pp} \in L^{1}(\Om).
\end{equation}

Recently in \cite{BO}, the authors obtained the following Calder\'on-Zygmund type relation provided the nonlinearity satisfies a small BMO condition and the domain is suitably flat in the sense of Reifenberg:
\begin{equation}
 \label{byun_ok_estimate}
 |\bff|^{\pp} \in L^{\qq}(\Om) \Longrightarrow |\nabla u|^{\pp} \in L^{\qq}(\Om),  
\end{equation}
 for any  $1<q^-\leq \qq \leq q^+ < \infty $ with $\qq$ being $\log$-H\"older continuous.  \emph{In particular, they cannot take $q^- = 1$, which  would recover \eqref{energy_introduction}. }

 {\bf The purpose of this paper is twofold: }Firstly, we bridge the gap between the estimates \eqref{energy_introduction} and \eqref{byun_ok_estimate} to obtain the relation
 \begin{equation}
 \label{our_estimate}
 |\bff|^{\pp} \in L^{\qq}(\Om) \Longrightarrow |\nabla u|^{\pp} \in L^{\qq}(\Om),  
\end{equation}
 {\bf provided  $1\leq q^-\leq \qq \leq q^+ < \infty $, i.e., we allow $q^- = 1$} and $\qq$ is $\log$-H\"older continuous (we assume the same structure conditions on the nonlinearity $\aa$ as in \cite{BO}). \emph{This represents an end point case of the estimate \eqref{byun_ok_estimate}. }

 Secondly, in \cite{BO}, they considered domains that were sufficiently flat in the sense of Reifenberg. Although this class includes domains that have fractal boundary, it excludes convex domains with sufficiently sharp corners. In this paper, we obtain the end point estimate for quasiconvex domains (see Subsection \ref{domain_structure} for the details). \emph{This class of domains includes both Reifenberg flat domains as considered in \cite{BO} and convex domains.  }

 The plan of the paper is as follows: In Section \ref{assumptions}, we collect all the assumptions that will be needed on the structure of the nonlinearity $\aa(\cdot,\cdot)$, the regularity of the boundary of the domain $\Om$ and the assumptions on the variable exponent $\pp$. In Section \ref{main_theorem}, we shall state the main results of this paper.  In Section \ref{background_material}, we shall collect all the preliminary results that will be needed in later parts of the paper.  Section \ref{aprioriestimates} is devoted to proving the main a priori estimates that will be needed. In Section \ref{covering_arguments}, we shall implement a covering type argument and finally in Section \ref{CZ_estimates}, we  prove the main theorems. 


\section{Assumptions on the structures of $\aa$, $\pp$ and $\Om$.}
\label{assumptions}

We shall first describe the assumptions imposed on the variable exponent:
\subsection{Structure of $\pp$}
\label{exponent_structure}

\begin{definition}
\label{definition_p_log}
 We say that, a bounded measurable function $\pp : \Om \rightarrow \RR$  belongs to the $\log$-H\"older class $\logh$, if the following conditions are satisfied:
 \begin{itemize}
  \item There exist constants $p^-$ and $p^+$ such that $1< p^- \leq p(x) \leq p^+ < \infty$ for every $x \in \Om$.
  \item $ |p(x)  - p(y)| \leq \frac{L}{- \log |x-y|}$ holds  for every $ x, y \in \Om$ with $ |x-y| \leq \frac12 $ and for some  $L>0$.
 \end{itemize}
\end{definition}

\begin{remark}\label{remark_def_p_log} We remark that  $\pp$ is log-H\"{o}lder continuous in $\Om$ if and only if there is a nondecreasing continuous function ${\rho} : [0,\infty) \rightarrow [0,\infty)$ such that 
\begin{itemize}
 \item $\lim_{r\rightarrow 0} {\rho}(r) = 0$ and $|p(x)- p(y)| \leq {\rho}(|x-y|)$ for every $x,y \in \Om$. 
 \item ${\rho}(r) \log \lbr \frac{1}{r} \rbr \leq {L}$ holds for all $ 0< r \leq \frac12.$
\end{itemize}
The function $\rho(r)$ is called the modulus of continuity of the variable exponent $\pp$. 
 \end{remark}

%
%


\subsection{Structure of the nonlinearity $\aa(\cdot,\cdot)$}
\label{nonlinear_structure}
We first assume that $\aa(\cdot,\cdot)$ is a Caratheodory function in the sense:
\begin{gather}
x \mapsto \aa(x,\zeta) \ \text{is measurable for every } \ \zeta \in \RR^n \nonumber, \\
\zeta \mapsto\aa(x,\zeta) \ \text{is continuous for almost every } \ x \in \Om \nonumber.
\end{gather}
Let $\mu \in [0,1]$ be giiven, then there exists two positive contants $\La_0,\La_1$ such that the following holds for almost every $x \in \Om$ and every $\zeta, \eta \in \RR^n$: 
\begin{gather}
(\mu^2 + |\zeta|^2)^{\frac12} |D_{\zeta} \aa(x,\zeta)| + |\aa(x,\zeta)| \leq \La_1 (\mu^2 + |\zeta|^2)^{\frac{p(x) -1}{2}} \label{bounded}, \\
(\mu^2 + |\zeta|^2 )^{\frac{p(x)-2}{2}} |\eta|^2 \La_0 \leq \iprod{D_{\zeta}\aa(x,\zeta)\eta}{\eta} \label{ellipticity}.
\end{gather}
%
%
%
%
We point out that from \eqref{ellipticity}, one can derive the following monotonicity bound: 
\begin{equation}
 \label{monotonicity}
 \iprod{\aa(x,\zeta) - \aa(x,\eta)}{\zeta - \eta} \geq \tilde{\La}_0 (\mu^2 + |\zeta|^2 + |\eta|^2)^{\frac{p(x)-2}{2}} |\zeta- \eta|^2,
\end{equation}
where $\tilde{\La}_0 = \tilde{\La}_0(\La_0,n,p^+,p^-)$.  By inserting $\eta=0$ into \eqref{monotonicity}, we also have the following coercivity bound:
\begin{equation}\label{coercivity}
 \tilde{\La}_2 |\zeta|^{p(x)} \leq \iprod{\aa(x,\zeta)}{\zeta} + \tilde{\La}_1,
\end{equation}
where $\tilde{\La}_1 = \tilde{\La}_1 (\La_1,\La_0,p^+,p^-,n)$ and $\tilde{\La}_2 = \tilde{\La}_2(\La_1,\La_0,p^+,p^-,n)$.

\subsection{Structure of $\Om$}
\label{domain_structure}
Let $\ga>0$, $\sig\in(0,1/4)$ and $S>0$ be given, then we describe the properties of a $(\gamma, \sigma, S)$-quasiconvex domain $\Om$ in this subsection:
\begin{definition}
 \label{quasiconvex_definition}
$\Om$ is said to be $(\gamma, \sigma, S)$-quasiconvex if for all $x \in \partial \Om$ and all $0<r\leq S$, the following properties hold:
 \begin{itemize}
  \item There exists a ball $B_{\sigma r}(z) \subset \Omega \cap B_r(x)$ where $z$ is relative with respect to $x$ and the given $\sigma \in \gh{0,1/4}$ is  independent of $x$,
  \item There exist a hyperplane $L(x,r)$ containing $x$, a unit normal vector $\vec{\bf n}(x,r)$ to $L(x,r)$ and a half space \[ H(x,r) = \{ y+ t \vec{\bf n}(x,r) : y \in L(x,r), t \geq -\gamma r\} \] such that \[ \Omega \cap B_r(x) \subset H(x,r) \cap B_r(x).\]
 \end{itemize}

\end{definition}

%
%

We now state the following two important properties regarding quasiconvex domains. The first lemma says that locally, we can approximate the domain $\Om$ by convex domains from outside at sufficiently small scales. 
\begin{lemma}[\cite{JJW1,JJW2,JJW3}] \label{convex-inside} 
Let  $\Om$ be a $(\gamma, \sigma, S)$-quasiconvex domain, then for any $x \in \partial \Om$ and $r \in (0,S/2]$, there exists a convex domain $F(x,r)$ such that the following holds 
\[ B_r(x) \cap \Om \subset F(x,r) \cap B_r(x)\quad \text{and} \quad D[\partial_w F(x,r), \partial_w \Om_r(x)] \leq c(\sigma) \gamma r.\]
%
\end{lemma}

Here, we have set $D[E,F] := \max\left\{\sup_{z\in E}d(z,F), \sup_{z\in F} d(z,E)\right\}$ which denotes the Hausdorff distance between two sets $E,F \subset \RR^n$, and $\partial_w$ denotes the wiggle part of the boundary, i.e.,
$$\partial_w \Om_r(x) := \partial \Om \cap B_r(x) \quad \text{and} \quad \partial_w F(x,r) :=\partial F(x,r) \cap B_r(x).$$

\begin{remark}
 Following the proof in \cite{JJW1}, we see that $F(x,r)$ is constructed to be the set \[ F(x,r) := \bigcap_{y \in \partial_w\Om_r(x)} H(y,2r),\] where $H(y,2r)$ is defined in Definition \ref{quasiconvex_definition}. 
\end{remark}

The second property of quasiconvex domains is that locally, it can be approximated by convex domains from the interior at suitably small scales. 
\begin{lemma}[\cite{JJW1,JJW2,JJW3}]\label{convex-outside}
Let $\Om$ be a $(\ga,\sig,S)$-quasiconvex domain. For the convex region $F(x,r)$ constructed in Lemma \ref{convex-inside}, there exists another convex region $F^*(x,r)$ such that the following holds:
\[ F^*_r(x) := F^*(x,r) \cap B_r(x) \subset \Om_r(x)\quad \text{and} \quad D[\partial_w F^*_r(x), \partial_w \Om_r(x)]  \leq \frac{32 \gamma r}{\sigma^3}.\]
%
\end{lemma}

\begin{remark}
 Following the proof in \cite{JJW1}, we see that $F^*(x,r)$ is constructed as follows: Using the spherical coordinates centered at $x_0$, it is easy to see that there exists a point $x_0 \in \Om_r(x)$ such that $B_{\sigma r}(x_0) \subset \Om_r(x)$. Define
 \[ 
 F^*(x,r) := \left\{(\rho,\theta_1,\ldots,\theta_{n-1}) \ : \begin{array}{l} \rho \leq \rho' \left( 1-\frac{16\gamma}{\sigma^3}\right), \ (\rho', \theta_1, \ldots , \theta_{n-1}) \in \partial F(x,r) \end{array} \right\}.
 \]
\end{remark}

A useful property of a $(\ga,\sig,S)$-quasiconvex domain is that it satisfies the following measure density estimates:
\begin{lemma}
 \label{measure_density_quasiconvex}	
 Since $\Om$ is $(\ga,\sigma,S)$-quasiconvex,  we have the following estimates: For any $x \in \pa \Om$ and any $r \in (0,S)$, then there holds
 \begin{equation}
  \label{lower_bound_domain}
  |\Om \cap B_r(x)| = |\Om_r(x)| \geq |B_{\sigma r}|.
 \end{equation}
Analogously, the following bound also holds:
 \begin{equation}
  \label{lower_bound_domain_complement}
 |B_r(x) \cap \Om^c|  \geq c(n) [(1-\ga)r]^n.
 \end{equation}
If we further assume that $\ga \leq \frac12$, then the bound in \eqref{lower_bound_domain_complement} can be made independent of $\ga$. 
\end{lemma}

\subsection{Smallness assumption}
\label{smallness_assumption}
In order to prove the main results, we need to assume a smallness condition satisfied by  $(\pp,\aa,\Om)$. 
\begin{definition}\label{further_assumptions}
 Let $\ga>0$, $\sig>0$ and $S_0>0$ be given, we then say $(\pp,\aa,\Om)$ is $(\ga,\sigma,S_0)$-vanishing  if the following three assumptions hold:
 \begin{description}[leftmargin=*]
  \item[(i) Assumption on $\pp$:] The variable exponent $\pp$ with modulus of continuity $\rho(r)$ as defined in Definition \ref{definition_p_log}, is further assumed to satisfy the smallness condition:
  \begin{equation}
   \label{small_px}
   \sup_{0<r\leq S_0} {\rho}(r) \log \lbr \frac{1}{r} \rbr \leq \ga.
  \end{equation}

  \item[(ii) Assumption on $\aa$:] For a bounded open set $U \subset \RR^n$, we write
  \begin{equation}
   \label{a_difference}
   \Theta(\aa,U)(x) := \sup_{\zeta \in \RR^n} \lbr[|] \frac{\aa(x,\zeta)}{(\mu^2 + |\zeta|^2)^{\frac{p(x)-1}{2}}} - \avg{\frac{\aa(\cdot,\zeta)}{(\mu^2 + |\zeta|^2)^{\frac{\pp-1}{2}}}}{U} \rbr[|].
  \end{equation}
where we have used the notation $\avg{f}{U} := \fint_U f(x)\ dx$.
  Note that if $\mu=0$, then $\zeta \in \RR^n\setminus \{0\}$.  Then $\aa$ satisfies the small BMO condition, i.e., there holds:	
\begin{equation}
 \label{small_aa}
 \sup_{0<r\leq S_0} \sup_{y \in  \RR^n} \fint_{B_r(y)} \Theta(\aa,B_r(y))(x) \ dx \leq \ga.
\end{equation}
\item[(iii) Assumption on $\pa \Om$:] We ask that $\Om$ is a $\gss$-quasiconvex domain  in the sense of  Definition \ref{quasiconvex_definition}. 
 \end{description}
\end{definition}

\subsection{Notations}
We shall use the following notation throughout the paper:
\begin{itemize}
     \item In what follows, the function $\rho(r)$ denotes the modulus of continuity of $p(x)$ and we denote $\om(r)$ for the modulus of continuity of $q(x)$.
\item For the variable exponent $\pp$, we shall denote by $\plog$ to include the constants $p^+$, $p^-$ and those that are part of the $\log$-H\"older continuity structure of $\pp$. Analogously, for variable exponents $\qq$, $r(\cdot)$ and $\sss$, we shall use $\qlog$, $\rlog$ and $\slog$ to denote corresponding constants. 
\item We shall sometimes also write $q_{\log}$  to denote constants that depend only on the log-H\"older continuity of $\qq$  and denote $q^{\pm}$ to denote the constants $q^+$ and $q^-$. 
     \item Constants with subscripts like radii $R_1,R_2\ldots$ and bounding values $M_1,M_2,\ldots$ will be fixed in subsequent sections once they are chosen. 
     \item We shall use $\apprle$,  $\apprge$ and $\approx$ to suppress writing the constants that could possibly change from line to line as long as they depend on $n,\plog,\qlog,\La_0,\La_1,\sig$, $S_0$ and related quantitites.
     \item We shall sometimes use $\sim$ to denote variables that occur only within the proof of the concerned result, for example $\tilde{r}, \tilde{m}, \cdots$.
     \item Given a variable exponent $\pp$, we shall use the following notation:
     \begin{equation*}\label{notation_p_inf}
 p_E^- : = \inf_{x \in E} p(x) \qquad \text{and} \qquad p_E^+ := \sup_{x \in E} p(x).
\end{equation*}
We will drop the set $E$ and denote $p^+:= \sup_{x\in \RR^n} p(x)$ and $p^-:=\inf_{x\in \RR^n} p(x)$

\end{itemize}

 
\section{Main theorems}
\label{main_theorem}

We now state the main results of the paper. The first theorem concerns local estimates around small balls:
\begin{theorem}
 \label{main_theorem1}
Assume that $u \in W_0^{1,p(\cdot)}(\OO)$ is the weak solution of the problem (\ref{basic_pde}) under the structure conditions (\ref{bounded}) and (\ref{ellipticity}). Let $0 < \sigma < 1/4$, $0< S_0 < 1$, and $q(\cdot)$ be log-H\"{o}lder continuous satisfying $1 < q^- \le q(\cdot) \le q^+ < \infty$. There exist constants ${\gamma_0} \in (0,1/4)$ and ${\delta_0} \in (0,1/4)$, both depending only on $\La_0$, $\La_1$, $\plog$, $\qlog$, $n$, $\sigma$, such that if $(p(\cdot),\aa,\OO)$ is $(\gamma,\sigma,S_0)$-vanishing for some $\ga \in (0,{\ga_0})$, then there exists a constant $C_0 = C_0(\La_0,\La_1,\plog, \qlog, n, S_0)>0$ so that for any $x_0 \in \OO$, $\delta \in (0,{\delta_0})$  and $r \in \left(0, 1/(C_0 M_0) \right]$, we have 
\begin{equation*}\label{main-r1}
\begin{aligned}
\mint{\OO_{r}(x_0)}{|\nabla u|^{(p(x)-\delta)q(x)}}{dx} &\le C \left\{\gh{\mint{\OO_{4r}(x_0)}{|\nabla u|^{p(x)-\delta}}{dx}}^{q_{\OO_{4r}(x_0)}^-} + \mint{\OO_{4r}(x_0)}{|\bff|^{(p(x)-\delta)q(x)}}{dx} +1 \right\}
\end{aligned}
\end{equation*}
for some constant $C=C(\La_0,\La_1,\plog, \qlog, n, \sigma)>0$ and $M_0$ as defined in \eqref{size_date}.
\end{theorem}

Using a standard covering argument, we can obtain the following  global estimates:
\begin{theorem}
\label{main_theorem2}
Let  $M^+ > 1$ be given and  let $r(\cdot)$ be a log-H\"{o}lder continuous function satisfying $1 \le r^- \leq r(\cdot) \leq r^+ < M^+ < \infty$. Let $0 < \sigma < 1/4$ and  $0< S_0 < 1$ be given, then under the assumptions in Theorem \ref{main_theorem1}, there is a constant ${\gamma_0} \in (0,1/4)$  depending only on $\La_0$, $\La_1$, $\plog$, $r_{\log}$, $M^+$, $n$ and $\sigma$, such that if $(p(\cdot),\aa,\OO)$ is $(\gamma,\sigma,S_0)$-vanishing for some $\ga \in (0,{\ga_0})$, then 
there exists a constant $C = C(\La_0, \La_1, \plog, r_{\log}, M^+, n, \sigma, \OO, S_0)>0$ such that the following global bound holds:
\begin{equation*}\label{main-r2}
\integral{\OO}{|\nabla u|^{p(x)r(x)}}{dx} \le C \mgh{\gh{\integral{\OO}{|\bff|^{p(x)r(x)}}{dx}}^{n(M^+ -1) + M^+} +1}.
\end{equation*}
\end{theorem}

\begin{remark}
In Theorem \ref{main_theorem2}, we restrict $r^+ < M^+$ and we cannot take $r^+=M^+$ in general (see Section \ref{CZ_estimates} for details). This is unfortunately an artifact of the variable exponent spaces and the techniques used in this paper. It would be interesting to know  if this restriction can be removed! 

\end{remark}

%

%
%


\section{Background material}
\label{background_material}

In this section, we shall collect and in some cases, prove all the necessary details needed in subsequent Sections. 

\subsection{Sobolev spaces with variable exponents}

Let $\tOm$ be a bounded domain, let $\sss$ be an admissible variable exponent as in Section \ref{exponent_structure} and let $\om: \tOm \to (0,\infty)$ be any weight function. Given a  positive integer $m$, the \emph{variable exponent Lebesgue space} $L^{\sss}_{\om}(\Omt,\RR^m)$ consists of all measurable functions $\bff: \tilde{\Om} \to \RR^m$ satisfying 
\[
 \int_{\tilde{\Om}} |\bff(x)|^{s(x)}\om(x) \ dx < \infty
\]
having the norm
\[
 \|\bff\|_{L^{\sss}_{\om}(\Omt,\RR^m)} := \inf \left\{ \la >0 : \int_{\tilde{\Om}} \left|\frac{\bff(x)}{\la} \right|^{s(x)}\om(x) \ dx \leq 1 \right\}.
\]

Analogously, we can define the \emph{variable exponent Sobolev space} as 
\[
 W^{1,\sss}_{\om}(\Omt,\RR^m) := \{ \bff \in L^{\sss}_{\om}(\Omt,\RR^m) : \nabla\bff \in L^{\sss}_{\om}(\Omt,\RR^{mn}) \}
\]
equipped with the norm 
\[
 \| \bff \|_{W^{1,\sss}_{\om}(\Omt,\RR^m)} := \| \bff \|_{L^{\sss}_{\om}(\Omt,\RR^m)} + \| \nabla \bff \|_{L^{\sss}_{\om}(\Omt,\RR^{mn})}.
\]

We shall denote ${W_{0,\om}^{1,\sss}}(\Omt,\RR^m)$  to be the closure of $C_c^{\infty}(\Omt,\RR^m)$ in $W^{1,\sss}_{\om}(\Omt,\RR^m)$. For $m=1$, we write $L^{\sss}_{\om}(\Omt)$, $W^{1,\sss}_{\om}(\Omt)$ and ${W_{0,\om}^{1,\sss}}(\Omt)$ for simplicity. In the case $\om \equiv 1$, the function spaces above become the standard variable exponent spaces as described in \cite{Diening}.

We will sometimes also use the following notation in the case $\om(\cdot) \equiv 1$:
\[ \varrho_{L^{\sss}(\Omt)}(f) := \int_{\Omt} |f(x)|^{s(x)} \ dx.\]

We mention the following useful relation between the modular and norm in the variable exponent spaces (see \cite[Lemma 3.2.5]{Diening} for details):
\begin{lemma}
 \label{integral_norm}
 For any $f \in L^{\sss}(\Omt)$, the following holds:
 \begin{equation*}
  \label{norm_integral}
  \min \left\{ \varrho_{L^{\sss}(\Omt)} (f)^{\frac{1}{s^-}}, \varrho_{{L^{\sss}(\Omt)}} (f)^{\frac{1}{s^+}} \right\} \leq \| f\|_{L^{\sss}(\Omt)} \leq \max \left\{ \varrho_{{L^{\sss}(\Omt)}} (f)^{\frac{1}{s^-}}, \varrho_{{L^{\sss}(\Omt)}} (f)^{\frac{1}{s^+}} \right\}.
 \end{equation*}
\end{lemma}

\subsection{Muckenhoupt weights in variable exponent spaces}

We shall define the class of variable exponent Muckenhoupt weights $A_{\sss}$ following \cite{DH} and use them to prove certain useful bounds on the maximal function. 

\begin{definition} We say that a  measurable function $w:\RR^n \to (0,\infty)$ is an $A_{\sss}$ weight if 
\begin{equation}
\label{variable_weight_definition}
[w]_{A_{\sss}} := \sup_{B} \frac{1}{|B|^{s_B}} \|w\|_{L^1(B)} \lbr[\|] \frac{1}{w} \rbr[\|]_{L^{\frac{s'(\cdot)}{\sss}}(B)} < \infty,
\end{equation}
where the supremum is taken over all balls $B \subset \RR^n$. Here we have defined $s_B := \lbr \fint_B \frac{1}{s(x)} \ dx \rbr^{-1}$. 
\end{definition}

Following \cite{DH}, we shall collect some of the important properties of $A_{\sss}$ weights. The first lemma shows that the Muckenhoupt weights form an increasing class (see \cite[Lemma 3.1]{DH}):
\begin{lemma}
 \label{weight_inclusion}
 Let $\sss,\qq \in \logh$ with $\qq \leq \sss$ pointwise, then there exists a constant $C_{incl} = C_{incl}(\slog,\qlog)$ such that  $[w]_{A_{\sss}} \leq C_{incl} [w]_{A_{\qq}}$.
 
\end{lemma}

The following Lemma lets us compare the behavior of the variable exponent Muckenhoupt weights over balls of different radii (see \cite[Lemma 3.3]{DH}): 
\begin{lemma}
\label{lemma3.3}
Let $\sss \in \logh$ and let $w \in A_{\sss}$ be a given weight, then the following holds for any $x,y \in \RR^n$ and $r,R>0$:
\begin{equation}
 \label{weight_lower_bound}
 w(B_r(x)) \apprge w(B_R(y)) \lbr \frac{r^n}{|x-y|^n + r^n + R^n} \rbr^{s^+},
\end{equation} 
where $w(B) := \integral{B}{w(x)}{dx}$.
\end{lemma}

Using the previous Lemma, we can obtain the following fundamental estimates which relate the norm and modular of a characteristic function in the weighted setting,  provided the weight is in $A_{\infty}$. This property will play a crucial role in the next subsection (see \cite[Lemma 3.4]{DH}):
\begin{lemma}
 \label{lemma3.4}
 Let $\sss \in \logh$ and $w \in A_{\infty}$ and suppose $B$ is a ball with $\diam B \leq 2$, then for any $x \in B$, there holds
 \[
  \|1\|_{L^{\sss}_w(B)} \approx w(B)^{\frac{1}{s_B^+}} \approx w(B)^{\frac{1}{s_B^-}} \approx w(B)^{\frac{1}{s(x)}} \approx w(B)^{\frac{1}{s_B}}.
 \]
\end{lemma}

For $w \in A_{\sss}$, we define a \emph{dual weight} by $w'(y) = w(y)^{1-s'(y)}$ where $s'(y) = \frac{s(y)}{s(y)-1}$. The next lemma proves the variable exponent analogue of the fact that  $[w]_{A_s} = [w']_{A_{s'}}^{s-1}$ and so $w \in A_s$ if and only if $w'\in A_{s'}$ (see \cite[Proposition 3.8]{DH} for the details):
\begin{proposition}
Let $\sss \in \logh$ and $w \in A_{\sss}$, then $w' \in A_{s'(\cdot)}$ and 
\[
 |B|^{-s_B} \| w\|_{L^1(B)} \left\| \frac{1}{w} \right\|_{L^{\frac{s'(\cdot)}{\sss}}(B)} \approx \frac{w(B)}{|B|} \lbr \frac{w'(B)}{|B|}\rbr^{s_B-1}.
\]
\end{proposition}

Let us now recall another basic property which is  the variable exponent analogue of the reverse factorization result (See \cite[Proposiiton 3.13]{DH} for the details):
\begin{proposition}
Let $\sss \in \logh(\RR^n)$ and $ w_1,w_2 \in A_1$, then $w_1(x) w_2(x)^{1-s(x)} \in A_{\sss}$. 
\end{proposition}

%
%

\subsection{Hardy Littlewood maximal function}

In this subsection, we shall prove a few important estimates involving the truncated maximal function.  We define the truncated maximal function by:
\[
 \mm_{<R} f(x) := \sup_{r<R} \fint_{B_r(x)} |f(y)| \ dy.
\]
Note that $\mm f$ is the standard maximal function of $f$. 

Following the ideas in \cite[Lemma 5.1]{DH}, we prove the following important theorem:
\begin{theorem}
\label{integral-px-maximal-weighted}
 Let $\sss \in \logh(\RR^n)$ and let $\om \in A_{\sss}$ be a Muckenhoupt weight. Assume that $M_1,M_2>0$ are given constants such that  $[\om]_{A_{\sss}} \leq M_1$ holds. Then there is $R_1 = R_1{(M_1,M_2,\slog)}>0$ such that for any $2r < R_1$ and for any $f \in L^{\sss}_{\om}(\RR^n)$  satisfying 
 \begin{equation}\label{assumption_on_f}  
 \int_{B_{2r}} |f(x)|^{s(x)} \om(x) \ dx + 1 \leq M_2,
 \end{equation} 
 then the following estimate holds:
 \[
  \int_{B_r} \mm_{<r}(|f|)^{s(x)}(x) \om(x)\ dx \leq C \lbr \int_{B_{2r}} |f(x)|^{s(x)}\om(x) \ dx + 1\rbr,
 \]
 where $C = C(n,M_1,\slog)$. 
\end{theorem}

\begin{proof}
Since $\om \in A_{\sss}$, we have the following bound due to Lemma \ref{weight_inclusion}:
\begin{equation}
 \label{px_weight_1}
 [\om]_{A_{s^+}} \leq C_{incl} [\om]_{A_{\sss}} \leq C_{incl} M_1.
\end{equation}
By the self-improvement property of Muckenhoupt weights (see \cite[Theorem 9.2.5]{Grafakos}), there exists $\tilde{\ep} = \tilde{\ep}(s^+,n,M_1)>0$  such that 
\begin{equation}
 \label{px_weight_2}
 [\om]_{A_{s^+-\tilde{\ep}}} \leq \tilde{c}(n,s^+) [\om]_{A_{s^+}}. 
\end{equation}
We now choose $R_1 = \min \lbr[\{] \tilde{r}_1,\tilde{r}_2,\frac14\rbr[\}] $, where $\tilde{r}_1$ and $\tilde{r}_2$ satisfy 
\begin{itemize}[leftmargin=*]
 \item With $\tilde{\ep}$ as in \eqref{px_weight_2}, using Remark \ref{remark_def_p_log}, we see that there exists $\tilde{r}_1(\slog)>0$ such that the following bound holds: $s^+_{B_{2\tilde{r}_1}} - \tilde{\ep} < s^-_{B_{2\tilde{r}_1}}$.
\item Let $\tilde{r}_2 < \frac{1}{2M_2}$.
\end{itemize}

Now fix any radius $r$ satisfying $2r < R_1$,  
it is  then easy to see using the $\log$-H\"older continuity of $\sss$ such that the following bound holds:
\begin{equation}\label{px_weight_4_2}
 |B_r(y)|^{s^-_{B_r(y)}-s^+_{B_r(y)}} \leq c(\slog,n).
\end{equation}
%
%
For any $y \in B_{r}$, consider the ball $B_\tau(y)$ with $\tau<r$ and set  $q(x) := \frac{s(x)}{s^-_{B_{2r}}}$.   Then for any $\be \geq 1$, there holds
\begin{equation}
 \label{px_weight_6}
 \begin{array}{ll}
  \lbr \fint_{B_\tau(y)} |f(x)| \ dx \rbr^{q(y)} & \leq \lbr \fint_{B_\tau(y)} |f(x)|^{q_{B_\tau(y)}^-} \ dx \rbr^{\frac{q(y)}{q_{B_\tau(y)}^-}}\\
  & \apprle \lbr \fint_{B_\tau(y)} |f(x)|^{q(x)}\be^{\frac{q(x)}{{q_{B_\tau(y)}^-}}-1} \ dx \rbr^{\frac{q(y)}{q_{B_\tau(y)}^-}}+\frac{1}{\be}.
 \end{array}
\end{equation}
Now set 
 $\be := \max \left\{ 1 , \om(B_{r})^{\frac{1}{s^-_{B_{2r}}}}\right\}$, then 
using \eqref{weight_lower_bound}, we see that 
%
 \begin{equation}
 \label{px_weight_7_1}
 \be^{\frac{q(x)}{{q_{B_\tau(y)}^-}}-1} \apprle (1+|x|)^{C(s(x) - s^-_{B_\tau(y)})} \leq c(\slog,n).
 \end{equation}

Making use of  \eqref{px_weight_7_1} into \eqref{px_weight_6}, we get 
\begin{equation}\label{px_weight_8}
\begin{array}{ll}
 \lbr \fint_{B_\tau(y)} |f(x)| \ dx \rbr^{q(y)} &\apprle |B_\tau(y)|^{1- \frac{q(y)}{q^-_{B_\tau(y)}}} \varrho_{L^{q(\cdot)}(B_\tau(y))}(f)^{\frac{q(y) - q^-_{B_\tau(y)}}{q^-_{B_\tau(y)}}} \fint_{B_\tau(y)} |f(x)|^{q(x)} \ dx +\\
 &\qquad \qquad \qquad  + \min \left\{ 1 , \om(B_{r})^{-\frac{1}{s^-_{B_{2r}}}}\right\}.
 \end{array}
\end{equation}
%
Let now bound  $\varrho_{L^{q(\cdot)}(B_\tau(y))}(f)^{\frac{q(x) - q^-_{B_\tau(y)}}{q^-_{B_\tau(y)}}}$ from above as follows:
\begin{equation}
 \label{px_weight_9}
 \begin{array}{ll}
  \varrho_{L^{q(\cdot)}(B_\tau(y))}(f) &\apprle \int_{B_\tau(y)} |f(x)|^{s(x)}\om(x) \ dx + \int_{B_\tau(y)} \om(x)^{-\frac{1}{s^-_{B_{2r}}-1}}\ dx \\
  & \apprle  \int_{B_\tau(y)} |f(x)|^{s(x)}\om(x) \ dx + [\om]_{A_{s^-_{B_{2r}}}}^{\frac{1}{s^-_{B_{2r}}-1}}\lbr\frac{|B_\tau(y)|^{s^-_{B_{2r}}}}{\om(B_\tau(y))}\rbr^{\frac{1}{s^-_{B_{2r}}-1}}  \\
&\overset{\redlabel{4.10.a}{a}}{\apprle}  M_2 + C_{incl} M_1\lbr\frac{|B_\tau(y)|^{s^-_{B_{2r}}}}{\om(B_\tau(y))}\rbr^{\frac{1}{s^-_{B_{2r}}-1}}.
  \end{array}
\end{equation}
To obtain \redref{4.10.a}{a}, we made use of \eqref{assumption_on_f}, \eqref{px_weight_2} and  \eqref{px_weight_7_1}.

Now substituting \eqref{px_weight_9} into \eqref{px_weight_8} followed by making use of Lemma \ref{lemma3.4}, the $\log$-H\"older continuity of $\sss$ along with the choice of $R_1$ and  \eqref{px_weight_4_2}, we obtain
%
%
%
\[
 \lbr \fint_{B_\tau(y)} |f(x)| \ dx \rbr^{q(y)} \apprle   \fint_{B_\tau(y)} |f(x)|^{q(x)} \ dx + \min \mgh{ 1 , \om(B_{r})^{-\frac{1}{s^-_{B_{2r}}}}}.
\]

Taking supremum over all $\tau<r$, we get
\begin{equation}
 \label{px-weight-10}
 \mm_{<r}(|f|)^{q(y)}(y) \apprle \mm_{<r}(|f|^{q(\cdot)}) (y) + \min \mgh{ 1 , \om(B_{r})^{-\frac{1}{s^-_{B_{2r}}}}}.
\end{equation}

Raising  \eqref{px-weight-10} to $s^-_{B_{2r}}$, followed by multiplying with $\om(y)$ and then integrating over $B_{r}$, we get
\[
\begin{array}{ll}
 \int_{B_{r}} \mm_{<r}(|f|)^{s(x)}(x) \om(x) \ dx  &\apprle \int_{B_{r}} \mm_{<r}(|f|^{q(\cdot)})^{s^-_{B_{2r}}} (x)\om(x) \ dx +\\
 & \qquad \qquad  + \om(B_{r})\min \mgh{ 1 , \om(B_r)^{-\frac{1}{s^-_{B_{2r}}}}}^{s^-_{B_{2r}}} \\
 & \overset{\redlabel{4.12.a}{b}}{\apprle} \int_{B_{2r}} |f(x)|^{s(x)} \om(x) \ dx + 1. 
\end{array}
\]
To obtain \redref{4.12.a}{b}, we made use of the weighted maximal function bound (see \cite[Theorem 9.19]{Grafakos}) for exponent $s^-_{B_{2r}}$ since we have $\om \in A_{s^-_{B_{2r}}}$. This completes the proof of the Theorem. 
\end{proof}
 
In the unweighted case i.e., $\om(x) \equiv 1$, Theorem \ref{integral-px-maximal-weighted} becomes
\begin{corollary}
 \label{unweighted_maximal_bound}
 Let $\sss \in \logh$,  $M_2>0$ be given and let  $R_2 := \min \left\{\frac{1}{2M_2}, \frac1{\om_n^{1/n}} \right\}$, where $\om_n$ is the volume of a unit ball  in $\RR^n$. Then for any $f \in L^{\sss}(\RR^n)$ and any radii $2r < R_2$ satisfying 
 \begin{equation*}\label{assumption_on_f_no_weight} \int_{B_{2R}} |f(x)|^{s(x)} \ dx + 1 \leq M_2, \end{equation*} 
 there holds
 \[
  \int_{B_r} \mm_{<r}(|f|)^{s(x)}(x) \ dx \apprle  \int_{B_{2r}} |f(x)|^{s(x)} \ dx + |B_{r}|.
 \]
\end{corollary}


In the constant exponent case, the results are well known (see for example \cite[Chapter 9]{Grafakos}):
\begin{lemma}\label{max ftn est1} The following bounds hold:
\begin{itemize}
\item (Weak 1-1 estimate) For any $f \in L^1(\RR^n)$ and for every $\alpha > 0$, there holds
\begin{equation}\label{weak 1-1}
\left| \mgh{x \in \RR^n : \M f(x) > \alpha} \right| \le \frac{c(n)}{\alpha} \integral{\RR^n}{|f|}{dx}.
\end{equation}

\item (Strong s-s estimate) For any $f \in L^s(\RR^n)$ with $1<s<\infty$, then $\M f \in L^s(\RR^n)$ and
\begin{equation}\label{strong p-p}
\Norm{\M f}_{L^s(\RR^n)} \le c(n,s) \Norm{f}_{L^s(\RR^n)}.
\end{equation}
\end{itemize}
\end{lemma}

\subsection{Poincar\'e type inequalities}

We shall recall the modular form of the Poincar\'{e} inequality which was proved in \cite[Theorem 8.2.4]{Diening}:
\begin{lemma}
 \label{poincare_inequality}
Let $B$ be any ball and let $v \in W^{1,\sss}(B)$, then there exists a constant $C = C(n,\slog)$ such that 
 \[
  \|v - \avg{v}{B}\|_{L^{\sss}(B)} \leq C \diam(B)\|\nabla v\|_{L^{\sss}(B)}.
 \]

\end{lemma}

The Poincar\'e inequality in Lemma \ref{poincare_inequality} is not entirely suitable for our purposes and so we need to prove an integral version of the result. To do this, we shall use 
the unweighted maximal function bound obtained in Corollary \ref{unweighted_maximal_bound} to prove the following scaled version of the Poincar\'{e} type inequality in the variable exponent spaces. We believe that this result is well known to experts, but we could not find this result in literature and hence we will present the proof. 

 \begin{lemma}
\label{scaled_poincare}
 Let $\sss \in \logh$ and $R_3 \ge 1$ be given and define  $R_3 := \min \left\{\frac{1}{2M_3}, \frac1{\om_n^{1/n}}, \frac12 \right\}$. Then for any $\phi \in W^{1,\sss}(B_{4r})$ with $4r < R_3$ satisfying 
 \begin{equation}\label{assumption_on_phi_no_weight} \int_{B_{4r}} |\nabla \phi (x)|^{s(x)} \ dx + 1 \leq M_3, \end{equation} 
 there exists a constant $C= C(n,\slog)$ such that 
 \[
  \int_{B_r} \lbr \frac{|\phi - \avg{\phi}{B_r}|}{\diam(B_r)}\rbr^{s(x)} \ dx \leq C\lbr   \int_{B_{r}} |\nabla \phi(x)|^{s(x)} \ dx + |B_{r}|\rbr.
 \]
%
%
 Since $\diam(B_r) = 2r \leq R_3 <1$, we also obtain
 \[
  \int_{B_r} {|\phi - \avg{\phi}{B_r}|}^{s(x)} \ dx \leq C \lbr  \int_{B_{r}} |\nabla \phi(x)|^{s(x)} \ dx + |B_{r}|\rbr.
 \]
\end{lemma}

\begin{proof}
 Let $B=B_r$ be any ball of radius $r \leq R_3$ and let $\phi \in W^{1,\sss}(B_s)$, then  using \cite[Theorem 1.51]{MZ}, we obtain
 \begin{equation}
 \label{eq5.16}
  |\phi(x) - \avg{\phi}{B}| \leq C(n) \int_{B} \frac{|\nabla \phi| \chi_{B}}{|x-y|^{n-1}} \ dy.
 \end{equation}
Modifying the proof of \cite[Theorem 1.32]{MZ}, we have the following bound (note that $\diam(B) \leq 1$):
\begin{equation}
\label{eq5.17}
 \int_{B} \frac{|\nabla \phi| \chi_{B}}{|x-y|^{n-1}} \ dy \leq C(n)  \mm_{<2R} (|\nabla \phi|\chi_B)(x). 
\end{equation}
Combining \eqref{eq5.16} and \eqref{eq5.17}, we get
\begin{equation}
 \label{eq5.18}
 \frac{|\phi(x) - \avg{\phi}{B}|}{\diam(B)} \leq C(n) \mm_{<2R} (|\nabla \phi|\chi_B)(x).
\end{equation}
We now exponentiate \eqref{eq5.18} by $\sss$ and then integrate over $B$. To control the maximal function term on the right,  we make use of Corollary \ref{unweighted_maximal_bound} (this is where the restriction on $R_3$ and \eqref{assumption_on_phi_no_weight} is needed) to get the final conclusion. This completes the proof of the Lemma. 
\end{proof}

We will also need the following  version of the Poincar\'{e} inequality which holds provided the \emph{zero set} of the function is sufficiently large (In the constant exponent case, this result is well known, see for example \cite[Corollary 8.2.7]{AH} and references therein). In the variable exponent case, we  believe the  result is well known to experts (see \cite[Lemma 3.3]{EHL} for a very similar result, in fact our proof follows along the same lines), but we could not find a reference for it. Hence we shall give a proof in \ref{proof_poincare} for the sake of completeness.
\begin{theorem}
 \label{measure_density_poincare}
 Let $\sss \in \logh$ and let ${M}_4\geq 1$, $\varepsilon \in (0,1)$ be given constants. Define  $R_4 := \min \left\{\frac{1}{2{M}_4}, \frac1{\om_n^{1/n}}, \frac12 \right\}$.  For any $\phi \in W^{1,\pp}(B_{2r})$ with $2r < R_4$ satisfying 
 \begin{equation}\label{measure_density_zero}|\{ N(\phi)\}| := |\{ x \in B_r : \phi(x) =0\}|> \varepsilon |B_r|\end{equation} and 
 \begin{equation}\label{assumption_on_phi_capacity} \int_{B_{2r}} |\nabla \phi (x)|^{s(x)} \ dx + 1 \leq {M}_4, \end{equation} 
 then there holds
 \[
  \int_{B_r} \lbr \frac{|\phi|}{\diam(B_r)}\rbr^{s(x)} \ dx \apprle  \int_{B_{2r}} |\nabla \phi(x)|^{s(x)} \ dx + |B_{r}|.
 \]
\end{theorem}

We now recall the analogue of Hedberg's lemma with truncated maximal function. The estimate using the non-truncated maximal function is due to Semmes (for a proof, see \cite[Section 5]{HK}) and  the proof of the theorem using truncated maximal function follows along similar lines. 

\begin{theorem}
\label{hedber_inequality}
 Let $1<q<\infty$ and $\tOm$ be a bounded domain such that $\tOm^c$ satisfies a uniform measure density condition, i.e.,  there exists $\ve \in (0,1)$ such that for every $r>0$ and $x \in \pa \tOm$, there holds $|\tOm^c \cap B_r(x)| \geq \ve |B_r(x)|$. Fix any $r \geq \diam (\tOm)$ and let  $u \in W_0^{1,q}(\tOm)$ be given, then after extending $u$ to be zero outside $\tOm$, the following estimate holds for a.e $x,y \in \tOm$,
 \[
  |u(x) - u(y) | \leq C(q,n,\ve) |x-y| \lbr \mm_{<r} (|\nabla u|^q)^{\frac1{q}}(x) + \mm_{<r} (|\nabla u|^q)^{\frac1{q}}(y)\rbr.
 \]
\end{theorem}
We need the assumption on the measure density of $\tOm^c$ above to ensure we can apply the constant exponent version of the Poincar\'e inequality analogous to Theorem \ref{measure_density_poincare}. 

\subsection{Two important Lemmas}
%
%
%
%

The first Lemma is the well known Gehring's lemma (see for example \cite[Propositon 1.1]{Giaquinta} for the details):
\begin{lemma}[Gehring's Lemma]
 \label{gehring}
 Let $1<s<\infty$ and  $f,g \in L^s(\tOm)$ be nonnegative functions. Suppose that there are constants $c_g>0$ and $0\leq\theta <1$ such that the following holds
 \[
  \fint_B f^s \ dx \leq c_g \lbr \fint_{2B} f \ dx \rbr^s + \theta \fint_{2B} f^s \ dx + \fint_{2B} g^s \ dx
 \]
for all balls $B$ such that $8B \subset \tOm$. Then for any $\mt >0$,  there are constants $r= r(n,s,c_g,\theta)>s$ and $C = C(n,s,c_g,\theta,r,\mt)>0$ such that the following holds:
\[
  \fint_B f^r \ dx  \leq C \lbr[\{] \lbr \fint_{2B} f^{\mt} \ dx \rbr^{\frac{r}{\mt}} +  \fint_{2B} g^r \ dx\rbr[\}].
\]

\end{lemma}

The second lemma is an estimate in $L\log L$-space which can be found in \cite{AM} and references therein:
\begin{lemma}
 \label{llogl}
 Let $\be >0$ and $s>1$, then for any $f \in L^s(\tOm)$, we have
 \[
  \fint_{\tOm} |f| \lbr[[] \log \lbr e + \frac{|f|}{\avg{|f|}{\tOm}}\rbr \rbr[]]^{\be}\ dx \leq C(n,s,\be) \lbr \fint_{\tOm} |f|^s \ dx \rbr^{\frac{1}{s}}. 
 \]
\end{lemma}

\subsection{Lipschitz-truncation}
In this subsection, we shall give the construction of the Lipschitz function. In this regard, we  need the following lemma whose proof can be found in \cite[Lemma 3.1]{AP1}.
\begin{lemma}\label{lipschitz_extension_one}
\renewcommand{\tv}{\check{v}}
 Let $\ga> 0$, $\sig \in (0,1]$ and $S_0>0$ be given and suppose that $\tilde{\Om}$ is a $(\ga, \sigma, S_0)$-quasiconvex domain. For $\sss \in \logh$, let $\ov \in W_0^{1,\sss}(\tilde{\Om})$, and let  $\de \in (0,1/4)$ be given such that $s^-_{\tilde{\Om}}-2\de >1$ and let $1<q < s^-_{\tilde{\Om}}-2\de$ be any exponent. Let $M_5 \geq 1$ be given and define $$R_5 := \min \left\{\frac{1}{2M_5}, \frac1{\om_n^{1/n}},\frac12, \frac{S_0}{2}\right\}.$$

 After extending $\ov$ by zero outside $\tilde{\Om}$, for any $2r \leq R_5$, we define the truncated function $\check{v}$ by 
 \begin{itemize}[leftmargin=*]
 \item If $B_{2r} \subset \tilde{\Om}$, then we set $\check{v} = \ov \phi$ for some cut off function $\phi \in C_c^{\infty}(B_{2r})$ with $\lsb{\chi}{B_{r}} \leq \phi \leq \lsb{\chi}{B_{2r}}$.
 \item If $B_{2r} \cap \partial \tilde{\Om} \neq \emptyset$, then we set $ \check{v}=\ov$. 
\end{itemize}
Let $\check{v}$ satisfy $\integral{\Omt}{|\nabla \check{v}|^{s(x)}}{dx} \le M_5$, and we define the following function:
 \[
  g(x) := \max \left\{ \mm_{<4r} (|\nabla {\tv}|^q)^{1/q}(x), \frac{|{\tv}(x)|}{d(x,\partial B_{2r})} \right\}. 
 \]

 Then the following holds:

\begin{itemize}
 \item $g(x) \approx \mm_{<4r} (|\nabla {\tv}|^q)^{1/q}(x)$ a.e $x \in \RR^n$.
 \item $ \int_{B_{2r}} g(x)^{s(x)-\de}\ dx \apprle \int_{B_{2r}} |\nabla {\tv}|^{s(x)-\de} \ dx + |B_r|$.
 \item $[ g^{-\de} ]_{A_{\frac{\sss}{q}}} \leq C(\slog,q,n) [ g^{-\de} ]_{A_{\frac{s^--\de}{q}}} \leq C(n,\slog,q).$
 \item The function $g^{-\de}$ is in the Muckenhoupt class $A_{\mathfrak{t}/q}$ for any $\mathfrak{t}$ such that $\mathfrak{t}-q>\de$ with $[g^{-\de}]_{A_{\mathfrak{t}/q}} \leq C = C(n,\mathfrak{t},q)$. In particular $g^{-\de} \in A_{\frac{\sss}{q}}$.
\end{itemize}
\end{lemma}
\begin{proof}
\renewcommand{\tv}{\check{v}}
Since $\Om$ is $(\ga, \sig, S_0)$-quasiconvex, we can use the results from \cite[Theorem 2]{Haj} to get  
\begin{equation*}\label{crucial_estimate}
g(x) \underset{C(n,s^-,q)}{\approx}   \mm_{<4r}( |\nabla \tv|^q)^{\frac1{q}}(x)  .
\end{equation*}
Now we can use Corollary \ref{unweighted_maximal_bound} to get the estimate
\[ \int_{B_{2r}} g(x)^{s(x)-\de}\ dx \leq C(n,\slog,q) \mgh{  \int_{B_{2r}} |\nabla {\tv}|^{s(x)-\de} \ dx+|B_r| }.\]
%
Before we prove the last claim, we first show that $g^{-\de} \in A_{\frac{s^--\de}{q}}$ as follows:
\begin{equation*}
 \begin{array}{ll}
  [g^{-\de}]_{A_{\frac{s^--\de}{q}}} & =: \sup_{B}\lbr \fint_B g^{-\de} \ dx \rbr \lbr \fint_B g^{\frac{\de q}{s^--\de-q}} \ dx \rbr^{\frac{s^--\de-q}{q}} \\
  &  \overset{\redlabel{4.22.a}{a}}{\apprle} \lbr \inf_{y\in B}  \mm_{<4r}( |\nabla \tv|^q)(y) \rbr^{\frac{-\de}{q}} \lbr \inf_{y\in B}  \mm_{<4r}( |\nabla \tv|^q)(y) \rbr^{\frac{\de}{q}} \\
  & \leq C(n,s^-,q).
 \end{array}
\end{equation*}
To obtain \redref{4.22.a}{a}, we  made use of \cite[Proposition 3.3]{Tor} concerning $A_1$ weights. 

Since $\frac{s^- -\de}{q} \leq \frac{\sss}{q}$, using Lemma \ref{weight_inclusion}, the following estimate holds:
\begin{equation*}
\label{crucial_bound}
 [ g^{-\de} ]_{A_{\frac{\sss}{q}}} \leq C(\slog,q) [ g^{-\de} ]_{A_{\frac{s^--\de}{q}}} \leq C(n,\slog,q).
\end{equation*}
Note that all the constants are independent of $\de$ since we assume $\de \leq 1/4$, which enables us to remove the dependence on $\de$.  This completes the proof of the Lemma. 
\end{proof}

We now present an extension lemma which can be found in \cite{Zhong} (see also \cite[Lemma 3.2]{AP1} for the details):
\begin{lemma}\label{lipschitz_extension_two}
 Let $\ga > 0$, $\sig >0$ and $S_0>0$ be given and let $\ov \in W_0^{1,s}(\tilde{\Om})$ for some $s> 1$. Suppose  $\tilde{\Om}$ is a $(\ga,\sigma,S_0)$-quasiconvex domain and $\la >0$ is any given constant. Extend $\ov$ by zero outside $\tilde{\Om}$ and set
 \[
  F_{\la}(\ov,\tilde{\Om}) : =  \left\{ x \in \tilde{\Om}: \mm_{<\diam{\tilde{\Om}}}(|\nabla \ov|^s)^{1/s}(x) \leq \la, \ |\ov(x)| \leq \la d(x,\pa \tilde{\Om}) \right\}.
 \]
Then there exists a $c\la$-Lipschitz function $\ov_{\la}$ defined on $\RR^n$ with $c = c(n)>1$ satisfying the following properties:
\begin{itemize}
 \item $\ov_{\la}(x) = \ov(x)$ and $\nabla \ov_{\la}(x) = \nabla \ov(x)$ for a.e $x \in F_{\la}$,
 \item $\ov_{\la}(x) = 0$ for every $x \in \tilde{\Om}^c$,
 \item $|\nabla \ov_{\la}(x)| \leq c(n) \la $ for a.e $x \in \RR^n$. 
\end{itemize}
\end{lemma}

\subsection{Some auxiliary results}
We will need the following sharp form of Young's inequality:
\begin{lemma}
 \label{young}
 Let $\tilde{\Om}$ be any bounded domain and let $\sss \in \logh$ and $q(\cdot)>0$ be any exponent satisfying $q(x) \leq s(x)$ for every $x \in \tilde{\Om}$. For $f \in L^{\sss}(\tilde{\Om})$, Young's inequality yields
 \begin{equation*}
  \label{young_one}
   \int_{\tilde{\Om}} |f(x)|^{q(x)} \ dx 
    \leq  \int_{\tilde{\Om}} |f(x)|^{s(x)}\ dx + |\tilde{\Om}|. 
 \end{equation*}
 
\end{lemma}

The next lemma is a reformulation of {\em Calder\'{o}n-Zygmund decomposition} in terms of  balls instead of cubes.
\begin{lemma}[\cite{JJW1}]
\label{calderon_zygmund}

  For any $\ga >0$, $\sig \in (0,1/4)$ and $S_0>0$, let $\Om$ be a $(\gamma, \sigma, S_0)$-quasiconvex domain.  Consider the subdomain $\OO_{r}(x_0) = \OO \cap B_{r}(x_0)$ with $r \in (0, S_0]$ and $x_0 \in \OO$. Let $C \subset D \subset \Om_{r}(x_0)$ be measurable sets and $0<\ep<1$ such that
 \begin{itemize}
  \item $|C| < \ep |B_{r}|$, and
  \item for all $x \in \Om$ and $\rho \in (0,2r]$, if $|C\cap B_{\rho}(x)| \geq \ep |B_{\rho}(x)|$, then$\ B_{\rho}(x) \cap \Omega_{r}(x_0) \subset D$.
 \end{itemize}
 Then we have the estimate \[ |C| \leq  \ep   \left( \frac{10}{\sigma} \right)^{n}|D|.\]
\end{lemma}

We end this section by introducing a well known lemma (see \cite[Lemma 4.1]{AP1} for details).
\begin{lemma}\label{distribution}
Let $f$ be a measurable function in a bounded open set $\tOm \subset \RR^n$ and  let $\lambda_0>0$ and $N>1$ be given constants,  then for any $0<q<\infty$ there holds
\begin{equation*}
f \in L^q(\tOm) \quad \iff \quad S := \sum_{k\ge 1} N^{qk} \left| \mgh{x \in \tOm : |f(x)| > N^{k}\lambda_0} \right| < \infty
\end{equation*}
with the estimate
\begin{equation*}
c^{-1}\lambda_0^q S \le \integral{\tOm}{|f|^q}{dx} \le c\lambda_0^q \gh{|\tOm|+S},
\end{equation*}
where the constant $c=c(N,q)>0$.
\end{lemma}

\section{A priori estimates}\label{aprioriestimates}

First let us prove some general a priori estimates.
\subsection{General a priori estimates}
Let us first recall the following energy estimate:
  \begin{lemma}\label{energy_homogeneous} Let $\tOm$ be any bounded domain and suppose that the nonlinearity $\aa(\cdot,\cdot)$ satisfies \eqref{bounded} and \eqref{ellipticity}. For any $\tu \in W^{1,\pp}(\tilde{\Om})$ and $\tbg \in L^{\pp}(\tilde{\Om},\RR^n)$, let $\tw \in W^{1,\pp}(\tilde{\Om})$ be the unique solution that  solves the following equation
 \begin{equation}\label{first_approx}
 \left\{\begin{array}{l}
  \dv \aa(x,\nabla \tw) = \dv(|\tbg|^{p(x)-2}\tbg) \qquad \text{in} \ \tilde{\Om}, \\
  \tw \in \tu + W_0^{1,\pp}(\tilde{\Om}).
  \end{array}\right.
 \end{equation}
 Then the following energy estimate holds:
 \begin{equation*}
  \int_{\tilde{\Om}} |\nabla w|^{p(x)} \ dx \leq C \lbr \int_{\tilde{\Om}} |\nabla u|^{p(x)} \ dx +\int_{\tilde{\Om}} |\tbg|^{p(x)} \ dx+ |\tilde{\Om}|\rbr,
 \end{equation*}
 where the constant $C = C(\La_0,\La_1,{\plog},n)$. 
 \end{lemma}

Next we prove a crucial difference estimates below the natural exponent:
\begin{theorem}[A priori estimates below the natural exponent]
\renewcommand{\omt}{\tilde{\Om}}
\label{boundary_below_exponent}
Let $\ga>0$, $\sig\in (0,1/4)$ and $S_0>0$ be given and suppose that $\tOm$ is a $(\ga,\sig,S_0)$-quasiconvex domain. Let $\tu \in W^{1,\pp}(\omt)$ and $\tbg \in L^{\pp}(\tilde{\Om},\RR^n)$ be given. For any  $M_6 \geq 1$,  define 
\begin{equation*}\label{R_1_bnd}
R_6 := \min \left\{\frac{1}{2M_6}, \frac{1}{\om_n^{1/n}},\frac12,\frac{S_0}{2} \right\}.
 \end{equation*}
Let $\omt_r := \omt \cap B_r$ be any region with $2r <R_6$ and consider the unique solution $w \in W^{1,p(\cdot)}(\omt)$ of the following problem:
 \begin{equation}\label{wapprox_boundary}\left\{
 \begin{array}{l}
  \dv \aa(x,\nabla \tw) = \dv (|\tbg|^{\pp-2}\tbg) \quad \text{in} \ \omt_r, \\
  \tw \in \tu + W_0^{1,\pp}(\omt_r).
  \end{array}\right.
 \end{equation}
 Suppose that the following bound holds:
 \begin{equation*}
  \label{hypothesis_below_exponent}
  \int_{\RR^n} |\nabla \tu - \nabla \tw|^{p(x)} \ dx+1=  \int_{\omt_r} |\nabla \tw - \nabla \tu|^{p(x)} \ dx +1  \leq M_6, 
 \end{equation*}
then there exists a constant $\de_1= \de_1(\La_0,\La_1,{\plog},n) \in (0,1/4)$ such that for any $\de \in(0,\de_1)$, there holds
 \begin{equation*}\label{apriori_bound}
  \int_{\omt_r} |\nabla \tw|^{p(x)-\de} \ dx \leq C \lbr \int_{\omt_r} |\nabla \tu|^{p(x)-\de} \ dx + \int_{\omt_r} |\tbg|^{p(x)-\de} \ dx+  |\omt_r|\rbr,
 \end{equation*}
where the constant $C = C(\La_0,\La_1,{\plog},n)$.
\end{theorem}

\begin{proof}
\renewcommand{\omt}{\tilde{\Om}}

Let $2r < R_6$ be fixed, since $\Omt$ is a $(\ga,\sigma,S_0)$-quasiconvex domain, $\Omt_r$ satisfies the measure density estimate of Lemma \ref{measure_density_quasiconvex}. Let $1<q< p^-_{\omt_r}-2\de$ be any fixed exponent and  let us define
\begin{equation*}
 \label{bnd_g}
  g(x) := \max \lbr[\{] \mm_{<2r}(|\nabla \tw - \nabla \tu|^q)^{\frac{1}{q}}(x), \frac{|\tw(x) - \tu(x)|}{d(x,\partial \omt_r)} \rbr[\}].
 \end{equation*}
Note that we have extended $\tw-\tu = 0$ outside $\omt_r$. 

By the restriction on $R_6$, we can apply Lemma \ref{lipschitz_extension_one} with $M_5 = M_6$ and $\check{v} = \tw-\tu$ to get
\begin{gather}
g(x) \approx \mm_{<2r}(|\nabla \tw - \nabla \tu|^q)^{\frac{1}{q}}(x) \qquad \text{for a.e.} \ x\in \RR^n, \label{g_equiv_bnd} \\
\int_{\omt_r} g(x)^{{p(x) - \de}}\ dx \leq  C({\plog},n) \lbr \int_{\omt_r} |\nabla \tw - \nabla \tu|^{p(x) -\de} \ dx+ |\omt_r|\rbr\label{g_maximal_bound_bnd}.
\end{gather}
Define the set 
 \begin{equation*}\label{f_lam_bnd_diff}
F_{\la} := \{ x \in \Om_r : g(x) \leq \la \}.  
 \end{equation*}
 We  now apply Lemma \ref{lipschitz_extension_two} with $s=q$ and $\bar{v}=\tilde{w}-\tilde{u}$ to get a Lipschitz function $v_{\la} $ which is a valid test function for \eqref{wapprox_boundary}, which gives
 \[
  \begin{array}{l}
   \int_{\flam} \iprod{\aa(x,\nabla \tw) -\aa(x,\nabla \tu)+\aa(x,\nabla \tu)-|\tbg|^{p(x) - 2} \tbg}{\nabla v_{\la}} \ dx = \\ 
    \qquad \qquad \qquad =- \int_{\flam^c} \iprod{\aa(x,\nabla \tw)}{\nabla v_{\la}}\ dx + \int_{\flam^c} \iprod{|\tbg|^{p(x) - 2} \tbg}{\nabla \vlam}\ dx\\
    \qquad \qquad \qquad\apprle \la \int_{\flam^c} (\mu^2 + |\nabla \tw|^2)^{\frac{p(x) -1}{2}}\  dx + \la \int_{\flam^c} |\tbg|^{p(x) -1} \ dx.\\
  \end{array}
 \]
  Multiplying the above expression by $\la^{-1-\de}$ and integrating over $(0,\infty)$, we get
  $I_1 +I_2 +I_3\apprle I_4 +I_5,$
where  we have set
\begin{itemize}
 \item $I_1 := \int_0^{\infty} \la^{-1-\de}\int_{\flam} \iprod{\aa(x,\nabla \tw) -\aa(x,\nabla \tu)}{\nabla v_{\la}} \ dx\ d\la$.
 \item $I_2 := \int_0^{\infty} \la^{-1-\de}\int_{\flam} \iprod{\aa(x,\nabla \tu)}{\nabla v_{\la}} \ dx\ d\la$.
 \item $I_3 = - \int_0^{\infty} \la^{-1-\de} \int_{\flam} \iprod{|\tbg|^{p(x) - 2} \tbg}{\nabla \vlam}\ dx \ d\la $.
 \item $I_4 := \int_0^{\infty} \la^{-\de}\int_{\flam^c} |\nabla \tw|^{p(x) -1} + \mu^{p(x) - 1} \ dx \ d\la$.
 \item $I_5 := \int_0^{\infty} \la^{-\de}\int_{\flam^c} |\tbg|^{p(x) -1} \ dx \ d\la$.
\end{itemize}
Let us now estimate each of the above terms as follows:
 \begin{description}[leftmargin=*]
  \item[Estimate for $I_1$:] Applying Fubini's theorem and \eqref{monotonicity}, we get 
  \begin{equation}\label{eq7.43}
   \begin{array}{ll}
    I_1 & = \frac{1}{\de} \int_{\Omt_r} g(x)^{-\de}  \iprod{\aa(x,\nabla \tw) -\aa(x,\nabla \tu)}{\nabla \tw - \nabla \tu} \ dx \\
    & \apprge \frac{1}{\de} \int_{\Omt_r} (\mu^2+|\nabla \tu|^2 + |\nabla \tw|^2)^{\frac{p(x)-2}{2}} |\nabla \tw - \nabla \tu|^2 g(x)^{-\de} \ dx.
   \end{array}
  \end{equation}
%
  In order to estimate \eqref{eq7.43}, we need to consider the case if $p(x)\geq 2$ and $p(x)<2$, hence let us  split $\Omt_r = \Omt_1 \cup \Omt_2$, where $$\Omt_1 := \{ x \in \Omt_r : p(x) \geq 2\}\qquad \text{and} \qquad \Omt_2:=\{ x \in \Omt_r : p(x) < 2\}.$$ 
  \begin{description}
   \item[Case $p(x) \geq 2$:] In this case, we can directly apply Young's inequality to get
%
\begin{equation}\label{eq7.45}
\begin{array}{@{}r@{}c@{}l@{}@{}}
 \int_{\Omt_1} |\nabla \tw - \nabla \tu|^{p(x) -\de} \ dx & \apprle& C(\ep_1) \int_{\Omt_1} |\nabla \tw- \nabla \tu|^{p(x)} g(x)^{-\de}\ dx  + \ep_1 \int_{\Omt_1} g(x)^{p(x) -\de} \ dx\\
 & \overset{\eqref{g_maximal_bound_bnd}}{\apprle}& C(\ep_1) \int_{\Omt_1} |\nabla \tw- \nabla \tu|^{p(x)} g(x)^{-\de}\ dx  +  \\
 & &\qquad \qquad \qquad + \ep_1 \int_{\Omt_r} |\nabla \tw - \nabla \tu|^{p(x) -\de}\ dx +\ep_1|\Omt_r|.
 \end{array}
\end{equation}
Thus combining \eqref{eq7.45} into \eqref{eq7.43}, we get
\begin{equation}\label{eq7.46}
 \begin{array}{ll}
  \int_{\Omt_1} |\nabla \tw - \nabla \tu|^{p(x) -\de}\ dx & \apprle C(\ep_1)  \de I_1 +  \ep_1 \int_{\Omt_r} |\nabla \tw - \nabla \tu|^{p(x) -\de}\ dx + \ep_1|\Omt_r|.
 \end{array}
\end{equation}
   \item[Case $p(x) < 2$:] In this case, using Young's inequality, we obtain
   \begin{equation}\label{eq7.47}
    \begin{array}{ll}
     |\nabla \tw - \nabla \tu|^{p(x) -\de} 
     & \apprle C(\ep_1)(\mu^2+|\nabla \tu|^2 + |\nabla \tw|^2)^{\frac{p(x) -2}{2}} |\nabla \tw - \nabla \tu|^2 g(x)^{-\de}\  + \\
     & \qquad  + \ep_1(\mu^2+|\nabla \tu|^2 + |\nabla \tw|^2 )^{\frac{p(x)-\de}{2}} +\ep_1g(x)^{p(x)-\de} .
    \end{array}
   \end{equation}
Integrating \eqref{eq7.47} over $\Omt_2$ and using \eqref{g_maximal_bound_bnd} and \eqref{eq7.43}, we get
\begin{equation}\label{eq7.48}
 \begin{array}{ll}
  \int_{\Omt_2} |\nabla \tw - \nabla \tu|^{p(x) -\de} \ dx & 
   \apprle C(\ep_1)\de I_1 + \ep_1\int_{\Omt_r} |\nabla \tw - \nabla \tu|^{p(x)-\de}  \ dx + \ep_1|\Omt_r|\\
  & \quad + C(\ep_1)\int_{\Omt_r} |\nabla \tu|^{p(x)-\de}\ dx.
 \end{array}
\end{equation}
  \end{description}
  
  Thus we can combine  \eqref{eq7.46} when $p(x) \geq 2$  and \eqref{eq7.48} when $p(x) <2$ to get
  \begin{equation}
   \label{estimate_I_1}
   \begin{array}{ll}
  \int_{\Omt_r} |\nabla \tw - \nabla \tu|^{p(x) -\de} \ dx & 
   \apprle C(\ep_1)\de I_1 + \ep_1\int_{\Omt_r} |\nabla \tw - \nabla \tu|^{p(x)-\de}  \ dx + \ep_1|\Omt_r|\\
  & \qquad + C(\ep_1)\int_{\Omt_r} |\nabla \tu|^{p(x)-\de}\ dx.
 \end{array}
  \end{equation}

\item[Estimate for $I_2$:] Making use of Fubini's theorem, Young's inequality, \eqref{bounded} and \eqref{g_equiv_bnd}, we get
\begin{equation}\label{eq7.49}
 \begin{array}{ll}
  I_2 
  & \apprle \frac{1}{\de} \int_{\Omt_r} g(x)^{-\de} (\mu^2+|\nabla \tu|^2)^{\frac{p(x) -1}{2}} |\nabla \tw - \nabla \tu|\ dx \\
  & \apprle \frac{C(\ep_2)}{\de} \int_{\Omt_r}  |\nabla \tu|^{p(x) -\de} + \frac{C(\ep_2)|\Omt_r|}{\de}+ \frac{\ep_2}{\de} \int_{\Om_r} |\nabla \tw - \nabla \tu|^{p(x)-\de}\ dx.
 \end{array}
\end{equation}

\item[Estimate for $I_3$:] Similar to $I_2$, after using Fubini's theorem, Young's inequality  and \eqref{g_equiv_bnd}, we get:
\begin{equation}
 \label{eq7.49_1}
\begin{array}{ll}
 I_3 
  & \apprle \frac{C(\ep_3)}{\de} \int_{\Omt_r}  | \tbg|^{p(x) -\de} \ dx+ \frac{\ep_3}{\de} \int_{\Om_r} |\nabla \tw - \nabla \tu|^{p(x)-\de}\ dx.
  \end{array}
\end{equation}

\item[Estimate for $I_4$ and $I_5$:] Again by using Fubini's theorem, Young's inequality and \eqref{g_maximal_bound_bnd}, we get
\begin{equation}\label{eq7.50}
 \begin{array}{ll}
  I_4+I_5 
  & \apprle \int_{\Omt_r} |\nabla \tw-\nabla \tu|^{p(x)-\de} \ dx + \int_{\Omt_r} |\nabla \tu|^{p(x)-\de} \ dx +\int_{\Omt_r} |\tbg|^{p(x)-\de} \ dx+ |\Omt_r|.
 \end{array}
\end{equation}

\end{description}

 Combining \eqref{estimate_I_1}--\eqref{eq7.50} followed by using the trivial bound $I_1 \leq |I_2| + |I_3| + I_4 + I_5$, we obtain 
\begin{equation*}
 \label{combined}
 \begin{array}{ll}
   \int_{\Omt_r} |\nabla \tw - \nabla \tu|^{p(x) -\de} \ dx & 
    \apprle  \lbr[[] \ep_1 +  C(\ep_1) (\ep_2+\ep_3+\de)\rbr[]] \int_{\Omt_r} |\nabla \tw - \nabla \tu|^{p(x)-\de}  \ dx +\\
    & \quad + C(\ep_1,\ep_2,\ep_3) \lbr \int_{\Omt_r} |\nabla \tu|^{p(x)-\de}\ dx + \int_{\Omt_r} |\tbg|^{p(x)-\de}\ dx + |\Omt_r| \rbr.
  \end{array}
\end{equation*}
Choosing $\ep_1$ small followed by $\ep_2,\ep_3$ and $\de_1$, we see that for all $\de \in (0,\de_1)$, the following estimate holds:
\begin{equation*}
 \label{combined_1}
   \int_{\Omt_r} |\nabla \tw - \nabla \tu|^{p(x) -\de} \ dx \apprle  \int_{\Omt_r} |\nabla \tu|^{p(x)-\de}\ dx + \int_{\Omt_r} |\tbg|^{p(x)-\de}\ dx + |\Omt_r|.
\end{equation*}
A simple application of triangle inequality then proves the Theorem. 
\end{proof}

\subsection{Fixing some universal constants}

Let us define the following constant:
\begin{equation}
 \label{size_date}
 M_0 := \int_{\Om} \bgh{|\bff|^{p(x)} +1} \ dx + 1.
\end{equation}

Now by applying Lemma \ref{energy_homogeneous} to \eqref{basic_pde}, we see that there is a constant $C = C(\plog,\La_0,\La_1,n)>0$ such that the following estimate holds:
\begin{equation}
 \label{size_solution}
 \int_{\Om} |\nabla u|^{p(x)} \ dx +1 \leq C \lbr \int_{\Om} \bgh{|\bff|^{p(x)} +1}\ dx + 1 \rbr =: M^u.
\end{equation}
Note that $M^u = M^u(M_0,\plog,\La_0,\La_1,n)$.

For any $r> 0$, let $\Om_r := B_r \cap \Om$ and consider the following first approximation to \eqref{basic_pde}:
\begin{equation}\label{w_approx_difference}\left\{
 \begin{array}{ll}
  \dv \aa(x,\nabla w) = 0 & \text{in} \ \Om_r, \\
  w \in u + W_0^{1,\pp}(\Om_r).
  \end{array}\right.
 \end{equation}
 Applying Lemma \ref{energy_homogeneous} to \eqref{w_approx_difference} and then making use of Lemma \ref{young}, we obtain a constant $C = C(\plog,\La_0,\La_1,n)>0$ such that the following estimates hold for any $q(x) \leq p(x)$:
 \begin{gather*}
  \int_{\Om_r} |\nabla w|^{q(x)} \ dx +1 \leq C  \lbr \int_{\Om} |\nabla u|^{p(x)} \ dx + 1 \rbr := M^w_1,    \\
    \int_{\Om_r} |\nabla w - \nabla u|^{q(x)} \ dx +1 \leq C  \lbr \int_{\Om} |\nabla u|^{p(x)} \ dx + 1 \rbr := M^w_2.
 \end{gather*}
%
 We now define the constant 
 \begin{equation}
    \label{size_w}
  M^w := M^w(M_0,\plog,\La_0,\La_1,n)  = \max \{ M^w_1,M^w_2\}.
 \end{equation}

 \subsection{Difference estimate below the natural exponent for the first approximation}
 
 In this subsection, we shall prove crucial difference estimates between solutions of \eqref{basic_pde} and \eqref{w_approx_difference} below the natural exponent.
 \begin{theorem}\label{boundary_difference}
  Let $\sig \in (0,1/4)$ and $S_0>0$ be given and let $\Om$ be a $(\ga,\sigma,S_0)$-quasiconvex domain for some $\ga >0$ to be chosen later. Suppose that $u \in W^{1,\pp}_0(\Om)$ is the unique solution of  \eqref{basic_pde}. Under the bounds \eqref{size_date} and \eqref{size_solution}, consider the unique solution  $w \in W^{1,\pp}(\Om_r)$ solving \eqref{w_approx_difference} in $\Om_r$ for some $r \leq \frac{R_7}{2}$ 
  where 
  \begin{equation*}
   \label{def_R_1}
   R_7 = \min \left\{\frac{1}{2M^u},\frac{1}{2M^w},\frac{1}{2M_0}, \frac{1}{\om_n^{1/n}},\frac12,\frac{S_0}{2}\right\}.
  \end{equation*}
 For any $0<\varepsilon<1$, there exist positive constants $\de_2=\de_2(\La_0,\La_1,{\plog},n,\ep)\in(0,1)$ and $\ga_2 = \ga_2(\La_0,\La_1,{\plog},n,\ep) \in (0,1)$ such that for all $\de \in (0,\de_2)$ and $\ga \in (0,\ga_2)$, we have the following result: Suppose for a given $2r < R_7$,  the following bounds hold:
 \begin{equation*}
  \label{bound_on_data}
  \fint_{\Om_{r}} |\nabla u|^{p(x) - \de} \ dx \leq \la \qquad \text{and} \qquad \fint_{\Om_{r}} |\bff|^{p(x) - \de} \ dx \leq \ga \la, 
 \end{equation*}
 then there exists a constant $C = C(\La_0,\La_1,{\plog},n)$ such that the following conclusions hold:
\begin{equation}\label{diff_conclusion_bnd}
 \fint_{\Om_{r}} |\nabla w|^{p(x) - \de} \ dx \leq C \la \qquad \text{and} \qquad \fint_{\Om_{r}} |\nabla w-\nabla u|^{p(x) - \de} \ dx \leq \varepsilon \la.
\end{equation}
 \end{theorem}

 \begin{proof}
\renewcommand{\omt}{\Omega}  
Note that the first conclusion in \eqref{diff_conclusion_bnd} follows by applying Theorem \ref{boundary_below_exponent} to \eqref{w_approx_difference}.  Hence we shall only need to  prove the second conclusion of \eqref{diff_conclusion_bnd}.

Since $\Om$ is a $(\ga,\sigma,S_0)$-quasiconvex domain, we have that $\Om_r$ satisfies the measure density estimate of Lemma \ref{measure_density_quasiconvex}. Let $1<q\leq p^-_{\Om_r}-2\de$ be any fixed exponent and  consider the following function:
\begin{equation*}
 \label{bnd_g_2}
  g(x) := \max \lbr[\{] \mm_{<2r}(|\nabla w - \nabla u|^q)^{\frac{1}{q}}(x), \frac{|w(x) - u(x)|}{d(x,\partial \Om_r)} \rbr[\}].
 \end{equation*}
Note here that we have extended $w-u = 0$ outside $\omt_r$.  By the restriction $2r \leq R_7$, we can apply Lemma \ref{lipschitz_extension_one} with $M_5 = \max\{ 2M^u,2M^w\}$ and $\ov = w-u$ to get
\begin{gather*}
g(x) \approx \mm_{<2r}(|\nabla w - \nabla u|^q)^{\frac{1}{q}}(x) \quad \text{on} \ \RR^n, \label{g_equiv_bnd_2} \\
\int_{\omt_r} g(x)^{{p(x) - \de}}\ dx \leq  C({\plog},n) \lbr \int_{\omt_r} |\nabla w - \nabla u|^{p(x) -\de} \ dx+ |\omt_r|\rbr\label{g_maximal_bound_bnd_2}.
\end{gather*}

Define the set $F_{\la} := \{ x \in \Om_R : g(x) \leq \la \}$, 
%
we can now apply Lemma \ref{lipschitz_extension_two} to get a Lipschitz function $v_{\la}$ which is a valid test function for \eqref{w_approx_difference}.

We can now proceed exactly as in the proof of Theorem \ref{boundary_below_exponent} to get for any $\ve \in (0,1)$,
\begin{equation*}
 \int_{\Om_r} |\nabla u - \nabla w|^{p(x)-\de}  \leq C(\La_0,\La_1,\plog,n)  \int_{\Om_r} |\bff|^{p(x)-\de} \ dx + \varepsilon |\Om_r| + \varepsilon \int_{\Om_r} |\nabla u|^{p(x) - \de} \ dx.
\end{equation*}
Now we can choose $\fint_{\Om_R} |\bff|^{p(x) - \de} \ dx \leq \ga_2 \lambda$ such that the second estimate in \eqref{diff_conclusion_bnd} holds for all $\ga \in (0,\ga_2)$.  This completes the proof of the Theorem. 
 \end{proof}

 
 \subsection{Boundary higher integrability below the natural exponent}
 
 
 In this subsection, we shall prove a new higher integrability result. To better highlight the result, let us recall the following local versions of \eqref{basic_pde}. Let $x_0 \in \overline{\Om}$ be a fixed point of reference, then in the interior case $\Om_r(x_0) = B_r(x_0) \subset \Om$, the equation becomes
 \begin{equation}
 \label{interior_equation}
 \dv \aa(x,\nabla u) = \dv (|\bff|^{p(x) -2} \bff) \quad \text{in} \ B_r(x_0), 
\end{equation}
and in  the boundary case  $B_r(x_0) \nsubseteq \Om$ is given by:
\begin{equation}
 \label{boundary_equation}
 \left\{\begin{array}{rcll}
 \dv \aa(x,\nabla u) &=& \dv (|\bff|^{p(x) -2} \bff) &  \quad \text{in} \ \Om_r(x_0),     \\
 u &=& 0 & \quad \text{on} \ \pa_w \Om_r(x_0).
 \end{array}\right.
\end{equation}

\begin{theorem}
\renewcommand{\omt}{\Om_{2\tr}}
\renewcommand{\omth}{\Om_{\tr}}
\label{higher_integrability_full}
 There exist $R_8 = R_8(\plog,\La_0,\La_1,n,M_0,S_0)>0$, $\de_3 = \de_3(n,\La_0,\La_1,\plog)$ and  $\sigma_1= \sigma_1(\plog,\La_0,\La_1,n)$ such that the following holds for any $2r<R_8$,  $\de \in (0,\de_3)$ and $\bar{\sigma} \in (0,\sigma_1)$:  if $u \in W_0^{1,\pp}(\Om)$  and $\bff \in L^{\pp}(\Om)$ solve \eqref{interior_equation} or \eqref{boundary_equation}, after extending  $u = 0 $  and $\bff = 0$ on $\Om^c$,  there holds
 \[
  \fint_{B_{\tr}} |\nabla u|^{(p(x)-\de)(1+\bar{\sigma})} \ dx \apprle \lbr \fint_{B_{2\tr}} |\nabla u|^{p(x)-\de} \ dx \rbr^{1+\bar{\sigma}} + \fint_{B_{2\tr}} |\bff|^{(p(x)-\de)(1+\bar{\sigma})}\ dx + 1
 \]
for all $2\tr < r$ and $B_{2\tr} \subset B_r(x_0)$.
%
%
\end{theorem}
\begin{proof}
\renewcommand{\pu}{p^+_{B_{2\tr}}}
\renewcommand{\pd}{p^-_{B_{2\tr}}}
\renewcommand{\omt}{\Om_{2\tr}}
\renewcommand{\omth}{\Om_{\tr}}
We shall only prove the boundary higher integrability.  For the interior higher integrability, the only modification we need to do is in the choice of $\tu$ in \eqref{def_u_tilde} by taking $\tu:= \phi (u - \avg{u}{\btr})$. 

Let $B_{2\tr} \subset B_r(x_0)$ be any ball such that $B_{2\tr} \cap \Om^c \neq \emptyset$.  Let $\phi \in C_c^{\infty}(B_{2\tr})$ be a standard cut off function such that $\lsb{\chi}{B_{2\tr}}\geq \phi \geq \lsb{\chi}{B_{\tr}}$ and $|\nabla \phi| \le \frac{2}{\tr}$. After extending $u =0$ on $\Om^c$, define 
\begin{equation}
\label{def_u_tilde}
  \tu := \phi u.
\end{equation}
For any fixed $q \in (1, p^-  - 2\de)$, we define the following function:
\begin{equation*}
 \label{def_g}
 g(x) := \max \left\{ \mm_{<2\tr} (|\nabla \tu|^q + \mu^q)^{\frac1{q}}(x), \frac{|\tu(x)|}{d(x,\pa \Om_{2\tr})}\right\}.
\end{equation*}

For any $\qq \leq \pp$, there holds
\begin{equation*}
 \label{bound_tu}
 \begin{array}{ll}
  \int_{\btr} |\nabla \tu|^{q(x)} \ dx+1 &\overset{\redlabel{5.39.a}{a}}{\apprle} \int_{\btr} |\nabla \tu|^{p(x)} \ dx+1 \apprle \int_{\btr} |\nabla u|^{p(x)} \ dx + \int_{\btr} \left| \frac{u}{\tr}\right|^{p(x)} dx+1\\
 & \overset{\redlabel{5.39.b}{b}}{\leq} C(n,\pp,M^u)  =: M^{\tu}.
 \end{array}
\end{equation*}
To obtain \redref{5.39.a}{a}, we used Lemma \ref{young} and to obtain \redref{5.39.b}{b}, we used Theorem \ref{measure_density_poincare}.

From Lemma \ref{lipschitz_extension_one}, we see that $[g^{-\de}]_{A_{\frac{\pp}{q}}} \leq C_0(n,\plog,q)$ and hence we can apply Theorem \ref{integral-px-maximal-weighted} with ${M}_1 = C_0(n,\plog,q)$ and ${M}_2 = M^{\tu}$ to obtain an $R_1= R_1(M_1,M_2,\plog)$ such that for any $2r \leq R_1$ and  for any $h \in L_{g^{-\de}}^{\frac{\pp}{q}}(B_{2r})$ with $\int_{B_{2r}} |h|^{\frac{p(x)}{q}} g^{-\de} \ dx + 1 \leq M^{\tu}$, there holds:
\begin{equation}
 \label{g_integral_bound}
 \int_{B_r} \mm_{<r} (|h|)^{\frac{p(x)}{q}} \ g^{-\de} \ dx  \apprle \int_{B_{2r}} |h|^{\frac{p(x)}{q}} g^{-\de}\ dx + 1.
\end{equation}
Further choose $\tilde{R}_1$ such that 
\begin{equation}
\label{def_ratio}
 \rho(2\tilde{R}_1) \leq  \sqrt{\frac{n+1}{n}}-1<1.
\end{equation}
Now choose $R_8 \leq \min \left\{ \frac{R_1}{2},\frac{\tilde{R}_1}{2}, \frac{1}{2 M^{\tu}}, \frac{1}{\om_n^{1/n}}\right\}$. With this choice of $R_8$, we have from Lemma \ref{lipschitz_extension_one} that 
\begin{gather}
g(x) \approx \mm_{<2\tr}(|\nabla \tu|^q+\mu^q)^{\frac{1}{q}}(x) \quad \text{on} \ \RR^n, \label{g_equiv_bnd_3} \\
\int_{B_{2\tr}} g(x)^{{p(x) - \de}}\ dx \leq  C({\plog},n) \lbr \int_{\btr} |\nabla \tu|^{p(x) -\de} \ dx+ |\btr| \rbr \label{g_maximal_bound_bnd_3}.
\end{gather}
We can now apply Lemma \ref{lipschitz_extension_two} to get a Lipschitz test function $\vlam \in W_0^{1,\infty}(\omt)$. Using this in \eqref{boundary_equation} and then  multiplying by $\la^{-1-\de}$ and integrating over $(0,\infty)$ with respect to $\la$, for $\flam := \{ x \in \omt: g(x) \leq \la \}$, we get
\[
 I_1 -I_2 = I_3 - I_4,
\]
where 
\begin{itemize}[leftmargin=*]
\begin{minipage}{.5\textwidth}
\item $I_1 := \int_0^{\infty} \la^{-1-\de} \int_{\flam} \iprod{\aa(x,\nabla u)}{\nabla \vlam}\ dx d \la$.
\item $I_2 := \int_0^{\infty} \la^{-1-\de} \int_{\flam} \iprod{|\bff|^{p(x)-2} \bff}{\nabla \vlam}\ dx d \la$.
\end{minipage}
\begin{minipage}{.5\textwidth}
   \item $I_3 := \int_0^{\infty} \la^{-1-\de} \int_{\flam^c} \iprod{|\bff|^{p(x)-2} \bff}{\nabla \vlam}\ dx d \la$.
    \item $I_4 := \int_0^{\infty} \la^{-1-\de} \int_{\flam^c} \iprod{\aa(x,\nabla u)}{\nabla \vlam}\ dx d \la$.
\end{minipage}
\end{itemize}

Let us now estimate each of the terms as follows:

\begin{description}[leftmargin=*]
 \item[Estimate for $I_1$:] Applying Fubini's theorem, we get
 \begin{equation*}
  I_1  = \frac{1}{\de} \int_{\btr} g^{-\de} \iprod{\aa(x,\nabla u)}{\nabla \tu}\ dx = \frac{1}{\de} \lbr \int_{\btrh} + \int_{D_1} + \int_{D_2} \rbr g^{-\de} \iprod{\aa(x,\nabla u)}{\nabla \tu}\ dx,
 \end{equation*}
where we have set 
\begin{gather}
 D_1 := \left\{ x \in \btr\setminus \btrh  : \mm_{<2\tr} (| \nabla \tu|^q + \mu^q)^{\frac1{q}} (x) \leq \de\mm_{<2\tr} \lbr (| \nabla u|^q +\mu^q)\chi_{\btr}\rbr^{\frac1{q}} (x) \right\} \label{def_D_1},\\
 D_2 := \btr \setminus ( D_1 \cup \btrh)\label{def_D_2}.
\end{gather}

\begin{description}[leftmargin=*]
 \item[Estimate on $\btrh$:] Since $\phi \equiv 1$ on $\btrh$, we see that 
 \begin{equation}\label{eq8.12}
 \begin{aligned}
  \int_{\btrh} g^{-\de} \iprod{\aa(x,\nabla u)}{\nabla \tu}\ dx & \apprge \int_{\btrh} g^{-\de} (\mu^2 + |\nabla u|^2)^{\frac{p(x)-2}{2}} |\nabla u|^2 \ dx \\
  & \apprge \left\{ \begin{array}{lcl}
  \int_{\btrh} g^{-\de}  |\nabla u|^{p(x)} \ dx &\text{if} &\mu=0, \\
  \int_{\btrh} g^{-\de}  |\nabla u|^{p(x)} \ dx - \frac{|B_{\tr}|}{\mu^\delta} &\text{if} &\mu \neq 0.
  \end{array}\right.
  \end{aligned}
 \end{equation}
We only prove this theorem for the case $\mu \neq 0$ since the case $\mu = 0$ can be obtained in the same way. We now apply \eqref{g_integral_bound} to obtain the estimate: 
\begin{equation*}\label{eq8.13}
 |\btrh| + \int_{\btrh} g^{-\de} |\nabla u|^{p(x)} \ dx \apprge \int_{\btrh} g^{-\de} \mm_{<2\tr} (|\nabla u|^q \chi_{\btrh} )^{\frac{p(x)}{q}} \ dx .
\end{equation*}
By definition of the maximal function, we have 
\begin{equation*}\label{g_bound_lemma}
 g(x) \leq C_1 \mm_{<2\tr} (|\nabla u|^q \chi_{\btrh})^{1/q}(x) + C_2 \lbr \fint_{\btr} |\nabla u|^q \ dx \rbr^{1/q} + C_3 \mu \quad \text{for all} \ x \in B_{\tr/2}.
\end{equation*}
Define the set 
\begin{equation*}
 \label{def_G}
 G:= \left\{ x \in B_{\tr/2}: C_1  \mm_{<2\tr} (|\nabla u|^q \chi_{\btrh})^{1/q}(x)  \geq C_2 \lbr \fint_{\btr} |\nabla u|^q \ dx \rbr^{1/q} + C_3 \mu \right\}.
\end{equation*}
Thus for $x \in G$, we see that 
 $g(x) \leq 2 C_1 \mm_{<2\tr} (|\nabla u|^q \chi_{\btrh})^{1/q}(x).$
Using this, we get
\begin{equation}
\label{8.16}
 \begin{aligned}
  \int_{\btrh} g^{-\de} \mm_{<2\tr} (|\nabla u|^q \chi_{\btrh} )^{\frac{p(x)}{q}} \ dx  & \apprge \int_{G} g^{-\de} \mm_{<2\tr} (|\nabla u|^q \chi_{\btrh} )^{\frac{p(x)}{q}} \ dx \\
  & \apprge \int_G  \mm_{<2\tr} (|\nabla u|^q \chi_{\btrh} )^{\frac{p(x)-\de}{q}} \ dx \\
  & \apprge \int_{B_{\tr/2}} |\nabla u|^{p(x) -\de} \ dx - \int_{\btrh \setminus G}  \mm_{<2\tr} (|\nabla u|^q \chi_{\btrh} )^{\frac{p(x)-\de}{q}} \ dx \\
  & \apprge \int_{B_{\tr/2}} |\nabla u|^{p(x) -\de} \ dx - \int_{\btrh}  \lbr \fint_{\btr} |\nabla u|^q \ dx \rbr^{\frac{p(y) - \de}{q}}\ dy - \mu |\btrh|.
 \end{aligned}
\end{equation}

We shall now estimate $\int_{\btrh}  \lbr \fint_{\btr} |\nabla u|^q \ dx \rbr^{\frac{p(y) - \de}{q}}\ dy$ as follows: fix any $y \in \btrh$ and set $\mathfrak{t}:= \frac{\pd-\de}{q}\geq \frac{p^--\de}{q}>1$, then we have 
\begin{equation}
\label{8.17}
 \begin{array}{ll}
  \lbr \fint_{\btr} |\nabla u|^q \ dx \rbr^{\frac{p(y) - \de}{q}} & \apprle \lbr \fint_{\btr} ( |\nabla u|+1)^{\frac{(p(x) -\de)}{\mathfrak{t} }} \ dx \rbr^{\frac{p(y) - \de}{q}} \\
  & \apprle \lbr \fint_{\btr} ( |\nabla u|+1)^{\frac{(p(x) -\de)}{\mathfrak{t} }} \ dx \rbr^{\mathfrak{t} + \frac{p(y) - \pd}{q}} \\
  &\apprle  \lbr \fint_{\btr} ( |\nabla u|+1)^{\frac{(p(x) -\de)}{\mathfrak{t} }} \ dx \rbr^{\mathfrak{t} } \lbr \frac{1}{|B_{2\tr}|} M^u\rbr^{\frac{p(y) - \pd}{q}} \\
  & \apprle \lbr \fint_{\btr}  |\nabla u|^{\frac{(p(x) -\de)}{\mathfrak{t} }} \ dx \rbr^{\mathfrak{t} } + 1. 
 \end{array}
\end{equation}

Thus using \eqref{8.17} into \eqref{8.16}, we get
\begin{equation}
\label{estimate_one}
 \begin{array}{ll}
    \int_{\btrh} g^{-\de} \iprod{\aa(x,\nabla u)}{\nabla \tu}\ dx   & \apprge \int_{B_{\tr/2}} |\nabla u|^{p(x) -\de} \ dx - |\btrh|  \lbr \fint_{\btr}  |\nabla u|^{\frac{(p(x) -\de)}{\mathfrak{t} }} \ dx \rbr^{\mathfrak{t} } - \frac{|\btrh|}{\mu^\de}.
 \end{array}
\end{equation}

 \item[Estimate on $D_1$:] Recall that $\mu \in (0,1)$, after using \eqref{bounded}, we get
 \begin{equation}
  \label{estimate_D_1}
  \begin{array}{@{}r@{}c@{}l@{}@{}}
   \int_{D_1} g^{-\de} \iprod{\aa(x,\nabla u)}{\nabla \tu}\ dx 
   & \apprle & \int_{D_1} \mm_{<2\tr} (|\nabla \tu|^q + \mu^q)^{\frac{1-\de}{q}} \mm_{<2\tr} ((|\nabla u|^q+\mu^q)\chi_{\btr})^{\frac{p(x)-1}{q}}  \ dx \\
   & {\overset{\eqref{def_D_1}}{\apprle}} &\de^{1-\de}\int_{D_1}  \mm_{<2\tr} ((|\nabla u|^q+\mu^q)\chi_{\btr})^{\frac{p(x)-\de}{q}}  \ dx \\
   & \overset{\text{Corollary \ref{unweighted_maximal_bound}}}{\apprle} &\de^{1-\de} \lbr \int_{\btr}|\nabla u|^{p(x) -\de} \ dx + |\btr| \rbr.
  \end{array}
 \end{equation}

 \item[Estimate on $D_2$:]
 Using the expansion $\nabla \tu = \phi \nabla u  + u \nabla \phi$, we see that 
 \[\int_{D_2} g^{-\de} \iprod{\aa(x,\nabla u)}{\nabla u} \phi \ dx \geq 0.\]
 Hence we can ignore this term and bound the other term from above as follows:
 \begin{equation*}
  \label{estimate_D_2}
  \begin{array}{@{}r@{}c@{}l@{}@{}}
  \int_{D_2} g^{-\de} \iprod{\aa(x,\nabla u)}{\nabla \phi} u \ dx & \apprle & \int_{D_2} g^{-\de} (\mu^{p(x)-1} + |\nabla u|^{p(x)-1})  \lbr[|] \frac{u}{2\tr} \rbr[|]\ dx \\
  & \overset{\eqref{g_equiv_bnd_3}}{\apprle} & \int_{D_2} \mm_{<2\tr}(|\nabla \tu|^q+\mu^q)^{\frac{-\de}{q}}  \mm_{<2\tr} ((|\nabla u|^q+\mu^q) \chi_{\btr})^{\frac{p(x)-1}{q}} \lbr[|] \frac{u}{2\tr}\rbr[|] \ dx \\
  &{\overset{\eqref{def_D_2}}{\apprle}}& \de^{-\de} \int_{D_2} \mm_{<2\tr} ((|\nabla u|^q+\mu^q) \chi_{\btr})^{\frac{p(x)-1-\de}{q}} \lbr[|] \frac{u}{2\tr}\rbr[|] \ dx. \\
  \end{array}
 \end{equation*}
We can now apply Young's inequality along with Corollary \ref{unweighted_maximal_bound} to get 
\begin{equation}
 \label{estimate_D_2_1}
 \begin{array}{@{}r@{}c@{}l@{} @{}}
  \int_{D_2} g^{-\de} \iprod{\aa(x,\nabla u)}{\nabla \phi} u \ dx 
  &\overset{\text{Corollary \ref{unweighted_maximal_bound}}}{ \apprle} & \varepsilon_1  \int_{\btr} |\nabla u|^{{p(x)-\de}}\ dx +\\
  && \qquad + C(\varepsilon_1)|\btr| \lbr[\{] \fint_{\btr} \lbr[|] \frac{u}{2\tr}\rbr[|]^{\pu-\de} \ dx + 1 \rbr[\}].
 \end{array}
\end{equation}

To control the term $ \fint_{\btr} \lbr[|] \frac{u}{2\tr}\rbr[|]^{\pu-\de} \ dx$, we apply the Sobolev-Poincar\'{e} inequality  with $\mathfrak{s} := \sqrt{\frac{n+1}{n}}$ (see \cite{MZ} or \cite[Theorem 3.3]{AP1} and references therein for the details) to obtain
  \begin{equation}
   \label{estimate_D_2_3}
   \begin{array}{ll}
    \fint_{\btr} \lbr[|] \frac{u}{2\tr}\rbr[|]^{\pu-\de} \ dx & \apprle \lbr \fint_{\btr} |\nabla u|^{\frac{\pd-\de}{\mathfrak{s}}}\ dx \rbr^{\frac{\mathfrak{s}(\pu-\de)}{(\pd-\de)}}.
   \end{array}
  \end{equation}
This is possible since \eqref{def_ratio} implies $\frac{\pu-\de}{\pd-\de} \leq \mathfrak{s}$, hence $\frac{\pd-\de}{\mathfrak{s}} \geq \frac{n (\pu-\de)}{n+1} \geq \frac{n (\pu-\de)}{n+(\pu-\de)}$, which implies $\lbr \frac{\pd-\de}{\mathfrak{s}}\rbr^* \geq \lbr \frac{n (\pu-\de)}{n+(\pu-\de)}\rbr^*=\pu-\de$.

Thus we can further bound \eqref{estimate_D_2_3} using the $\log$-H\"older continuity of $\pp$ to obtain:
\begin{equation}
\label{estimate_D_2_4}
 \begin{array}{r@{}l@{}}
   \lbr \fint_{\btr} |\nabla u|^{\frac{\pd-\de}{\mathfrak{s}}}\ dx \rbr^{\frac{\mathfrak{s}(\pu-\de)}{(\pd-\de)}} & \apprle  \lbr \fint_{\btr} |\nabla u|^{{p(x)-\de}{}} +1 \ dx \rbr^{\frac{\mathfrak{s}\rho(4\tr)}{p^-}}  \lbr \fint_{\btr} |\nabla u|^{\frac{p(x) -\de}{\mathfrak{s}}}+1\ dx \rbr^{{\mathfrak{s}}} \\
   & \apprle \lbr \frac{1}{|\btr|} M^u \rbr^{\frac{\mathfrak{s}\rho(4\tr)}{p^-}}\lbr \fint_{\btr} |\nabla u|^{\frac{p(x) -\de}{\mathfrak{s}}}+1\ dx \rbr^{{\mathfrak{s}}} \\
   & \apprle \lbr \fint_{\btr} |\nabla u|^{\frac{p(x) -\de}{\mathfrak{s}}}\ dx \rbr^{{\mathfrak{s}}} +1.
 \end{array}
\end{equation}

Thus combining \eqref{estimate_D_2_4} and \eqref{estimate_D_2_3} into \eqref{estimate_D_2_1}, we get
\begin{equation}
 \label{estimate_D_2_final}
\int_{D_2} g^{-\de} \iprod{\aa(x,\nabla u)}{\nabla \phi} u \ dx \apprle \varepsilon_1  \int_{\btr} |\nabla u|^{{p(x)-\de}}\ dx +  C(\varepsilon_1) |\btr| \mgh{\lbr \fint_{\btr} |\nabla u|^{\frac{p(x) -\de}{\mathfrak{s}}}\ dx \rbr^{{\mathfrak{s}}} + 1}.
\end{equation}

  \end{description}

 \item[Estimate for $I_2$:] Proceeding as in the proof of Theorem \ref{boundary_below_exponent}, after applying Fubini's theorem, \eqref{g_maximal_bound_bnd_3} and Theorem \ref{measure_density_poincare}, we get
 \begin{equation}
  \label{i_2_2}
  \begin{array}{rcl}
  I_2 
  &\apprle &\frac{C(\varepsilon_2)}{\de} \int_{\btr} |\bff|^{p(x) -\de} \ dx  +  \frac{\varepsilon_2}{\de} \int_{\btr}|\nabla u|^{p(x) - \de} \ dx +  \frac{\varepsilon_2}{\de}|\btr|. 
  \end{array}
 \end{equation}

 \item[Estimate for $I_3$ and $I_4$:] Using the bound $|\nabla \vlam| \leq C \la$, followed by applying Fubini's theorem, \eqref{g_maximal_bound_bnd_3} and Theorem \ref{measure_density_poincare}, we get
 \begin{equation}
 \label{bound_i_3}
  \begin{array}{rcl}
   I_3 
   &\apprle &\int_{\btr} |\nabla u|^{p(x) - \de} \ dx + \int_{\btr} |\bff|^{p(x) - \de} \ dx+ |\btr|,
  \end{array}
 \end{equation}
 \begin{equation}
 \label{bound_i_4}
  \begin{array}{rcl}
   I_4 
   & \apprle &\int_{\btr} |\nabla u|^{p(x) - \de} \ dx + |\btr|.
  \end{array}
 \end{equation}

\end{description}

Combining \eqref{estimate_one}, \eqref{estimate_D_1} and \eqref{estimate_D_2_final}--\eqref{bound_i_4}, for  $\kappa :=\min \{ \mathfrak{t}, \mathfrak{s}\}$ where $\mathfrak{t}=\frac{p^-_{\btr}-\de}{q}$ as used to obtain \eqref{estimate_one} and $\mathfrak{s} = \sqrt{\frac{n+1}{n}}$ as used to obtain  \eqref{estimate_D_2_3}, we get 
\begin{equation*}
 \begin{array}{ll}
  \int_{B_{\tr/2}} |\nabla u|^{p(x) - \de} \ dx & \apprle  C(\varepsilon_1) |\btrh| \lbr \fint_{\btr} |\nabla u|^{\frac{p(x)-\de}{{\kappa}}} \ dx \rbr^{{\kappa}}+  (\de^{1-\de} + \varepsilon_1+\varepsilon_2 + \de) \int_{\btr} |\nabla u|^{p(x) - \de} \ dx\\
  & \qquad +  (C(\varepsilon_2) + \de) \int_{\btr} |\bff|^{p(x) - \de} \ dx + (1+\de^{1-\de} + \de+ C(\varepsilon_1)) |\btr| + \frac{|\btrh|}{\mu^\de}.
 \end{array}
\end{equation*}
Choosing $\varepsilon_1, \varepsilon_2$ small followed by $\de \in (0,\de_3)$ with $\de_3$ small, for some $\theta \in \lbr 0,\frac12\rbr $, we get
\begin{equation*}
 \begin{array}{ll}
  \fint_{B_{\tr/2}} |\nabla u|^{p(x) - \de} \ dx & \apprle \lbr \fint_{\btr} |\nabla u|^{\frac{p(x)-\de}{{\kappa}}} \ dx \rbr^{{\kappa}} + \theta \fint_{\btr} |\nabla u|^{p(x) - \de} \ dx + \fint_{\btr} |\bff|^{p(x) - \de} \ dx + \frac{1}{\mu^\de} + 1.
 \end{array}
\end{equation*}
\emph{In the case $\mu=0$, we can eliminate the term $\frac{1}{\mu^\de}$ above  because of \eqref{eq8.12}}. We can now apply Lemma \ref{gehring} to obtain the desired higher integrability. 
\end{proof}

%


 \subsection{Estimates for the homogeneous problem}
 In this section, we obtain useful  higher integrability results for the homogeneous problem that was proved in \cite[Lemma 3.5]{BO}.
 
\begin{theorem}[\cite{BO}]\label{higher_integrability} Let $\Om$ be a $(\ga,\sig,S_0)$-quasiconvex domain for some $\ga>0$, $\sig>0$ and $S_0>0$. Let $u \in W^{1,\pp}_0(\Om)$ solve \eqref{basic_pde} and $w \in W^{1,\pp}(\Om_r)$ solve \eqref{w_approx_difference} with $r \leq \frac{R_9}{2}$ where 
 $$R_9 :=  \min \left\{\frac{1}{2M^u},\frac{1}{2M^w},\frac{1}{2M_0}, \frac{1}{\om_n^{1/n}}, S_0 ,1\right\} \quad \text{and} \quad \rho(R_9) \leq \frac{1}{2n} < 1.$$ Recall that  $\rho$ is the $\log$-H\"older modulus of continuity function for $\pp$.
 
 Then there exists $\sigma_2 =\sigma_2(\La_0,\La_1,\plog,n) \in (0,4(p^--1)]$ such that for any  $\bar{\sigma} \in (0,\sigma_2]$, $\mathfrak{t} >0 $ and any $0<2\tr\leq r$ with $B_{2\tr}(x_0) \subset B_r$, there holds:
%
 \begin{equation}
  \label{higher_integrability_interior_estimate}
   \fint_{B_{\tr}(x_0)} |\nabla w|^{p(x) (1+\bar{\sigma})} \ dx  \le C \lbr \fint_{B_{2\tr}(x_0)} |\nabla w|^{p(x) \mathfrak{t} } \ dx \rbr^{\frac{1+\bar{\sigma}}{\mathfrak{t}}} + 1,
 \end{equation}
where the constant $C = C(\La_0,\La_1,{\plog},n,\mathfrak{t})$. 
\end{theorem}

 As a consequence, we obtain the following important corollary.
 \begin{corollary}
 \label{gehring_improved}
 Under the assumptions of Theorem \ref{higher_integrability}, we have the ameliorated estimate for any $2\tr < r$ such that $B_{2\tr}(x_0) \subset B_r$ with $2r \leq R_9$:
 \begin{equation*}
  \fint_{B_{\tr}(x_0)} |\nabla w|^{p(x)} \ dx \le C  \lbr \fint_{B_{2\tr}(x_0)} |\nabla w| \ dx \rbr^{p^+_{B_{2\tr}}}+1,
 \end{equation*}
for some $C = C({\plog},\La_0,\La_1,n)$.
\end{corollary}
\begin{proof}
From H\"{o}lder's inequality and Theorem \ref{higher_integrability}, setting  $\mathfrak{t}=\frac{1}{p^+}$, we get
\begin{equation*}\label{hi2-1}
    \begin{aligned}
        \fint_{B_{\tr}(x_0)} |\nabla w|^{p(x)}\ dx 
        &\apprle \lbr \fint_{B_{2\tr}({x}_0)} (|\nabla w|+1) \ dx \rbr^{p^+_{B_{2\tr}}}+1,
    \end{aligned}
\end{equation*}
which proves the Corollary.
\end{proof}
 
 We can further ameliorate the estimates in Theorem \ref{higher_integrability} and Corollary \ref{gehring_improved} as follows: 
 \begin{lemma}
 \label{higher_w}
  Under the assumptions of Theorem \ref{higher_integrability}, there exists $R_{10} = R_{10}(\plog,\La_0,\La_1,n,M_0) \in (0,R_9 / 2)$ such that for  any solution $w \in W^{1,\pp}(\Om_{4r})$ with $4r < R_{10}$ solving
  \begin{equation}\label{higher_integrability_pde_ameliorated}\left\{
  \begin{array}{ll}
   \dv \aa(x,\nabla w) =  0 & \text{in}\  \Om_{4r},\\
 w   \in u + W_0^{1,\pp} (\Om_{4r}),
 \end{array}\right.
 \end{equation}
 the following estimates hold:
 \begin{equation}
  \label{higher_integrability_ameliorated}
  \begin{aligned}
   \fint_{\Om_{3r}} |\nabla w|^{\pu} \ dx  \le C  \fint_{\Om_{4r}} |\nabla w|^{p(x)  } \ dx +1, \\
  \fint_{\Om_{3r}} |\nabla w|^{\pu(1+\frac{\sigma_2}{4})} \ dx  \le C  \lbr \fint_{\Om_{4r}} |\nabla w|^{p(x)} \ dx\rbr^{1+\frac{\sigma_2}{4}} +1, \\
    \fint_{\Om_{3r}} |\nabla w|^{(\pu-\de)(1+\frac{\sigma_2}{4})} \ dx  \le C  \lbr \fint_{\Om_{4r}} |\nabla w|^{p(x)} \ dx\rbr^{1+\frac{\sigma_2}{4}} +1,
  \end{aligned}
 \end{equation}
where the constants $C = C(\La_0,\La_1,{\plog},n)$.
 \end{lemma}
\begin{proof}
\renewcommand{\pu}{p^+_{\Om_{4r}}}
\renewcommand{\pd}{p^-_{\Om_{4r}}}
 We make the following choice for $R_{10}$: 
\begin{equation*}\label{further_px_restriction}
  \pu-\pd \leq {\rho}(8r) \leq {\rho}(2R_{10}) \leq \frac{\sigma_2}{4} \qquad \text{and} \quad R_{10} \leq \frac{R_9}{2}.
 \end{equation*}
Then for any $x \in \Om_{4r}$, we see that 
\begin{equation}
\label{8.20}
 \pu \leq p(x) \lbr 1 + \frac{\pu-\pd}{p^-}\rbr \leq p(x) (1+ {\rho}(8r)). 
\end{equation}
and from the restriction of $\sigma_2 \leq 4 (p^--1)$ (see Theorem \ref{higher_integrability}), we get
\begin{equation}
 \label{8.21}
 \begin{array}{ll}
  \pu \lbr 1 + \frac{\sigma_2}{4} \rbr & \leq (p(x) + \pu-\pd) \lbr 1+ \frac{\sigma_2}{4} \rbr  \leq p(x) \lbr 1+ \frac{\sigma_2}{4} \rbr + (\pu-\pd) p^- \\
  & \leq p(x) \lbr 1 + \frac{\sigma_2}{4} + {\rho}(8r) \rbr. \\
 \end{array}
\end{equation}
Since $R_{10} \leq \frac{1}{4M^w}$, making use of the $\log$-H\"older continuity of $\pp$, we get
 \begin{equation}
 \label{bound_w}
\lbr \fint_{B_{4r}} |\nabla w|^{p(x)} \ dx \rbr^{{\rho}(8r)} \leq \lbr \frac{1}{|B_{4r}|} M^w  \rbr^{{\rho}(8r)} \leq C({\plog},n).
 \end{equation}

\begin{description}[leftmargin=*]
 \item[First estimate in \eqref{higher_integrability_ameliorated}:]  
 Using Lemma \ref{measure_density_quasiconvex}, we get
 \begin{equation}
  \label{first_bound_w}
  \begin{array}{@{} r@{}c@{}l @{}}
\fint_{\Om_{3r}} |\nabla w|^{\pu}\ dx & \apprle &  \fint_{B_{3r}} |\nabla w|^{\pu}\ dx     \overset{\eqref{8.20}}{\apprle} \fint_{B_{3r}} |\nabla w|^{p(x) (1 + {\rho}(8r))} \ dx + 1 \\
   & \overset{\eqref{higher_integrability_interior_estimate}}{\apprle}& \lbr \fint_{B_{4r}} |\nabla w|^{p(x)} \ dx \rbr^{1+ {\rho}(8r)} + 1 \\
   & \overset{\eqref{bound_w}}{\apprle} &
   \fint_{\Om_{4r}} |\nabla w|^{p(x)} \ dx + 1.
  \end{array}
 \end{equation}
 \item[Second estimate in \eqref{higher_integrability_ameliorated}:] Again making use of Lemma \ref{measure_density_quasiconvex}, we get
 \begin{equation}
  \label{second_bound_w}
  \begin{array}{@{} r@{}c@{}l @{}}
\fint_{\Om_{3r}} |\nabla w|^{\pu \lbr 1 + \frac{\sigma_0}{4}\rbr}\ dx 
& \stackrel{\eqref{8.21}}{\apprle}& \fint_{B_{3r}} |\nabla w|^{p(x) \lbr 1 + \frac{\sigma_2}{4} + {\rho}(8r)\rbr}\ dx + 1 \\
   & \stackrel{\eqref{higher_integrability_interior_estimate}}{\apprle}& \lbr \fint_{B_{4r}} |\nabla w|^{p(x)} \ dx \rbr ^{ 1 + \frac{\sigma_2}{4} + {\rho}(8r) } + 1 \\
   & \stackrel{\eqref{bound_w}}{\apprle}& 
   \lbr \fint_{\Om_{4r}} |\nabla w|^{p(x)} \ dx \rbr ^{ 1 + \frac{\sigma_2}{4} } + 1.
  \end{array}
 \end{equation}
\item[Third estimate in \eqref{higher_integrability_ameliorated}:] Using H\"older's inequality, we obtain the following chain of estimate:
\begin{equation}
 \label{third_bound_w}
 \begin{array}{@{} r@{}c@{}l @{}}
  \fint_{\Om_{3r}}|\nabla w|^{(\pu-\de)\lbr 1 + \frac{\sigma_2}{4}\rbr} \ dx 
  & \overset{\eqref{second_bound_w}}{\apprle}& \lbr \lbr \fint_{\Om_{4r}} |\nabla w|^{p(x)} \ dx \rbr ^{ 1 + \frac{\sigma_2}{4} } + 1 \rbr^{\frac{\pu-\de}{\pu}} \\
  & \apprle &   \lbr \fint_{\Om_{4R}} |\nabla w|^{p(x)} \ dx  \rbr^{ 1 + \frac{\sigma_2}{4} } + 1. \\
 \end{array}
\end{equation}
To obtain the last inequality above, we observed that $\frac{\pu-\de}{\pu} \leq 1$ holds. 
\end{description}

This proves the lemma. 
 \end{proof}

 \subsection{Approximation by constant exponent operator}

\begin{remark}
\label{remark_boundary_region}
In order to simplify the exposition, we only consider the following two situations in this section: 
 \begin{itemize}
  \item In the interior case, we have $\Om_R = B_R$ .
  \item In the boundary case, we only consider regions of the form $\Om_R(x_0) = B_R(x_0) \cap \Om$ with $x_0 \in \pa \Om$. This simplifies the exposition for this section.  
 \end{itemize}
\end{remark}

\renewcommand{\pu}{p^+_{\Om_{8r}}}
\renewcommand{\pd}{p^-_{\Om_{8r}}}
\renewcommand{\omf}{\Om_{8r}}
\renewcommand{\omt}{\Om_{3r}}

Let us consider the problem  \eqref{wapprox_boundary} in $\Om_{8r}$ with $8r < R_0$ where $R_0$ is defined in Definition \ref{final_radius_restriction}, i.e.,  we have the following PDE: 
\begin{equation}\label{approx_w}\left\{
  \begin{array}{ll}
   \dv \aa(x,\nabla w) =  0 & \text{in}\  \Om_{8r},\\
 w   \in u + W_0^{1,\pp} (\Om_{8r}).
 \end{array}\right.
 \end{equation}

Let us define $\bb = \bb(x,\zeta) : \omf \times \RR^n \to \RR^n$ such that 
\begin{equation}
 \label{definie_b}
\bb(x,\zeta) := \aa(x,\zeta) (\mu^2 + |\zeta|^2)^{\frac{\pu-p(x)}{2}}, 
\end{equation}
then the following holds for every $\eta, \zeta \in \RR^n$ and all $x \in \omf$ (see \cite{BO1} for the details):
\begin{gather}
 (\mu^2 + |\zeta|^2 )^{\frac12} |D_{\zeta} \bb(x,\zeta)| + |\bb(x,\zeta)| \leq 3 \La_1 (\mu^2 + |\zeta|^2)^{\frac{\pu-1}{2}} \label{bounded_b},\\
 \iprod{D_{\zeta} \bb(x,\zeta) \eta}{\eta} \geq \frac{\La_0}{2} (\mu^2 + |\zeta|^2)^{\frac{\pu-2}{2}}|\eta|^2 \label{ellipticity_b}.
\end{gather}
Now define $\bbb=  \bbb(\zeta): \RR^n \to \RR^n$ which denotes the integral average of $\bb(\cdot,\zeta)$ on $B_{8r}^+$ such that
\begin{equation}
 \label{definition_bbb}
 \bbb(\zeta) := \fint_{B_{8r}^+} \bb(x,\zeta) \ dx. 
\end{equation}
It is easy to see that $\bbb$ also satisfies \eqref{bounded_b} and \eqref{ellipticity_b} with $\bb(x,\zeta)$ replaced by $\bbb(\zeta)$. Moreover a direct calculation now gives
\begin{equation}
\def\arraystretch{2.2}
 \label{BMO_bbb}
 \begin{array}{ll}
  \sup_{\zeta \in \RR^n} \frac{|\bb(x,\zeta) - \bbb(\zeta)|}{(\mu^2 + |\zeta|^2)^{\frac{\pu-1}{2}}} & = \sup_{\zeta \in \RR^n} \lbr[|] \frac{\bb(x,\zeta) }{(\mu^2 + |\zeta|^2)^{\frac{\pu-1}{2}}} - \fint_{B^+_{8r}} \frac{\bb(x,\zeta) }{(\mu^2 + |\zeta|^2)^{\frac{\pu-1}{2}}} \ dx \rbr[|]  \\
  & = \sup_{\zeta \in \RR^n} \lbr[|] \frac{\aa(x,\zeta) }{(\mu^2 + |\zeta|^2)^{\frac{p(x)-1}{2}}} - \fint_{B^+_{8r}} \frac{\aa(x,\zeta) }{(\mu^2 + |\zeta|^2)^{\frac{p(x)-1}{2}}} \ dx \rbr[|] \\
  & = \Theta(\aa,B_{8r}^+)(x),
 \end{array}
\end{equation}
where $\Theta(\aa,B_{8r}^+)(x)$ is defined in \eqref{a_difference}.

From the choice of $R_0$ (see Definition \ref{final_radius_restriction}), we see that Lemma \ref{higher_w} is applicable and thus $w \in W^{1,\pu}(\omt)$.  We now consider the following problem:
\begin{equation}
 \label{vapprox_first}
 \left\{ 
 \begin{array}{ll}
  \dv \bbb(\nabla v) = 0 & \text{in} \ \omt, \\
  v \in w + W_0^{1,\pu}(\omt).
 \end{array}\right. 
\end{equation}

%
Given $w$ as in \eqref{approx_w} and $v$ as in \eqref{vapprox_first}, we see that there is a constant $C = C(\plog,\La_0,\La_1,n)$  such that the  following energy estimates hold trivially:
\begin{equation*}\begin{array}{rcl}
  \int_{\Om_{3R}} |\nabla v|^{\pu} \ dx &\leq& C \lbr \int_{\Om_{3R}} |\nabla w|^{\pu} + 1 \ dx \rbr, \\
  \int_{\Om_{3R}} |\nabla v - \nabla w|^{\pu} \ dx &\leq &C \lbr \int_{\Om_{3R}} |\nabla w|^{\pu} + 1 \ dx \rbr .
 \end{array}\end{equation*}

Using \eqref{size_w}, we see that for any $q \leq \pu$, there holds
 \begin{equation*}
  \begin{array}{l}
  \int_{\Om_{3R}} |\nabla v|^{q} \ dx \leq C (M^w+1) =:M^v_1,  \\
  \int_{\Om_{3R}} |\nabla v - \nabla w|^{q} \ dx \leq C( M^w+1) =:M^v_2.
  \end{array}
 \end{equation*}
Let us now define
 \begin{equation}
  \label{size_v}
  M^v := \max \{ M^v_1, M^v_2\}.
 \end{equation}

\subsection{Fixing a few more constants}\label{main_constants}

The first constant that we define concerns the higher integrability exponent:
\begin{equation}
   \label{sigma_0}
\sigma_0 = \sigma_0(n,\plog,\La_0,\La_1) := \min \{ \sigma_1,\sigma_2\},
\end{equation}
where $\sigma_1$ is from Theorem \ref{higher_integrability_full} and $\sigma_2$ is from Theorem \ref{higher_integrability}.


\renewcommand{\pu}{p^+_{\Om_{8r}}}
\renewcommand{\pd}{p^-_{\Om_{8r}}}
\renewcommand{\omf}{\Om_{8r}}
\renewcommand{\omt}{\Om_{3r}}

\begin{definition}
\label{final_radius_restriction}
Let us now define ${R}_0 = {R}_0({\plog},\La_1,\La_1,n,M_0)$, where $M_0$ is defined in \eqref{size_date} to satisfy all the following properties:
\begin{itemize}
 \item ${R}_0 \leq \min \lbr[\{] \frac{1}{4M_0},\frac{1}{4M^v}, \frac{1}{4M^u}, \frac{1}{4M^w}, \frac14,\frac{S_0}{2}, \frac{{R}_8}{2}, \frac{{R}_{10}}{2}\right\}$, where $M_0$ is from \eqref{size_date}, $M^u$ is from \eqref{size_solution}, $M^w$ is from \eqref{size_w}, $R_8$ is from Theorem \ref{higher_integrability_full}, $R_{10}$ is from Lemma \ref{higher_w}, and $S_0>0$ is a universal constant. 
 \item $ {\rho}(2{R}_0) \leq \min \lbr[\{]\sqrt{\frac{n+1}{n}}-1, \frac{\La_0}{2\La_1}, \frac{\sigma_0}{4}, \frac{1}{2n},\frac14 \rbr[\}]$ where $\rho$ is the modulus of continuity of $\pp$.
\item $\om(2R_0) \leq \min \left\{\frac{q^-\sigma_0}{8}, \frac{q^-\tilde{\sigma}}{2}, \frac{(q^-)^2}{4q^+}, \frac{\sigma_0}{2}, \frac14 \right\}$ where $\om$ is the modulus of continuity of $\qq$ and the exponent $\tilde{\sigma}$ is chosen to be $\tilde{\sig} := \min \left\{ \frac{q^--1}{2},1\right\}$.
\end{itemize}
\end{definition}

 
 \begin{definition}
      \label{def_delta_0}
      For any $\ve \in (0,1)$, let us now set $\de_0 = \de_0(n,\plog,\qlog,\La_0,\La_1,\ve)$ such that the following holds:
      \begin{itemize}
           \item $\de_0 = \min\{ \de_1,\de_2,\de_3\}$ where $\de_1$, $\de_2$ and $\de_3$ are from Theorem \ref{boundary_below_exponent}, Theorem \ref{boundary_difference} and Theorem \ref{higher_integrability_full}, respectively.
           \item $\de_0 \leq \de_4$ where $\de_4$ is defined in Theorem \ref{difference_w_v_boundary}. 
      \end{itemize}
 \end{definition}

 \begin{definition}
      \label{def_ga}
      For any $\ve \in (0,1)$, let us now fix the exponent $\ga_0 = \ga_0(n,\plog,\qlog,\La_0,\La_1,\ve)$ such that the following holds:
      \begin{itemize}
           \item $\ga_0 = \min \{ \ga_2,\ga_3,\ga_4\}$ where $\ga_2$, $\ga_3$ and $\ga_4$ are from Theorem \ref{boundary_difference}, Theorem \ref{difference_w_v_boundary} and Theorem \ref{l_infinity_boundary}, respectively. 
      \end{itemize}

 \end{definition}

\begin{remark}
We point out that the above constants $\delta_0$ and $\gamma_0$ are independent of $\qlog$ in Section \ref{aprioriestimates}. In Section \ref{covering_arguments}, however, $\delta_0$ and $\gamma_0$ are dependent on $\qlog$, see Lemma \ref{co2}.
\end{remark}


 \subsection{Estimates satisfied by the constant exponent approximation in \eqref{vapprox_first}}
%

 \begin{theorem}
 \label{difference_w_v_boundary}
For any $\sig \in (0,1/4)$ and $S_{0}>0$,  let $\Om$ be a $(\ga,\sig,S_{0})$-quasiconvex domain for a $\ga \in (0,\ga_3)$ with $\ga_3$ to be chosen and let $r$ be such that  $8r< R_0$ with $R_0$ as defined in Definition \ref{final_radius_restriction}. Let $u \in W_0^{1,\pp}(\Om)$, $w \in W^{1,\pp}(\Om_{8r})$ and  $v \in W^{1,\pu}(\omt)$ solve \eqref{basic_pde}, \eqref{approx_w} and \eqref{vapprox_first}, respectively and fix any $\la>1$. Since $u=0$ on $\pa \Om$, we extend $u=0$ on $\Om^c$ followed by  $w=u$ on $\omf^c$ and $v=w$ on $\omt^c$. For every $\varepsilon \in (0,1)$, there exist $\de_4 = \de_4(\varepsilon,n,\La_0,\La_1,\plog)>0$ and  $\ga_3 = \ga_3(\varepsilon,\La_0,\La_1,\plog,n)>0$ such that for all $\de \in (0,\de_4)$ and $\ga \in (0,\ga_3)$, the following holds:
 \begin{equation}
  \label{bound_on_data_two}
  \fint_{\omf}|\nabla u|^{p(x)-\de} \ dx \leq \la \qquad \text{and} \qquad \fint_{\omf} |\bff|^{p(x)-\de} \ dx \leq \ga \la,
 \end{equation}
then there exists a constant $C = C(\La_0,\La_1,{\plog},n,\sigma)$ such that the following conclusions hold:
\begin{equation*}
 \label{diff_conclusion}
 \begin{array}{c}
 \fint_{\omt} |\nabla v|^{\pu-\de} \ dx + \fint_{\omt} |\nabla w|^{\pu-\de} \ dx \leq C \la, \\
 \fint_{\omt} |\nabla w - \nabla v|^{\pu-\de} \ dx \leq \varepsilon \la. 
 \end{array}
\end{equation*}

\end{theorem}
\begin{proof}
 {\color{black}\bf Let us first prove $\fint_{\omt} |\nabla w|^{\pu-\de} \ dx \leq C \la$: }
Applying H\"older's inequality and Jensen's inequality along with \eqref{higher_integrability_ameliorated} and Corollary \ref{gehring_improved} , we get
 \begin{equation}
 \label{first_estimate}
  \begin{array}{ll}
   \fint_{\omt} |\nabla w|^{\pu -\de} \ dx & {\leq} \lbr \fint_{\omt} |\nabla w|^{\pu} \ dx \rbr^{\frac{\pu-\de}{\pu}}\\
   & {\apprle} \lbr \fint_{\Om_{4r}} |\nabla w|^{p(x)} +1 \ dx \rbr^{\frac{\pu-\de}{\pu}}\\
   &{\apprle} \lbr \lbr \fint_{\omf} |\nabla w|\ dx \rbr^{\pu} +1  \rbr^{\frac{\pu-\de}{\pu}} \\
   & \apprle  \lbr  \fint_{\omf} \lbr |\nabla w|  +1\rbr^{{p(x)-\de}{}}  \ dx \rbr^{\frac{\pu-\de}{\pd-\de}}.\\
   \end{array}
   \end{equation}
   From the restriction $8r\leq R_0$, we see that the following bound holds:
   \begin{equation}
    \label{bound_w_2}
    \lbr \fint_{B_{8r}} |\nabla w|^{p(x)-\de} +1 \ dx \rbr^{\frac{\pu-\pd}{\pd-\de}}  \leq \lbr \frac{1}{|B_{8r}|} M^w \rbr^{C{\rho(8r)}} \leq C. 
   \end{equation}
Since all the assumptions of Theorem \ref{boundary_difference} are satisfied, we see that $\fint_{\Om_{8R}} |\nabla w|^{p(x) - \de} \ dx \leq C \la$ holds  for all $\de \in (0,\de_2)$ with $\de_2$ as in Theorem \ref{boundary_difference}. Using this and \eqref{bound_w_2} into
%
%
\eqref{first_estimate}, we get the estimate: \begin{equation}\label{bound_w_4}\fint_{\omt} |\nabla w|^{\pu-\de} \ dx \leq C \la.\end{equation}
%
{\color{black}\bf Let us now prove $\fint_{\omt} |\nabla v|^{\pu-\de} \ dx \leq C \la$:} Since $\bbb(\cdot)$ is a  constant exponent operator, we can apply Theorem \ref{boundary_below_exponent}  such that for any $\de \in (0,\de_1)$ with $\de_1$ as in  Theorem \ref{boundary_below_exponent}, the following estimate holds:
 \begin{equation}
  \label{9.13}
  \fint_{\omt} |\nabla v|^{\pu-\de} \ dx \leq C(n,\plog,\La_0,\La_1) \lbr \fint_{\omt} |\nabla w|^{\pu-\de} \ dx + 1 \rbr.
 \end{equation}
\emph{Alternatively, we can also apply \cite[Corollary 3.5]{AP1} to obtain \eqref{9.13} for suitably small $\de$. }
 Using \eqref{bound_w_4} into \eqref{9.13} gives the desired estimate for any  $\de \leq \min\{\de_1,\de_2\}$ where $\de_1$ is coming from applying Theorem \ref{boundary_below_exponent} to \eqref{vapprox_first} and    $\de_2$ is obtained from Theorem \ref{boundary_difference}.

\noindent {\color{black} \bf Let us now prove $\fint_{\omt} |\nabla w - \nabla v|^{\pu-\de} \ dx \leq \ve \la$:} Since $w-v \in W_0^{1,\pu}(\omt)$, we shall define
 \begin{equation*}
  \label{def_g_w_v}
  g(x) := \max \left\{ \mm_{<6r} (|\nabla w - \nabla v|^q)^{1/q}(x), \frac{|w(x)-v(x)|}{d(x,\pa\omt)} \right\}.
 \end{equation*}
Using Lemma \ref{lipschitz_extension_one} and Lemma \ref{lipschitz_extension_two}, we get a Lipschitz function denoted by $v_{\la}$ which is a valid test function for \eqref{vapprox_first}. Using this as a test function, we get
\begin{equation}
 \label{w_v_1}
 \begin{array}{ll}
 \int_{\omt} \iprod{\bbb(\nabla w)-\bbb(\nabla v)}{\nabla \vlam}\ dx & = \int_{\omt} \iprod{\bbb(\nabla w) - \bb(x,\nabla w)}{\nabla \vlam} \ dx + \\
 &  \qquad + \int_{\omt} \iprod{\bb(x,\nabla w)-\aa(x,\nabla w)}{\nabla \vlam} \ dx. 
 \end{array}
\end{equation}
Define the set 
$F_{\la} := \{ x \in \Om_{3r} : g(x) \leq \la \}$, 
now multiplying \eqref{w_v_1} by $\la^{-1-\de}$ and integrating over $(0,\infty)$, we obtain 
\begin{equation}\label{w_v_1.1}
 I_1 = I_2 + I_3 + I_4 + I_5 +I_6,
\end{equation}
where
\begin{itemize}
 \item $I_1 := \int_0^{\infty} \la^{-1-\de} \int_{\flam} \iprod{\bbb(\nabla w)-\bbb(\nabla v)}{\nabla w - \nabla v}\ dx$.
 \item $I_2 := \int_0^{\infty} \la^{-1-\de} \int_{\flam} \iprod{\bbb(\nabla w)-\bb(x,\nabla w)}{\nabla w - \nabla v}\ dx$.
 \item $I_3 := \int_0^{\infty} \la^{-1-\de} \int_{\flam} \iprod{\bb(x,\nabla w)-\aa(x,\nabla w)}{\nabla w - \nabla v}\ dx$.
 \item $I_4 :=  - \int_0^{\infty} \la^{-1-\de} \int_{\flam^c} \iprod{\bbb(\nabla w)-\bbb(\nabla v)}{\nabla \vlam}\ dx$.
 \item $I_5 :=  \int_0^{\infty} \la^{-1-\de} \int_{\flam^c} \iprod{\bbb(\nabla w)-\bb(x,\nabla w)}{\nabla \vlam}\ dx$.
 \item $I_6 :=  \int_0^{\infty} \la^{-1-\de} \int_{\flam^c} \iprod{\bb(x,\nabla w)-\aa(x,\nabla w)}{\nabla \vlam}\ dx$.
\end{itemize}

\begin{description}[leftmargin=*]
 \item[Estimate for $I_1$:] Applying Fubini's theorem and Young's inequality and following the calculation of \eqref{estimate_I_1}, we get
%
\begin{equation}
 \label{w_v_5}
 \begin{array}{ll}
   \int_{\omt} |\nabla w - \nabla v|^{\pu-\de} \ dx &\apprle  
  C(\ep_1) \de I_1  +   {\ep_1}\int_{\omt}|\nabla w|^{{\pu-\de}}\ dx+ \ep_1|\omt|.\\
  \end{array}
\end{equation}

 \item[Estimate for $I_2$:] 
 Using Fubini's theorem and Young's inequality along with \eqref{BMO_bbb}, we get
 \begin{equation}
  \label{w_v_6}
  \begin{array}{r@{}c@{}l}
   |I_2| 
   & \apprle &\frac{1}{\de} \int_{\omt} \Theta(\aa,B_{8r}^+) (\mu^2 + |\nabla w|^2 )^{\frac{\pu-1}{2}} |\nabla w - \nabla v| g(x)^{-\de} \ dx \\
   &\apprle &\frac{\ep_2}{\de} \int_{\omt} |\nabla w - \nabla v|^{\pu-\de} \ dx + \frac{c(\ep_2)}{\de} \underbrace{\int_{\omt} \Theta(\aa,B_{8r}^+)^{\frac{\pu-\de}{\pu-1}} (1+|\nabla w|)^{\pu-\de} \ dx}_{=:I_{2, 2}}.
  \end{array}
 \end{equation}
 Let us now estimate $I_{2, 2}$ as follows:
 \begin{equation}
 \label{w_v_7}
 \begin{array}{ll}
 I_{2, 2} 
&   \apprle \lbr \fint_{\omt}\Theta(\aa,B_{8r}^+)^{\frac{\pu-\de}{\pu-1}\frac{1+\sigma_0/4}{\sigma_0/4}} \ dx\rbr^{\frac{\sigma_0/4}{1+\sigma_0/4}} \lbr \fint_{\omt}|\nabla w|^{(\pu-\de)(1+\sigma_0/4)} \ dx + 1 \rbr^{\frac1{1+\sigma_0/4}}. \\
 \end{array}
 \end{equation}
 Using Theorem \ref{higher_integrability_full} and \eqref{first_estimate}, we also get
 \begin{equation}
 \label{w_v_7_1}
 \begin{array}{ll}
  \lbr \fint_{\omt}|\nabla w|^{(\pu-\de)(1+\sigma_0/4)} \ dx + 1 \rbr^{\frac1{1+\sigma_0/4}} & \apprle 
\lbr \fint_{\Om_{4r}} |\nabla w|^{p(x)} \ dx    \rbr^{\frac{\pu-\de}{\pu}} +1 \apprle \la.
 \end{array}
 \end{equation}
 
%
%
%
%
Using Lemma \ref{measure_density_quasiconvex}, we have the bound
  \begin{equation}
   \label{upper_bound_measure}
   |\omt \setminus B_{3r}^+| \leq C(n) \lbr \frac{96 \ga r}{\sigma^3}\rbr^n.
  \end{equation}
%
%
As a consequence, after applying Young's inequality and using the bound $\Theta(\aa,B^+_{8r}) \leq 2 \La_1$, we get
%
 \begin{equation*}
  \label{w_v_8}
  \begin{array}{ll}
  \fint_{\omt}\Theta(\aa,B_{8r}^+)^{\frac{\pu-\de}{\pu-1}\frac{1+\sigma_0/4}{\sigma_0/4}} \ dx 
    & \apprle \frac{|B_{3r}|}{|\omt|} (2\La_1)^{{\frac{\pu-\de}{\pu-1}\frac{1+\sigma_0/4}{\sigma_0/4}}-1}\fint_{B_{8r}^+}  \Theta(\aa,B_{8r}^+) \ dx +\\ 
  & \qquad + (2\La_1)^{\frac{\pu-\de}{\pu-1}\frac{1+\sigma_0/4}{\sigma_0/4}}	 \frac{|\omt \setminus B_{3r}^+|}{|\omt|}. \\
  \end{array}
 \end{equation*}
Using Lemma \ref{measure_density_quasiconvex} along with  \eqref{small_aa} and \eqref{upper_bound_measure}, we get
\begin{equation}
 \label{w_v_9}
 \fint_{\omt}\Theta(\aa,B_{8r}^+)^{\frac{\pu-\de}{\pu-1}\frac{1+\sigma_0/4}{\sigma_0/4}} \ dx \leq C_0 (\ga + \ga^n) \leq C_0 \ga, 
\end{equation}
where $C_0 = C_0(\La_0,\La_1,\sigma,\plog,n)$. Here we have used $\ga < 1$ to bound $\ga^n< \ga$.

Combining \eqref{w_v_9} and  \eqref{w_v_7_1} into \eqref{w_v_7}, we get
 $I_{2, 2} \apprle  \ga^{\frac{\sigma_0/4}{1+\sigma_0/4}} \la$ which is then substituted into \eqref{w_v_6} to get
%
%
\begin{equation}
 \label{w_v_10}
 \begin{array}{ll}
 |I_2| 
 & \apprle \frac{\ep_2}{\de} \int_{\omt} |\nabla w - \nabla v|^{\pu-\de} \ dx + \frac{c(\ep_2)}{\de} |\omt|  \ga^{\frac{\sigma_0/4}{1+\sigma_0/4}} \la. \\
 \end{array}
\end{equation}
Recall the definition of  $\sig_0$ from  \eqref{sigma_0}.

 \item[Estimate for $I_3$:] Applying Fubini's theorem followed by Young's inequality, we get
 \begin{equation}
  \label{w_v_11}
  |I_3| \apprle \frac{\ep_3}{\de} \int_{\omt} |\nabla w - \nabla v|^{\pu-\de} \ dx + \frac{C(\ep_3)}{\de} \int_{\omt} |\bb(x,\nabla w) - \aa(x,\nabla w)|^{\frac{\pu-\de}{\pu-1}} \ dx.
 \end{equation}
We shall now proceed with estimating the second term in \eqref{w_v_11} as follows: Denote $\overline{\Om}_{3r} = \{ x \in \omt : \mu^2 + |\nabla w|^2 >0\}$, then   using \eqref{definie_b}, we see that 
\begin{equation*}
 \label{w_v_12}
 |\bb(x,\nabla w) - \aa(x,\nabla w)| = |\aa(x,\nabla w)| \left|1 - (\mu^2 + |\nabla w|^2)^{\frac{\pu-p(x)}{2}} \right|.
\end{equation*}
For each $x \in \overline{\Om}_{3r}$, in view of the mean value theorem applied to $(\mu^2 + |\nabla w|^2)^{\frac{\pu-p(x)}{2}\mt}$, there exists $\mt_x \in [0,1]$ such that we get
\begin{equation}
 \label{w_v_13}
 (\mu^2 + |\nabla w|^2)^{\frac{\pu-p(x)}{2}} -1 = \frac{\pu-p(x)}{2} (\mu^2 + |\nabla w|^2)^{\frac{\pu-p(x)}{2}\mt_x} \log (\mu^2 + |\nabla w|^2). 
\end{equation}
This implies
\begin{equation*}
 \label{w_v_13.1}
 \begin{array}{l}
 |\aa(x,\nabla w)| \left|1 - (\mu^2 + |\nabla w|^2)^{\frac{\pu-p(x)}{2}} \right| 
 \leq \frac{\La_1{\rho}(8r)}{2} (\mu^2 + |\nabla w|^2)^{\frac{(\pu-p(x))\mt_x + p(x) -1}{2}} \log (\mu^2 + |\nabla w|^2). 
 \end{array}
\end{equation*}

Let us now define the sets 
\begin{equation}
\label{split_sets}
\overline{\Om}_{3r}^1 : = \{ x \in \overline{\Om}_{3r} : |\nabla w(x)| \leq 1\} \quad \text{and} \quad  \overline{\Om}_{3r}^2 : = \{ x \in \overline{\Om}_{3r} : |\nabla w(x)| > 1\}. 
\end{equation}
Recall that $\mu \leq 1$ and hence using the inequality $t^{\be} |\log t| \leq \max \left\{ \frac{1}{e^{\be}}, 2^{\be}\log 2\right\}$ which holds for all $t \in (0,2]$ and any $\be >0$, we get for $x \in \overline{\Om}_{3r}^1$
\begin{equation}
 \label{w_v_14}
 |\aa(x,\nabla w)| \left|1 - (\mu^2 + |\nabla w|^2)^{\frac{\pu-p(x)}{2}} \right| \leq \frac{\La_1 {\rho}(8r)}{2} \max \left\{ \frac{1}{e^{\frac{p^--1}{2}}}, 2^{\frac{p^+-1}{2}}\log 2 \right\}.
\end{equation}
To obtain the above estimate, with  $\be(x) := \frac{\mt_x (\pu-p(x)) + p(x) -1}{2}$, there holds 
\begin{equation*}
 \label{bound_beta}
 \frac{p^--1}{2} \leq \be(x) \leq \frac{\pu-1}{2} \leq \frac{p^+-1}{2}.
\end{equation*}

Hence using \eqref{split_sets} and combining \eqref{w_v_14} into \eqref{w_v_13}, we get
\begin{equation}
 \label{w_v_15}
 \begin{array}{ll}
  |\bb(x,\nabla w) - \aa(x,\nabla w)|& \apprle \chi_{\overline{\Om}_{3r}^1}\frac{\La_1 \rho(8r)}{2} \max \left\{ \frac{1}{e^{\frac{p^--1}{2}}}, 2^{\frac{p^+-1}{2}}\log 2 \right\} + \\
  & \quad + \chi_{\overline{\Om}_{3r}^2} {{\rho}(8r)}{} |\nabla w|^{(\pu-p(x))\mt_x + p(x) -1} \log(e + |\nabla w|). 
 \end{array}
\end{equation}

Combining \eqref{w_v_15} and \eqref{w_v_11}, we get
\begin{equation}
 \label{w_v_16}
 \begin{array}{ll}
   \int_{\omt} |\bb(x,\nabla w) - \aa(x,\nabla w)|^{\frac{\pu-\de}{\pu-1}} \ dx 
   & \apprle {\rho}(8r)^{\frac{\pu-\de}{\pu-1}}|{\Om}_{3r}| \lbr 1+ J\rbr. 
 \end{array}
\end{equation}
where $J:= \fint_{\omt} |\nabla w|^{\pu-\de} [\log(e+|\nabla w|)]^{\frac{\pu-\de}{\pu-1}} $.  Using the inequality $\log(e + ab) \leq \log(e+a) + \log(e+b)$ for  $a,b>0$, we get
\begin{equation*}
 \label{w_v_17}
 \begin{array}{ll}
  J & \apprle \fint_{\omt} |\nabla w|^{\pu-\de} \lbr[[]\log \lbr e + \frac{|\nabla w|^{\pu-\de}}{\avg{|\nabla w|^{\pu-\de}}{\omt}} \rbr \rbr[]]^{\frac{\pu-\de}{\pu-1}}\ dx +\\
  & \qquad + \fint_{\omt} |\nabla w|^{\pu-\de} \lbr[[]\log \lbr e + {\avg{|\nabla w|^{\pu-\de}}{\omt}} \rbr \rbr[]]^{\frac{\pu-\de}{\pu-1}}\ dx \\
  & =: J_1 + J_2.
 \end{array}
\end{equation*}
\begin{description}
 \item[Estimate for $J_1$:] We now apply Lemma \ref{llogl} with $\be = \frac{\pu-\de}{\pu-1}$, $s=\frac{\pu}{\pu-\de}$ and $f = |\nabla w|^{\pu-\de}$ to get
 \begin{equation}
  \label{w_v_18}
  J_1 \leq \lbr \fint_{\omt} |\nabla w|^{\pu} \ dx \rbr^{\frac{\pu-\de}{\pu}} \overset{\eqref{first_estimate}}{\apprle} \la. 
 \end{equation}

 \item[Estimate for $J_2$:] We see that 
 \begin{equation*}
  \label{w_v_19}
  \begin{array}{ll}
   \log \lbr e + \avg{|\nabla w|^{\pu-\de}}{\omt}\rbr & \leq \log \lbr e + C\avg{|\nabla w|^{p(x)-\de}+1}{\omt}\rbr \\
   & {\apprle}\log \lbr \frac{1}{|\omt|}\rbr + \log (e|\omt| + M^w).
  \end{array}
 \end{equation*}
Since we have $R_0 \leq \frac{1}{2M^w}$, we get
 \begin{equation*}
  \label{w_v_20}
  \begin{array}{ll}
   \log \lbr e + \avg{|\nabla w|^{\pu-\de}}{\omt}\rbr & \apprle \log \lbr \frac{1}{|\omt|}\rbr + \log ( M^w) + 1 
    \apprle \log \lbr \frac{1}{8r} \rbr.
  \end{array}
 \end{equation*}
Thus we can estimate $J_2$ by:
 \begin{equation}
  \label{w_v_21}
   J_2 
   \apprle  \fint_{\omt} |\nabla w|^{\pu-\de} \lbr[[] \log \lbr \frac{1}{8r} \rbr \rbr[]]^{\frac{\pu-\de}{\pu-1}}\ dx 
   \overset{\eqref{first_estimate}}{\apprle} \lbr[[] \log \lbr \frac{1}{8r} \rbr \rbr[]]^{\frac{\pu-\de}{\pu-1}}\la. 
 \end{equation}
\end{description}
Now combining \eqref{w_v_21} and \eqref{w_v_18} into \eqref{w_v_16}, we get 
\begin{equation}
 \label{w_v_22}
   \int_{\omt} |\bb(x,\nabla w) - \aa(x,\nabla w)|^{\frac{\pu-\de}{\pu-1}} \ dx \apprle |\omt|\lbr {\rho}(8r) \log \lbr \frac{1}{8r} \rbr  \rbr^{\frac{\pu-\de}{\pu-1}}\la.
\end{equation}
Hence using \eqref{w_v_22} into \eqref{w_v_11} and from \eqref{small_px},  we get
\begin{equation}
 \label{w_v_23}
   |I_3| \apprle \frac{\ep_3}{\de} \int_{\omt} |\nabla w - \nabla v|^{\pu-\de} \ dx + \frac{C(\ep_3)}{\de}|\omt|\ga^{\frac{\pu-\de}{\pu-1}}\la.
\end{equation}

\item[Estimate for $I_4$:] Applying Fubini's theorem and Young's inequality, \eqref{9.13}, \eqref{first_approx} and the standard maximal function bound,  we get
   \begin{equation}
   \label{w_v_24}
   I_4 \apprle \int_{\omt} |\nabla w - \nabla v|^{\pu-\de} \ dx + |\omt|. 
 \end{equation}

 \item[Estimate for $I_5$:] Again making use of Fubini's theorem and Young's inequality, we get
 \begin{equation}
  \label{w_v_25}
  \begin{array}{ll}
   I_5 \apprle \int_{\omt} g(x)^{\pu-\de} \ dx + \int_{\omt} |\bbb(\nabla w) - \bb(x,\nabla w)|^{\frac{\pu-\de}{\pu-1}} \ dx. 
  \end{array}
 \end{equation}
 Note that the last term in \eqref{w_v_25} is exactly the same as the last term in \eqref{w_v_6} and thus we can directly use \eqref{w_v_10} to obtain the estimate
 \begin{equation}
  \label{w_v_26}
  I_5 \apprle\int_{\omt} |\nabla w - \nabla v|^{\pu-\de} \ dx +  |\omt|\ga^{\frac{\sigma_0/4}{1+\sigma_0/4}} \la. 
 \end{equation}

 \item[Estimate for $I_6$:] Exactly as the estimate for $I_5$, we analogously get (see \eqref{w_v_11} for the details)
 \begin{equation}
  \label{w_v_28}
  I_6 \apprle\int_{\omt} |\nabla w - \nabla v|^{\pu-\de} \ dx +  |\omt|\ga^{\frac{\pu-\de}{\pu-1}} \la. 
 \end{equation}

 \end{description}

We now combine \eqref{w_v_10}, \eqref{w_v_23}, \eqref{w_v_24}, \eqref{w_v_26} and \eqref{w_v_28} into \eqref{w_v_5} and use \eqref{w_v_1.1} to get
\begin{equation*}
 \label{w_v_29}
 \begin{array}{ll}
   \int_{\omt} |\nabla w - \nabla v|^{\pu-\de} \ dx & \apprle C(\ep_1)(\ep_2+\ep_3+\de) ) \int_{\omt} |\nabla w - \nabla v|^{\pu-\de} \ dx + \\
   & \qquad + C(\ep_1,\ep_2,\ep_3)|\omt| \lbr \ga^{\frac{\sigma_0/4}{1+\sigma_0/4}} + \ga^{\frac{p^--\de}{p^+-\de}}\rbr  \la   \\ 
   & \qquad + \ep_1 \int_{\omt} |\nabla w|^{\pu-\de} \ dx + (\ep_1 + \delta) |\omt|.
 \end{array}
\end{equation*}
Choose $\ep_1$ small followed by $\ep_2$ and $\ep_3$, then finally $\de$ and $\ga$ small and lastly making use of \eqref{first_estimate}, the proof follows.

\end{proof}

\subsection{$L^{\infty}$ gradient estimates up to the boundary}

We need $L^{\infty}$ bound for the gradient in both the interior and boundary case for the constant coefficient equation \eqref{vapprox_first}.
Note that the operator $\bbb$ defined in \eqref{definition_bbb} is independent of $x$. Hence in the interior case, we get
\begin{theorem}
 \label{l_infinity_interior}
 Let $\omt = B_{3r}$ and \eqref{bound_on_data_two} hold. Suppose that $v \in W^{1,\pu}(B_{3r})$  with $B_{3r} \subset \Om$ solves 
 \begin{equation}
 \label{vapprox_first_1}
 \left\{ 
 \begin{array}{ll}
  \dv \bbb(\nabla v) = 0 & \text{in} \ B_{3r}, \\
  v \in w + W_0^{1,\pu}(B_{3r}), 
 \end{array}\right. 
\end{equation}
then the following holds:
\begin{equation*}
 \|\nabla v\|_{L^{\infty}(B_r)} \leq \lbr \fint_{B_{3r}} |\nabla v|^{\pu} \ dx + 1\rbr^{\frac{1}{\pu}}  \leq C \la.
\end{equation*}

\end{theorem}

\begin{proof}
 We only need to show $\lbr \fint_{B_{3r}} |\nabla v|^{\pu} \ dx + 1\rbr^{\frac{1}{\pu}}  \leq C \la$  as the first bound is from \cite{DiBenedetto}. Thus from the energy estimate in Lemma \ref{energy_homogeneous} applied to \eqref{vapprox_first_1}, we get
 \[
\lbr   \fint_{B_{3r}} |\nabla v|^{\pu} \ dx + 1\rbr^{\frac{1}{\pu}} \apprle\lbr \fint_{B_{3r}} |\nabla w|^{\pu} \ dx + 1\rbr^{\frac{1}{\pu}}.
 \]
Similar to how \eqref{first_estimate} was obtained, we see that 
 \[
  \lbr \fint_{B_{3r}} |\nabla w|^{\pu} \ dx + 1\rbr^{\frac{1}{\pu}} \apprle  \fint_{B_{8r}} |\nabla w|^{p(x)-\de} \ dx + 1.
 \]
Using the calculations from \eqref{vapprox_first_1}, we can conclude the following bound holds:
\[
 \lbr   \fint_{B_{3r}} |\nabla v|^{\pu} \ dx + 1\rbr^{\frac{1}{\pu}} \apprle  \fint_{B_{8r}} |\nabla w|^{p(x)-\de} \ dx + 1 \leq C \la. 
\]
\end{proof}

In the boundary case, we need to make us of the fact that our domain $\OO$ is $\gss$-quasiconvex. Let us consider the following problem with $x_0 \in \pa \Om$:
\begin{equation}
 \label{approxv}
 \left\{ \begin{array}{ll}\dv \bbb(\nabla \overline{V}) =0 \quad & \ \ \text{in} \ \ F_{2r}^*(x_0), \\
 \overline{V} = 0 & \ \ \text{on} \ \ \partial_w F_{2r}^*(x_0).\\
 \end{array}
 \right.
\end{equation}
Here $F^*_{2r}$ is defined in Lemma \ref{convex-outside} and we have extended $\overline{V} = 0$ on $B_{2r}(x_0) \setminus F_{2r}^*(x_0)$. 

\begin{theorem}[\cite{BHSW}]
\label{l_infinity_boundary}
 For any $\varepsilon \in (0,1)$, there exists $\ga_4  = \ga_4 (\La_0,\La_1,\plog,n,\varepsilon)$  such that if \eqref{bound_on_data_two} holds and $\Om$ is a $\gss$-quasiconvex domain for any $\ga \in (0,\ga_4)$ and $\sig\in(0,1/4)$, then the unique weak solution of \eqref{approxv} satisfies the estimate:
 \begin{gather}
  \|\nabla \overline{V}\|_{L^{\infty}(B_{\frac{5\sigma r}{12}}(x_0))} \leq C \left( \int_{F_{2r}^*(x_0)} |\nabla \overline{V}|^{\pu} \ dx + 1 \right)^{\frac{1}{\pu}} \leq C \la ,\label{V-1}
 \end{gather}
where the constant $C = C(\La_0,\La_1,{\plog},n,\sigma)$.

 Furthermore, suppose that $v \in W^{1,\pu}(\omt)$ solves \eqref{vapprox_first} with the bound $\fint_{\omt} |\nabla v|^{\pu - \de} \leq C \la$, then there holds for any $\de \in \gh{0,1/4}$:
 \begin{gather}
  \fint_{B_{\frac{\sigma r}{2}}(x_0)} |\nabla \overline{V}|^{\pu - \de} \ dx \leq C \la \quad \text{and} \quad \fint_{B_{\frac{\sigma r}{2}}(x_0)} |\nabla \overline{V} - \nabla v|^{\pu - \de} \ dx \leq \varepsilon \la.\label{V-2}
 \end{gather}

\end{theorem}

\begin{proof}
 Similar to how we obtained the estimate in Theorem \ref{l_infinity_interior}, we can analogously obtain
 \[
 \left( \int_{F_{2r}^*(x_0)} |\nabla \overline{V}|^{\pu} \ dx + 1 \right)^{\frac{1}{\pu}} \leq C \la . 
 \]
The first bound in \eqref{V-1} is from \cite[Lemma 4.2]{BHSW} and applying H\"older's inequality to \cite[Corollary 4.4]{BHSW} proves \eqref{V-2} (see also \cite[Corollary 2.7]{AP2} for an expanded version of the calculation).
\end{proof}


\section{Covering arguments}\label{covering_arguments}
Let $\ga \in (0,\ga_0)$, $\delta \in (0,\delta_0)$, $\sig \in (0,1/4)$ and  $S_0 >0$ be given, where $\ga_0$ and $\delta_0$ are from Definition \ref{def_ga} and Definition \ref{def_delta_0}, respectively. Assume that  $(\pp,\aa,\Om)$ is $(\ga,\sigma,S_0)$-vanishing in the sense of Definition \ref{further_assumptions}.
%
%
Let $r \le R_0/4$, where $R_0$ is from Definition \ref{final_radius_restriction} and let $\OO_{kr} = \OO_{kr}(x_0)$ for $x_0 \in \OO$ and $k \in \NN$, we then set
\begin{gather*}
\M^*(\nabla u)(x) := \M\gh{|\nabla u|^{(p(\cdot)-\delta)\frac{q(\cdot)}{q_{\OO_{4r}}^-}}\lsb{\chi}{\OO_{2r}}}(x),\\
\M^*_{1+\tilde{\sigma}}(\bff)(x) := \bgh{\M \gh{ \bgh{|\bff|^{(p(\cdot)-\delta)\frac{q(\cdot)}{q_{\OO_{4r}}^-}}+1}^{1+\tilde{\sigma}} \lsb{\chi}{\OO_{2r}}}(x)}^\frac{1}{1+\tilde{\sigma}},
\end{gather*}
where $\tilde{\sigma}$ is given in Definition \ref{final_radius_restriction}, $\M$ is the standard maximal operator and $\chi$ is the standard characteristic function.
For a fixed $\ve \in (0,1)$ and $N>1$, we further define
\begin{equation}\label{cc4}
\lambda_0 := \frac{1}{\ve} \mgh{ \mint{\OO_{4r}}{|\nabla u|^{p(x)-\delta}}{dx} + \gh{ \mint{\OO_{4r}}{|\bff|^{(p(x)-\delta)(1+\tilde{\sigma})}}{dx}}^{\frac{1}{1+\tilde{\sigma}}} +1},
\end{equation}
and  for $k \in \NN \cup \{0\}$, the upper-level sets
\begin{gather*}
C_k := \mgh{x \in \OO_{r} : \M^*(\nabla u)(x) > N^{k+1}\lambda_0},\\
D_k := \mgh{x \in \OO_{r} : \M^*(\nabla u)(x) > N^{k}\lambda_0} \cup \mgh{x \in \OO_{r} : \M^*_{1+\tilde{\sigma}}(\bff)(x) > \gamma N^{k}\lambda_0}.
\end{gather*}
Note that $\ve$ and $N$ are to be chosen later depending only on $\La_0,\La_1,\plog, \qlog, n, \sigma$.

We now verify two assumptions in Lemma \ref{calderon_zygmund}. 
\begin{lemma}\label{co1}
There exists a constant $N_1 = N_1(\La_0,\La_1,\plog, \qlog, n, \sigma) > 1$ such that for any fixed $N \ge N_1$ and $k \in \NN \cup \{0\}$, there holds
\begin{equation*}\label{co1_r}
|C_k| < \frac{\varepsilon}{(1000)^n} |B_{r}|.
\end{equation*}
\end{lemma}

\begin{proof}
The proof is similar to that for \cite[Lemma 4.1]{BO}.
\end{proof}

\begin{lemma}\label{co2}
There exist $N_2 = N_2(\La_0,\La_1,\plog, \qlog, n, \sigma) >1$, $\gamma_0 = \gamma_0(\La_0,\La_1,\plog, \qlog, n, \sigma,\varepsilon) \in \gh{ 0,1/4 }$ and $\delta_0 = \delta_0(\La_0,\La_1,\plog, \qlog, n, \sigma,\varepsilon) \in \gh{ 0,1/4 }$ such that for any fixed $N \ge N_2$, $k \in \NN \cup \{0\}$, $y_0 \in C_k$, and $r_0 \le \frac{\sigma r}{1000}$, if
\begin{equation}\label{co2_r}
|C_k \cap B_{r_0}(y_0)| \ge \varepsilon |B_{r_0}(y_0)|,
\end{equation}
then $B_{r_0}(y_0) \cap \OO_{r} \subset D_k$.
\end{lemma}

\begin{proof}
The proof is similar to that for \cite[Lemma 4.2]{BO}. For the sake of convenience, we outline the proof.

{\em Step 1.} For simplicity, we write $\lambda_k = N^k \lambda_0 > 1$, where $N \ge N_2 >1$. If not, suppose that $B_{r_0}(y_0) \cap \OO_{r} \not\subset D_k$. Then there is a point $y_1 \in B_{r_0}(y_0) \cap \OO_{r}$ such that $y_1 \notin D_k$, that is,
\begin{equation*}\label{co2-1}
\mint{B_{r}(y_1)}{|\nabla u|^{(p(x)-\delta)\frac{q(x)}{q_{\OO_{4r}}^-}}\chi_{\OO_{2r}}}{dx} \le \lambda_k,
\end{equation*}
and
\begin{equation*}\label{co2-2}
\gh{ \mint{B_{r}(y_1) }{\bgh{|\bff|^{(p(x)-\delta)\frac{q(x)}{q_{\OO_{4r}}^-}}+1}^{1+\tilde{\sigma}}\chi_{\OO_{2r}}}{dx} }^\frac{1}{1+\tilde{\sigma}} \le \gamma\lambda_k
\end{equation*}
for all $r > 0$.

{\em Step 2.} We divide the proof into two cases: the interior $(B_{9r_0}(y_1) \subset \OO)$ and boundary $(B_{9r_0}(y_1) \not\subset \OO)$ case. We only prove the boundary case since the interior one can be proved in a similar way.

{\em Step 3.} Let $B_{9r_0}(y_1) \not\subset \OO$. We can find a boundary point $\tilde{y}_1 \in \partial\OO \cap B_{9r_0}(y_1)$. Then it satisfies that
\begin{equation}\label{co2-3}
\OO_{2r_0}(y_0) \subset \OO_{3r_0}(y_1) \subset \OO_{20r_0}(\tilde{y}_1) \quad \text{and} \quad \OO_{\frac{48r_0}{\sigma}}(\tilde{y}_1) \subset \OO_{\frac{57r_0}{\sigma}}(y_1) \subset \OO_{r}(y_0) \subset \OO_{2r}.
\end{equation}
Then we have
\begin{equation*}\label{co2-6}
\mint{\OO_{\frac{48r_0}{\sigma}}(\tilde{y}_1)}{|\nabla u|^{p(x)-\delta}}{dx} \le c_1 \lambda_k^{\frac{q_{\OO_{4r}}^-}{q_{\OO_{\frac{48r_0}{\sigma}}(\tilde{y}_1)}^+}} \quad \text{and} \quad \mint{\OO_{\frac{48r_0}{\sigma}}(\tilde{y}_1)}{|\bff|^{p(x)-\delta}}{dx} \le c_1 \delta^\frac{q^-}{q^+} \lambda_k^{\frac{q_{\OO_{4r}}^-}{q_{\OO_{\frac{48r_0}{\sigma}}(\tilde{y}_1)}^+}}
\end{equation*}
for some constant $c_1 = c_1(\La_0,\La_1,\plog, \qlog, n, \sigma) > 0$.

{\em Step 4.} Applying Theorem \ref{boundary_difference}, Theorem \ref{difference_w_v_boundary}, and Theorem \ref{l_infinity_boundary}, there exist the constants $\gamma_0$ and $\delta_0$, both depending only on $\La_0$, $\La_1$, $\plog$, $\qlog$, $n$, $\sigma$, $\eta$, such that 
\begin{gather*}
\mint{\OO_{\frac{384r_0}{\sigma}}(\tilde{y}_1)}{|\nabla u- \nabla w|^{p(x)-\delta}}{dx} \le c_1 \eta \lambda_k^{\frac{q_{\OO_{4r}}^-}{q_{\OO_{\frac{48r_0}{\sigma}}(\tilde{y}_1)}^+}},\\
\mint{\OO_{20r_0}(\tilde{y}_1)}{|\nabla w-\nabla \bar{V}|^{p_{\OO_{\frac{384r_0}{\sigma}}}^+ -\delta}}{dx} \le c_1 \eta \lambda_k^{\frac{q_{\OO_{4r}}^-}{q_{\OO_{\frac{48r_0}{\sigma}}(\tilde{y}_1)}^+}},\\
\Norm{\nabla \bar{V}}_{L^\infty \gh{\OO_{20r_0}(\tilde{y}_1)}} \lesssim c_1 \lambda_k^{\frac{q_{\OO_{4r}}^-}{q_{\OO_{\frac{48r_0}{\sigma}}(\tilde{y}_1)}^+} \frac{1}{p_{\OO_{\frac{384r_0}{\sigma}}}^+}}. \label{co2-11}
\end{gather*}

{\em Step 5.} We next obtain
\begin{equation}\label{co2-12}
\mint{\OO_{20r_0}(\tilde{y}_1)}{|\nabla u-\nabla \bar{V}|^{(p(x)-\delta)\frac{q(x)}{q_{\OO_{4r}}^-}}}{dx} \le c_2 \eta^\frac{1}{4} \lambda_k,
\end{equation}
and
\begin{equation*}\label{co2-13}
\Norm{|\nabla \bar{V}|^{(p(\cdot)-\delta)\frac{q(\cdot)}{q_{\OO_{4r}}^-}}}_{L^\infty \gh{\OO_{20r_0}(\tilde{y}_1)}} \le c_2 \lambda_k
\end{equation*}
for some constant $c_2 = c_2(\La_0,\La_1,\plog, \qlog, n, \sigma) >0$.

{\em Step 6.} We now have
\begin{equation}\label{co2-14}
C_k \cap \OO_{r_0}(y_0) \subset \mgh{x \in \OO_{r_0}(y_0) : \M \gh{|\nabla u-\nabla \bar{V}|^{(p(\cdot)-\delta)\frac{q(\cdot)}{q_{\OO_{4r}}^-}} \chi_{\OO_{2r_0}(y_0)}}(x) > \lambda_k}
\end{equation}
provided that $N \ge N_2 \ge \max \mgh{2^{\frac{p^+ q^+}{q_-} -1}(1+c_2), 3^n}$.

{\em Step 7.} We finally conclude, using (\ref{weak 1-1}), (\ref{co2-14}), (\ref{co2-3}), and (\ref{co2-12}), that
\begin{equation*}
\begin{aligned}
|C_k \cap B_{r_0}(y_0)| &\apprle \frac{1}{\lambda_k} \integral{\OO_{2r_0}(y_0)}{|\nabla u-\nabla \bar{V}|^{(p(x)-\delta)\frac{q(x)}{q_{\OO_{4r}}^-}}}{dx} \apprle c_2 \eta^\frac{1}{4} |B_{r_0}(y_0)| < \varepsilon |B_{r_0}(y_0)|,\\
\end{aligned}
\end{equation*}
by taking $\eta$ sufficiently small. As a consequence $\gamma_0$ and $\delta_0$ are also determined, which is a contradiction to (\ref{co2_r}).
\end{proof}

Applying Lemma \ref{calderon_zygmund}, we finally obtain the following power decay estimate:
\begin{corollary}\label{co3}
Let $N = \max\{N_1, N_2\} >1$, where $N_1$ and $N_2$ are given in Lemma \ref{co1} and Lemma \ref{co2}, respectively. Then there exist $\gamma_0 \in (0,1/4)$ and $\delta_0 \in (0,1/4)$, both depending only on $\La_0$, $\La_1$, $\plog$, $\qlog$, $n$, $\sigma$, $\varepsilon$, such that 
\begin{equation*}\label{co3-1}
|C_k| \le \varepsilon \gh{\frac{10}{\sigma}}^n |D_k| \quad \text{for} \ \ k \in \NN \cup \{0\}.
\end{equation*}
Moreover, by iteration, we obtain
\begin{equation*}\label{co3-2}
\begin{aligned}
&\left| \mgh{x \in \OO_{r} : \M^*(\nabla u)(x) > N^k\lambda_0} \right|\\
&\quad \le \varepsilon_1^k \left| \mgh{x \in \OO_{r} : \M^*(\nabla u)(x) > \lambda_0} \right| + \sum_{i=1}^k \varepsilon_1^i \left| \mgh{x \in \OO_{r} : \M^*_{1+\tilde{\sigma}}(\bff)(x) > \gamma N^{k-i}\lambda_0} \right|,
\end{aligned}
\end{equation*}
where $\varepsilon_1 : = \varepsilon \gh{\frac{10}{\sigma}}^n$.
\end{corollary}


\section{Calder\'{o}n-Zygmund type estimates}\label{CZ_estimates}

We are ready to prove our main theorems. In this section, we omit $x_0$ in the $\OO_{r}(x_0)$.
\begin{proof}[Proof of Theorem \ref{main_theorem1}]
\label{local estimate}
We first recall Section \ref{main_constants}. Note that $\max\mgh{M^u,M^v,M^v} \lesssim M_0$. Fix any $x_0 \in \OO$ and $r \in \left(0, \frac{1}{C_0 M_0} \right]$ with 
\begin{equation}\label{le-0.5}
\frac{1}{C_0(n,\Lambda_0,\Lambda_1,\plog,\qlog, S_0)} := \min \mgh{\frac{1}{4}, \frac{S_0}{2}, \frac{R_8}{2}, \frac{R_{10}}{2}, \frac{\rho^{-1}(d_1)}{2}, \frac{\omega^{-1}(d_1)}{2}},
\end{equation}
where
\begin{gather*}
d_1 := \min \mgh{\sqrt{\frac{n+1}{n}}-1, \frac{\Lambda_0}{2\Lambda_1}, \frac{\sigma_0}{4}, \frac{1}{2n}, \frac14},\quad d_2 := \min \mgh{\frac{q^- \sigma_0}{8}, \frac{q^- \tilde{\sigma}}{2}, \frac{(q^-)^2}{4q^+}, \frac{\sigma_0}{2}, \frac14},\\
\rho^{-1}(t) := \sup\{r \in (0,1) : \rho(r) \le t\}, \quad \text{and} \quad \omega^{-1}(t) := \sup\{r \in (0,1) : \omega(r) \le t\},
\end{gather*}
for $t>0$. Note that $\rho^{-1}$ and $\omega^{-1}$ is well defined by the definition of $\rho$ and $\omega$, respectively. Thus we can apply all results in Section \ref{covering_arguments}.

Let $\de_0$ be such that Lemma \ref{co2} and Corollary \ref{co3} holds and let $\de \in (0,\de_0)$ be given. Set
\begin{equation*}\label{le-1}
S := \sum_{k=1}^{\infty} N^{k q_{\OO_{4r}}^-} \left| \mgh{x \in \OO_{r} : \M^*(\nabla u)(x) > N^k\lambda_0} \right|.
\end{equation*}
By Corollary \ref{co3}, Lemma \ref{distribution}, Fubini's theorem and Lemma \ref{max ftn est1}, we deduce
\begin{equation*}\label{le-4}
S \le \sum_{k=1}^{\infty} \gh{N^{q^+} \varepsilon_1}^k \mgh{2|\OO_{r}| +  \frac{C}{\gh{\gamma \lambda_0}^{q_{\OO_{4r}}^-}} \integral{\OO_{2r}}{\bgh{|\bff|^{(p(x)-\delta)q(x)}+1}}{dx}}.
\end{equation*}
Now we select $\varepsilon = \varepsilon(\La_0,\La_1,\plog, \qlog, n, \sigma) >0$ such that $N^{q^+} \varepsilon \gh{\frac{10}{\sigma}}^n = N^{q^+} \varepsilon_1 = \frac{1}{2}$ and a corresponding $\gamma_0$ and $\delta_0$, both depending only on $\La_0$, $\La_1$, $\plog$, $\qlog$, $n$, $\sigma$. Then we see
\begin{equation}\label{le-5}
S \le 2|\OO_{r}| +  \frac{C}{\gh{\lambda_0}^{q_{\OO_{4r}}^-}} \integral{\OO_{2r}}{\bgh{|\bff|^{(p(x)-\delta)q(x)}+1}}{dx}.
\end{equation}

According to Lemma \ref{distribution}, (\ref{le-5}), (\ref{cc4}), and H\"{o}lder's inequality, we obtain
\begin{equation}\label{le-6}
\begin{aligned}
\mint{\OO_{r}}{|\nabla u|^{(p(x)-\delta)q(x)}}{dx} &\le \mint{\OO_{r}}{\M^*(\nabla u)^{q_{\OO_{4r}}^-}}{dx} \le C \lambda_0^{q_{\OO_{4r}}^-} \gh{1 + \frac{S}{|\OO_{r}|}}\\
&\le C \mgh{\lambda_0^{q_{\OO_{4r}}^-} + \mint{\OO_{2r}}{\bgh{|\bff|^{(p(x)-\delta)q(x)}+1}}{dx}}\\
&\le C \gh{\mint{\OO_{4r}}{|\nabla u|^{p(x)-\delta}}{dx}}^{q_{\OO_{4r}}^-} + C \gh{ \mint{\OO_{4r}}{|\bff|^{(p(x)-\delta)(1+\tilde{\sigma})}}{dx}}^{\frac{q_{\OO_{4r}}^-}{1+\tilde{\sigma}}}\\
&\quad + C \mint{\OO_{4r}}{|\bff|^{(p(x)-\delta)q(x)}}{dx} + C\\
&\le C \mgh{\gh{\mint{\OO_{4r}}{|\nabla u|^{p(x)-\delta}}{dx}}^{q_{\OO_{4r}}^-} + \mint{\OO_{4r}}{|\bff|^{(p(x)-\delta)q(x)}}{dx} +1},
\end{aligned}
\end{equation}
for some constant $C=C(\La_0,\La_1,\plog, \qlog, n, \sigma)>0$. Here the last inequality above have used, in Definition \ref{final_radius_restriction}, the fact that $1+\tilde{\sigma} \le q^- \le q_{\OO_{4r}}^-$.
\end{proof}

We extend the local estimate (\ref{le-6}) up to the boundary.
\begin{proof}[Proof of Theorem \ref{main_theorem2}]
We first choose $r = \frac{1}{C_0 M_0}$, where $C_0$ and $M_0$ are given in (\ref{le-0.5}). From the standard covering argument, we can find finitely many disjoint open balls $\mgh{B_{\frac{r}{3}}(y_k)}_{k=1}^m$, $y_k \in \OO$, such that $\bar{\OO} \subset \bigcup_{k=1}^m B_{r}(y_k)$. Note that for an integrable function $f$ there exists a constant $C = C(n)>0$ so that $\sum_{k=1}^m \integral{\OO_{4r}(y_k)}{f}{dx} \le C \integral{\OO}{f}{dx}.$ Let $\de_0$ be as in Theomem \ref{main_theorem1} and let $\de \in (0,\de_0)$. 

Then it follows from (\ref{le-6}) that
\begin{equation}\label{ge-1}
\begin{aligned}
&\integral{\OO}{|\nabla u|^{(p(x)-\delta)q(x)}}{dx} \le \sum_{k=1}^m \integral{\OO_{r}(y_k)}{|\nabla u|^{(p(x)-\delta)q(x)}}{dx}\\
&\quad \lesssim \sum_{k=1}^m \mgh{r^n \gh{\mint{\OO_{4r}(y_k)}{|\nabla u|^{p(x)-\delta}}{dx}}^{q_{\OO_{4r}(y_k)}^-} + \integral{\OO_{4r}(y_k)}{\bgh{|\bff|^{(p(x)-\delta)q(x)}+1}}{dx}}\\
&\quad \lesssim r^{n(1-q^+)} \gh{\integral{\OO}{\bgh{|\nabla u|^{p(x)-\delta}+1}}{dx}}^{q^+} + \integral{\OO}{\bgh{|\bff|^{(p(x)-\delta)q(x)}+1}}{dx}.
\end{aligned}
\end{equation}

From (\ref{size_solution}) and the definition of $r$, we obtain 
\begin{equation}\label{ge-2}
\integral{\OO}{|\nabla u|^{(p(x)-\delta)q(x)}}{dx}  \lesssim \gh{\integral{\OO}{|\bff|^{p(x)}}{dx}}^{n(q^+ -1) + q^+}  + \integral{\OO}{|\bff|^{(p(x)-\delta)q(x)}}{dx} + |\OO| + 1.
\end{equation}

Let $M^+$ and $M^-$ be any two  constants such that additionally we have $1 < M^- \le q^- \leq q(\cdot) \le q^+\leq M^+<\infty$. Following the proof of 
Theorem \ref{main_theorem1}, we see that ${\de_0}$ can be chosen to depend on $M^+$ instead of $q^{\pm}$. This, in particular, implies that we can choose $\de_0$ independent of $M^-$. 


Let us now define $r(x) :=  \frac{p(x) - \de}{p(x)} q(x)$ for $\de \in (0,\de_0)$ (it is important to note that we cannot take $\de=0$), then we trivially have
$$r^- \ge \lbr \frac{p(x) - \de}{p(x)}\rbr^- M^- \quad \text{and} \quad r^+ \le \lbr \frac{p(x) - \de}{p(x)}\rbr^+ M^+.$$
Note that $r(\cdot)$ is clearly log-H\"{o}lder continuous with the log-H\"older constants equivalent to the ones satisfied by $\qq$.

Since all the estimates above are independent of $M^-$ and $\de_0$ is is independent of $M^-$, we can choose $M^-$ small such that $\lbr \frac{p(x) - \de_0}{p(x)}\rbr^- M^- \leq 1$. This in particular allows $r^-=1$. 

For this choice of the exponent $r(\cdot)$, we conclude from (\ref{ge-2}) that
\begin{equation*}\label{ge-3}
\begin{aligned}
\integral{\OO}{|\nabla u|^{p(x)r(x)}}{dx} &\le C \mgh{\gh{\integral{\OO}{|\bff|^{p(x)r(x)}}{dx}}^{n(q^+ -1) + q^+} + 1}\\
&\le C \mgh{\gh{\integral{\OO}{|\bff|^{p(x){r}(x)}}{dx}}^{n(M^+ -1) + M^+} +1}
\end{aligned}
\end{equation*}
for some constant $C=C(\La_0,\La_1,\plog, r_{\log}, M^+, n, \sigma,\OO,S_0)>0$, which completes the proof.
\end{proof}


\appendix

\section{Proof of Theorem \ref{measure_density_poincare}}
\label{proof_poincare}

\begin{proof}
 \renewcommand{\pu}{p^+_B}
 \renewcommand{\pd}{p^-_B}
 
 Let $B = B_r$ be a ball of radius $2r<R_4\leq 1$ and $\phi \in W^{1,\sss}(2B)$ be a function satisfying all the hypothesis. We can now apply Lemma \ref{scaled_poincare} to obtain the estimate
 \begin{equation*}
 \label{eq5.20}
  \fint_B |\phi - \avg{\phi}{B}|^{s(x)} \ dx \leq_{\slog,n} \frac{1}{|B|}\int_B |\nabla \phi|^{s(x)} \ dx + 1.
 \end{equation*}

%
 Using the trivial bound $|N(\phi)| \leq |B|$, we get
 \begin{equation}
 \label{5.20}
  \fint_B |\phi - \avg{\phi}{B}|^{s(x)} \ dx \leq_{\slog,n} \frac{1}{|N(\phi)|} \int_B |\nabla \phi|^{s(x)} \ dx +1 .
 \end{equation}

 If $|\avg{\phi}{B}| =0$, then there is nothing to prove.  Hence we shall now proceed assuming $|\avg{\phi}{B}| >0$. 
%
%
 By a triangle inequality, we get
 \begin{equation}
  \label{eq5.22}
  \fint_B |\phi|^{s(x)}\ dx \apprle  \fint_B |\phi - \avg{\phi}{B}|^{s(x)}\ dx + \fint_B |\avg{\phi}{B}|^{s(x)} \ dx.
 \end{equation}
The first term in the right of \eqref{eq5.22} can be controlled using \eqref{5.20}. We shall now estimate the second term on the right on \eqref{eq5.22} as follows:

Set $\upsilon(x): = |\phi(x) - \avg{\phi}{B}|$ and consider the cut-off function $\eta \in C_c^{\infty}(2B)$ such that $\eta \equiv 1$ on $\overline{B}$ and $|\nabla \eta| \leq \frac{c}{\diam(B)}$.  Define $\varpi(x) :=  \frac{\upsilon(x) \eta(x)}{|\avg{\phi}{B}|}$, then we see that $\spt(\varpi) \subset 2B$. Thus from the definition of $N(\phi)$, we see that 
\begin{equation*}
\label{property_test_function}
\varpi(x)^{s(x)}  \geq \chi_{N(\phi)}(x),\quad  \text{and} \quad  \varpi(x)^{s(x)} = 1 \ \text{for} \  x \in N(\phi).  
\end{equation*}

\noindent {\bf Claim:} The following bound holds:
\begin{equation}\label{claim_1}\int_{2B} |\nabla (\upsilon \eta)|^{s(x)} \ dx \apprle \int_{2B} |\nabla \phi|^{s(x)} + |B|.\end{equation}

To prove this, we estimate as follows:
\begin{equation*}
 \begin{array}{@{}r@{}c@{}l@{}@{}}
  \int_{2B} |\nabla (\upsilon \eta)|^{s(x)} \ dx & \apprle &\int_{2B} \lbr \frac{|\phi - \avg{\phi}{B}|}{\diam(B)}\rbr^{s(x)} \ dx + \int_{2B} |\nabla \phi|^{s(x)} \ dx \\
  & \overset{\text{Lemma\ \ref{scaled_poincare}}}{\apprle} &\int_{2B}|\nabla \phi|^{s(x)} \ dx + |B| + \int_{2B} |\nabla \phi|^{s(x)} \ dx +\\
  & &\qquad + \diam(B)^{-s^+_{2B}} |B| \max \{ |\avg{\phi}{2B} - \avg{\phi}{B}|^{s^-_{2B}}, |\avg{\phi}{2B} - \avg{\phi}{B}|^{s^+_{2B}}\}. \\
 \end{array}
\end{equation*}

Observe that 
\begin{equation*}
 \begin{array}{ll}
  |\avg{\phi}{2B}- \avg{\phi}{B}| & \apprle \fint_{2B} |\phi(x) - \avg{\phi}{2B}| \ dx  \apprle \diam(B) \fint_{2B} |\nabla \phi| \ dx \\
  & \apprle \diam(B) \lbr \fint_{2B} |\nabla \phi|^{s^-_{2B}} \ dx \rbr^{\frac1{s^-_{2B}}}.
 \end{array}
\end{equation*}
Using the $\log$-H\"older continuity of $\sss$ along with the restriction $r \leq \frac{1}{M_4}$ with $M_4$ as in \eqref{assumption_on_phi_capacity}, we get
\begin{equation*}
 \begin{array}{l}
  \diam(B)^{-s^+_{2B}} |B| \max \{ |\avg{\phi}{2B} - \avg{\phi}{B}|^{s^-_{2B}}, |\avg{\phi}{2B} - \avg{\phi}{B}|^{s^+_{2B}}\} \\
  \qquad\qquad  \apprle  \max \left\{ \diam(B)^{s^-_{2B}-s^+_{2B}} \int_{2B} |\nabla \phi|^{s^-_{2B}} \ dx, |B|^{1-\frac{s^+_{2B}}{s^-_{2B}}} \lbr \int_{2B} |\nabla \phi|^{s^-_{2B}} \ dx\rbr^{\frac{s^+_{2B}}{s^-_{2B}}}\right\} \\
  \qquad\qquad  \apprle  \max \left\{ \diam(B)^{s^-_{2B}-s^+_{2B}} , |B|^{1-\frac{s^+_{2B}}{s^-_{2B}}}  \right\}  \int_{2B} |\nabla \phi|^{s^-_{2B}} \ dx 
  \apprle \int_{2B} |\nabla \phi|^{s(x)} \ dx + |B|. 
 \end{array}
\end{equation*}
This proves the Claim of estimate \eqref{claim_1}.

We can thus make use of Lemma \ref{measure_density_quasiconvex} along with \cite[Theorem 8.2.4]{Diening} and \eqref{norm_integral} to get
\begin{equation*}
 \label{5.25}
 |N(\phi)| \leq \int_{2B}\varpi(x)^{s(x)} \ dx \apprle \max \left\{ \lbr \int_{2B}|\nabla \varpi(x)|^{s(x)} \ dx \rbr^{\frac{s^+_B}{s^-_B}}, \lbr \int_{2B}|\nabla \varpi(x)|^{s(x)} \ dx \rbr^{\frac{s^-_B}{s^+_B}} \right\}.
\end{equation*}

Depending on the cases if  $\int_{2B} |\nabla \varpi|^{s(x)} \ dx \geq 1$ or $\int_{2B} |\nabla \varpi|^{s(x)} \ dx < 1$ or  $\left|\int_B \phi(x) \ dx\right| \geq 1$ or $\left|\int_B \phi(x) \ dx\right| \leq 1$, we can use the $\log$-H\"older continuity of $\sss$ and \eqref{claim_1} to obtain
\begin{equation*}
         \fint_B |\avg{\phi}{B}|^{s(x)} \ dx  \apprle  \frac{1}{|N(\phi)|}  \int_{2B} |\nabla \phi|^{s(x)}+1 \ dx .
\end{equation*}

This completes the proof of the Theorem. 
\end{proof}

\section*{Acknowledgement}

The first author would like to thank Peter H\"ast\"o and Jose Maria Martell for clarifying certain aspects of variable exponent spaces and Muckenhoupt weights.


\begin{thebibliography}{10}

\bibitem{AM}
Emilio Acerbi and Giuseppe Mingione.
\newblock Gradient estimates for the {$p (x)$}-{L}aplacean system.
\newblock {\em Journal f{\"u}r die reine und angewandte Mathematik},
  2005(584):117--148, 2005.

\bibitem{AM2}
Emilio Acerbi and Giuseppe Mingione.
\newblock Gradient estimates for a class of parabolic systems.
\newblock {\em Duke Mathematical Journal}, 136(2):285--320, 2007.

\bibitem{AH}
David~R Adams and Lars~I Hedberg.
\newblock {\em Function spaces and potential theory}, volume 314.
\newblock Springer Science \& Business Media, 2012.

\bibitem{AP1}
Karthik Adimurthi and Nguyen~Cong Phuc.
\newblock Global {L}orentz and {L}orentz--{M}orrey estimates below the natural
  exponent for quasilinear equations.
\newblock {\em Calculus of Variations and Partial Differential Equations},
  54(3):3107--3139, 2015.

\bibitem{AP2}
Karthik Adimurthi and Nguyen~Cong Phuc.
\newblock An end-point global gradient weighted estimate for quasilinear
  equations in non-smooth domains.
\newblock {\em Manuscripta Mathematica}, 150(1-2):111--135, 2016.

\bibitem{AS}
Stanislav~Nikolaevich Antontsev and Sergei~I. Shmarev.
\newblock A model porous medium equation with variable exponent of
  nonlinearity: existence, uniqueness and localization properties of solutions.
\newblock {\em Nonlinear Analysis: Theory, Methods \& Applications},
  60(3):515--545, 2005.

\bibitem{BaB}
Paolo Baroni and Verena B{\"o}gelein.
\newblock Calder{\'o}n--{Z}ygmund estimates for parabolic {$ p (x, t)
  $}-{L}aplacian systems.
\newblock {\em Revista matem{\'a}tica iberoamericana}, 30(4):1355--1386, 2014.

\bibitem{VB}
Verena B{\"o}gelein.
\newblock Global {C}alder{\'o}n--{Z}ygmund theory for nonlinear parabolic
  systems.
\newblock {\em Calculus of Variations and Partial Differential Equations},
  51(3-4):555--596, 2014.

\bibitem{BHSW}
Sun-Sig Byun, Hun Kwon, Hyoungsuk So, and Lihe Wang.
\newblock Nonlinear gradient estimates for elliptic equations in quasiconvex
  domains.
\newblock {\em Calculus of Variations and Partial Differential Equations},
  54(2):1425--1453, 2015.

\bibitem{BO}
Sun-Sig Byun and Jihoon Ok.
\newblock On {$W^{ 1, q (\cdot)}$}-estimates for elliptic equations of {$p
  (x)$}-laplacian type.
\newblock {\em Journal de Math{\'e}matiques Pures et Appliqu{\'e}es},
  106(3):512--545, 2016.

\bibitem{BO1}
Sun-Sig Byun, Jihoon Ok, and Seungjin Ryu.
\newblock Global gradient estimates for elliptic equations of {$p
  (x)$}-laplacian type with {BMO} nonlinearity.
\newblock {\em Journal f{\"u}r die reine und angewandte Mathematik (Crelles
  Journal)}, 2016(715):1--38, 2016.

\bibitem{BW2}
Sun-Sig Byun and Lihe Wang.
\newblock Elliptic equations with {BMO} coefficients in reifenberg domains.
\newblock {\em Communications on Pure and Applied Mathematics},
  57(10):1283--1310, 2004.

\bibitem{CP}
Luis~A Caffarelli and Ireneo Peral.
\newblock On {$W^{ 1, p}$} estimates for elliptic equations in divergence form.
\newblock {\em Communications on Pure and Applied Mathematics}, 51(1):1--21,
  1998.

\bibitem{CZ}
Alberto~Pedro Calder{\'o}n and Antoni Zygmund.
\newblock On the existence of certain singular integrals.
\newblock {\em Acta Mathematica}, 88(1):85--139, 1952.

\bibitem{CLR}
Yunmei Chen, Stacey Levine, and Murali Rao.
\newblock Variable exponent, linear growth functionals in image restoration.
\newblock {\em SIAM journal on Applied Mathematics}, 66(4):1383--1406, 2006.

\bibitem{USW}
David Cruz-Uribe and Li-An~Daniel Wang.
\newblock Extrapolation and weighted norm inequalities in the variable lebesgue
  spaces.
\newblock {\em Transactions of the American Mathematical Society},
  369(2):1205--1235, 2017.

\bibitem{DiBenedetto}
Emmanuele DiBenedetto.
\newblock {$C^{1+\alpha}$} local regularity of weak solutions of degenerate
  elliptic equations.
\newblock {\em Nonlinear Analysis}, 7(8):827--850, 1983.

\bibitem{BM}
Emmanuele DiBenedetto and Juan~J. Manfredi.
\newblock On the higher integrability of the gradient of weak solutions of
  certain degenerate elliptic systems.
\newblock {\em American Journal of Mathematics}, 115(5):1107--1134, 1993.

\bibitem{Diening}
Lars Diening, Petteri Harjulehto, Peter H{\"a}st{\"o}, and Michael Ruzicka.
\newblock {\em Lebesgue and {S}obolev spaces with variable exponents}.
\newblock Springer, 2011.

\bibitem{DH}
Lars Diening and Peter H{\"a}st{\"o}.
\newblock Muckenhoupt weights in variable exponent spaces.
\newblock {\em Preprint}, 2008.

\bibitem{DLR}
Lars Diening, Daniel Lengeler, and M~R{\u{u}}{\v{z}}i{\v{c}}ka.
\newblock The {S}tokes and {P}oisson problem in variable exponent spaces.
\newblock {\em Complex Variables and Elliptic Equations}, 56(7-9):789--811,
  2011.

\bibitem{DMS}
Frank Duzaar, Giuseppe Mingione, and Klaus Steffen.
\newblock {\em Parabolic systems with polynomial growth and regularity}, volume
  214.
\newblock American Mathematical Society, 2011.

\bibitem{EHL}
Michela Eleuteri, Petteri Harjulehto, and Teemu Lukkari.
\newblock Global regularity and stability of solutions to elliptic equations
  with nonstandard growth.
\newblock {\em Complex Variables and Elliptic Equations}, 56(7-9):599--622,
  2011.

\bibitem{Giaquinta}
Mariano Giaquinta.
\newblock {\em Multiple Integrals in the Calculus of Variations and Nonlinear
  Elliptic Systems.(AM-105)}, volume 105.
\newblock Princeton University Press, 2016.

\bibitem{Grafakos}
Loukas Grafakos.
\newblock {\em Classical fourier analysis}, volume~2.
\newblock Springer, 2008.

\bibitem{GP}
Jos{\'e}~J Guadalupe and Mario P{\'e}rez.
\newblock Perturbation of orthogonal fourier expansions.
\newblock {\em Journal of approximation theory}, 92(2):294--307, 1998.

\bibitem{Haj}
Piotr Haj{\l}asz.
\newblock Pointwise hardy inequalities.
\newblock {\em Proceedings of the American Mathematical Society},
  127(2):417--423, 1999.

\bibitem{HK}
Juha Heinonen and Pekka Koskela.
\newblock Weighted {S}obolev and {P}oincar{\'e} inequalities and quasiregular
  mappings of polynomial type.
\newblock {\em Mathematica Scandinavica}, pages 251--271, 1995.

\bibitem{HU}
Eurica Henriques and Jos{\'e}~Miguel Urbano.
\newblock Intrinsic scaling for {PDE}'s with an exponential nonlinearity.
\newblock {\em Indiana University Mathematics Journal}, 55(5):1701--1722, 2006.

\bibitem{IW1}
Tadeusz Iwaniec.
\newblock Projections onto gradient fields and {$L^{p}$}-estimates for
  degenerated elliptic operators.
\newblock {\em Studia Mathematica}, 75(3):293--312, 1983.

\bibitem{IKM}
Tadeusz Iwaniec, Pekka Koskela, and Gaven Martin.
\newblock Mappings of {BMO}-distortion and {B}eltrami-type operators.
\newblock {\em Journal d'Analyse Math{\'e}matique}, 88(1):337--381, 2002.

\bibitem{JJW1}
Huilian Jia, Dongsheng Li, and Lihe Wang.
\newblock Global regularity for divergence form elliptic equations on
  quasiconvex domains.
\newblock {\em Journal of Differential Equations}, 249(12):3132--3147, 2010.

\bibitem{JJW2}
Huilian Jia, Dongsheng Li, and Lihe Wang.
\newblock Global regularity for divergence form elliptic equations in {O}rlicz
  spaces on quasiconvex domains.
\newblock {\em Nonlinear Analysis: Theory, Methods \& Applications},
  74(4):1336--1344, 2011.

\bibitem{JJW3}
Huilian Jia and Lihe Wang.
\newblock Divergence form parabolic equations on time-dependent quasiconvex
  domains.
\newblock {\em International Journal of Mathematics}, 23(12):1250128, 2012.

\bibitem{KZ}
Juha Kinnunen and Shulin Zhou.
\newblock A local estimate for nonlinear equations with discontinuous
  coefficients.
\newblock {\em Communications in Partial Differential Equations},
  24(11-12):2043--2068, 1999.

\bibitem{MZ}
Jan Mal{\`y} and William~P Ziemer.
\newblock {\em Fine regularity of solutions of elliptic partial differential
  equations}.
\newblock Number~51. American Mathematical Soc., 1997.

\bibitem{Mil}
Mario Milman.
\newblock Rearrangements of {BMO} functions and interpolation.
\newblock In {\em Interpolation Spaces and Allied Topics in Analysis}, pages
  208--212. Springer, 1984.

\bibitem{MM}
Masashi Misawa.
\newblock {$ L^q $} estimates of gradients for evolutional {$p$}-{L}aplacian
  systems.
\newblock {\em Journal of Differential Equations}, 219(2):390--420, 2005.

\bibitem{RR}
Kumbakonam~R Rajagopal and M~Ru{\v{z}}i{\v{c}}ka.
\newblock Mathematical modeling of electrorheological materials.
\newblock {\em Continuum mechanics and Thermodynamics}, 13(1):59--78, 2001.

\bibitem{Ru}
Michael Ruzicka.
\newblock {\em Electrorheological fluids: modeling and mathematical theory}.
\newblock Springer Science \& Business Media, 2000.

\bibitem{Se}
Stephen Semmes.
\newblock Hypersurfaces in {$\mathbb{R}^n$} whose unit normal has small {BMO}
  norm.
\newblock {\em Proceedings of the American Mathematical Society},
  112(2):403--412, 1991.

\bibitem{Tor}
Alberto Torchinsky.
\newblock {\em Real-variable methods in harmonic analysis}, volume 123.
\newblock Academic Press, 1986.

\bibitem{VVZ}
Vasilii~Vasil'evich Zhikov.
\newblock Averaging of functionals of the calculus of variations and elasticity
  theory.
\newblock {\em Izvestiya Rossiiskoi Akademii Nauk. Seriya Matematicheskaya},
  50(4):675--710, 1986.

\bibitem{Zhong}
Xiao Zhong.
\newblock {\em On nonhomogenous quasilinear elliptic equations}.
\newblock PhD thesis, Dissertation, University of Jyv\"askyl\"a, Jyv\"askyl\"a,
  1998.

\end{thebibliography}

\end{document}